\documentclass[12pt]{amsart}
\usepackage{tcolorbox}
\usepackage{amsmath}
\usepackage{amsfonts}
\usepackage{amssymb}
\usepackage[all,cmtip]{xy}           
\usepackage{bbm}
\usepackage{bbding}
\usepackage{txfonts}
\usepackage[shortlabels]{enumitem}
\usepackage{mathrsfs}
\usepackage{amscd}
\usepackage{tikz-cd}
\usepackage[active]{srcltx}
\usepackage{tikz}
\usepackage[colorlinks,final,backref=page,hyperindex]{hyperref}
\usetikzlibrary{decorations.pathreplacing}
\usetikzlibrary{calc}
\usetikzlibrary{arrows,shapes,chains}

\makeatletter

\newtheorem{thm}{Theorem}[section]
\newtheorem{prop}[thm]{Proposition}
\newtheorem{lem}[thm]{Lemma}
\newtheorem{cor}[thm]{Corollary}
\newtheorem{prop-def}{Proposition-Definition}[section]

\theoremstyle{definition}
\newtheorem{defn}[thm]{Definition}

\newtheorem{thmA}{Theorem}

\newtheorem{remark}[thm]{Remark}
\newtheorem{exam}[thm]{Example}

\newcommand{\nc}{\newcommand}

\nc{\delete}[1]{{}}
\nc{\mmargin}[1]{}

\newcommand{\Z}{{\mathbb{Z}}}
\newcommand{\V}{{\mathcal{V}}}
\newcommand{\M}{{\mathcal{M}}}
\nc{\U}{{\mathcal U}}
\nc{\D}{{\mathcal D}}
 \newcommand{\Path}{{\mathcal{P}}}

\nc{\Alg}{\mathrm{Alg}}
\nc{\rmH}{\mathrm{H}}
\nc{\rmB}{\mathrm{B}}
\nc{\rmZ}{\mathrm{Z}}
\nc{\rmC}{\mathrm{C}}
\nc{\DT}{\mathrm{DT}}
\nc{\C}{\mathbf{C}}
\nc{\sh}{\mathrm{sh}}

\nc{\mlabel}[1]{\label{#1}}  
\nc{\mcite}[1]{\cite{#1}}  
\nc{\mref}[1]{\ref{#1}}  
\nc{\mbibitem}[1]{\bibitem{#1}} 

\delete{
	\nc{\mlabel}[1]{\label{#1}  
		{\hfill \hspace{1cm}{\bf{{\ }\hfill(#1)}}}}
	\nc{\mcite}[1]{\cite{#1}{{\bf{{\ }(#1)}}}}  
	\nc{\mref}[1]{\ref{#1}{{\bf{{\ }(#1)}}}}  
	\nc{\mbibitem}[1]{\bibitem[\bf #1]{#1}} 
	
}

 \font\cyrs=wncyr7

\nc{\lan}{\langle}
\nc{\ran}{\rangle}
\nc{\vep}{\varepsilon}
\nc{\bin}[2]{ (_{\stackrel{\scs{#1}}{\scs{#2}}})}  
\nc{\binc}[2]{(\!\! \begin{array}{c} \scs{#1}\\
		\scs{#2} \end{array}\!\!)}  
\nc{\bincc}[2]{  ( {\scs{#1} \atop
		\vspace{-1cm}\scs{#2}} )}  
\nc{\oline}[1]{\overline{#1}}
\nc{\mapm}[1]{\lfloor\!|{#1}|\!\rfloor}
\nc{\bs}{\bar{S}}
\nc{\la}{\longrightarrow}
\nc{\ot}{\otimes}
\nc{\rar}{\rightarrow}
\nc{\lon }{\,\rightarrow\,}
\nc{\dar}{\downarrow}
\nc{\dap}[1]{\downarrow \rlap{$\scriptstyle{#1}$}}
\nc{\defeq}{\stackrel{\rm def}{=}}
\nc{\dis}[1]{\displaystyle{#1}}
\nc{\dotcup}{\ \displaystyle{\bigcup^\bullet}\ }
\nc{\hcm}{\ \hat{,}\ }
\nc{\hts}{\hat{\otimes}}
\nc{\hcirc}{\hat{\circ}}
\nc{\lleft}{[}
\nc{\lright}{]}
\nc{\curlyl}{\left \{ \begin{array}{c} {} \\ {} \end{array}
	\right .  \!\!\!\!\!\!\!}
\nc{\curlyr}{ \!\!\!\!\!\!\!
	\left . \begin{array}{c} {} \\ {} \end{array}
	\right \} }
\nc{\longmid}{\left | \begin{array}{c} {} \\ {} \end{array}
	\right . \!\!\!\!\!\!\!}
\nc{\ora}[1]{\stackrel{#1}{\rar}}
\nc{\ola}[1]{\stackrel{#1}{\la}}
\nc{\scs}[1]{\scriptstyle{#1}} \nc{\mrm}[1]{{\rm #1}}
\nc{\dirlim}{\displaystyle{\lim_{\longrightarrow}}\,}
\nc{\invlim}{\displaystyle{\lim_{\longleftarrow}}\,}
\nc{\dislim}[1]{\displaystyle{\lim_{#1}}} \nc{\colim}{\mrm{colim}}
\nc{\mvp}{\vspace{0.3cm}} \nc{\tk}{^{(k)}} \nc{\tp}{^\prime}
\nc{\ttp}{^{\prime\prime}} \nc{\svp}{\vspace{2cm}}
\nc{\vp}{\vspace{8cm}}
\nc{\modg}[1]{\!<\!\!{#1}\!\!>}
\nc{\intg}[1]{F_C(#1)}
\nc{\lmodg}{\!<\!\!}
\nc{\rmodg}{\!\!>\!}
\nc{\cpi}{\widehat{\Pi}}
\nc{\ssha}{{\mbox{\cyrs X}}} 
\nc{\tsha}{{\mbox{\cyrt X}}}
\nc{\shpr}{\diamond}    
\nc{\labs}{\mid\!}
\nc{\rabs}{\!\mid}

\nc{\RBO}{{\mathrm{RBO}_\lambda}}
\nc{\RBLO}{\mathrm{RBLO}_{\lambda}}
\nc{\Sh}{\mathrm{Sh}}
\nc{\RBLA}{{_\lambda\mathfrak{RBL}}}
\nc{\RBLAz}{{_0\mathfrak{RBL}}}
\nc{\frakLie}{{\mathfrak{Lie}}}
\nc{\sgn}{\mathrm{sgn}}

\nc{\Ass}{\mathrm{Ass}}
\nc{\rmLM}{\mathrm{LM}}
\nc{\rmRM}{\mathrm{RM}}
\nc{\rmBM}{\mathrm{BM}}
\nc{\rmAA}{\mathrm{AA}}
\nc{\rmLMA}{\mathrm{LMAlg}}
\nc{\rmRMA}{\mathrm{RMAlg}}
\nc{\rmBMA}{\mathrm{BMAlg}}
\nc{\rmAAA}{\mathrm{AAAlg}}

\nc{\rmrRBz}{\mathrm{rRBLA_{0}}}
\nc{\rmrRB}{\mathrm{rRBLA_{\lambda}}}
\nc{\Lie}{\mrm{Lie}}
\nc{\Liemod}{\mrm{LieRep}}
\nc{\LieAct}{\mrm{LieAct}}
\nc{\rmLR}{\mathrm{LRLie}}
\nc{\rmLAct}{\mathrm{LALie}}
\nc{\rmrRBOz}{\mathrm{rRBLO_{0}}}
\nc{\rmrRBO}{\mathrm{rRBLO_{\lambda}}}
\nc{\rmRBLA}{\mathrm{RBL_{\lambda}}}

\nc{\rDAz}{\mathrm{rDA}_0}
\nc{\rDA}{\mathrm{rDA}_{\lambda}}
\nc{\rDOz}{\mathrm{rDO}_0}
\nc{\rDO}{\mathrm{rDO}_{\lambda}}

\nc{\ad}{\mrm{ad}}
\nc{\ann}{\mrm{ann}}
\nc{\Aut}{\mrm{Aut}}
\nc{\bim}{\mbox{-}\mathsf{Bimod}}
\nc{\br}{\mrm{bre}}
\nc{\can}{\mrm{can}}
\nc{\Cont}{\mrm{Cont}}
\nc{\rchar}{\mrm{char}}
\nc{\cok}{\mrm{coker}}
\nc{\de}{\mrm{dep}}
\nc{\dtf}{{R-{\rm tf}}}
\nc{\dtor}{{R-{\rm tor}}}

\nc{\Div}{{\mrm Div}}
\nc{\Diff}{\mrm{DA}}
\nc{\Diffl}{\mathsf{DA}_\lambda}
\nc{\diffo}{{\mathsf{DO}_\lambda}}
\nc{\alg}{\mathsf{Alg}}
\nc{\End}{\mrm{End}}
\nc{\Ext}{\mrm{Ext}}
\nc{\rmf}{\mrm{f}}
\nc{\Fil}{\mrm{Fil}}
\nc{\Fr}{\mrm{Fr}}
\nc{\Frob}{\mrm{Frob}}
\nc{\Gal}{\mrm{Gal}}
\nc{\GL}{\mrm{GL}}
\nc{\h}{\mrm{H}}
\nc{\Hom}{\mrm{Hom}}
\nc{\Hoch}{\mrm{Hoch}}

\nc{\hsr}{\mrm{H}}
\nc{\hpol}{\mrm{HP}}
\nc{\id}{\mrm{Id}}
\nc{\im}{\mrm{im}}
\nc{\Id}{\mrm{Id}}
\nc{\ID}{\mrm{ID}}
\nc{\Irr}{\mrm{Irr}}
\nc{\incl}{\mrm{incl}}
\nc{\length}{\mrm{length}}
\nc{\NLSW}{\mrm{NLSW}}
\nc{\mchar}{\rm char}
\nc{\mpart}{\mrm{part}}
\nc{\ql}{{\QQ_\ell}}
\nc{\qp}{{\QQ_p}}
\nc{\rank}{\mrm{rank}}
\nc{\rcot}{\mrm{cot}}
\nc{\rdef}{\mrm{def}}
\nc{\rdiv}{{\rm div}}
\nc{\rtf}{{\rm tf}}
\nc{\rtor}{{\rm tor}}
\nc{\res}{\mrm{res}}
\nc{\s}{\mrm{S}}
\nc{\SL}{\mrm{SL}}
\nc{\Spec}{\mrm{Spec}}
\nc{\tor}{\mrm{tor}}
\nc{\Tr}{\mrm{Tr}}
\nc{\tr}{\mrm{tr}}
\nc{\wt}{\mrm{wt}}
\def\ot{\otimes}
\def\antish{{\rotatebox[origin=c]{180}{\rm !}}}

\nc{\In}{\mrm{in}}
\nc{\Out}{\mrm{out}}
\nc{\prof}{\mrm{prof}}
\nc{\pth}{\mrm{Path}}

\nc{\f}{\mrm{f}}

\nc{\bfk}{{\bf k}}
\nc{\bfone}{{\bf 1}}
\nc{\bfzero}{{\bf 0}}
\nc{\bfhom}{\mathbf{Hom}}
\nc{\detail}{\marginpar{\bf More detail}
	\noindent{\bf Need more detail!}
	\svp}
\nc{\gap}{\marginpar{\bf Incomplete}\noindent{\bf Incomplete!!}
	\svp}
\nc{\FMod}{\mathbf{FMod}}
\nc{\Mon}{\mathbf{Mon}}
 \nc{\sproof}{\noindent{  \textit{Sketch of Proof:} }}
\nc{\remarks}{\noindent{\bf Remarks: }}
\nc{\Rep}{\mathbf{Rep}}
\nc{\Rings}{\mathbf{Rings}}
\nc{\Sets}{\mathbf{Sets}}
\nc{\ob}{\mathsf{Ob}}

\nc{\Ver}{\mathbf{Vert}}
\nc{\Par}{\mathbf{Parent}}
\nc{\Int}{\mathbf{Int}}
\nc{\Leaf}{\mathbf{Leaf}}
\nc{\Root}{\mathbf{Root}}
\nc{\Tip}{\mathrm{Tip}}
\nc{\Nontip}{\mathrm{Nontip}}
\nc{\Edge}{\mathbf{Edge}}

\nc{\BA}{{\mathbb A}}   \nc{\CC}{{\mathbb C}}
\nc{\DD}{{\mathbb D}}   \nc{\EE}{{\mathbb E}}
\nc{\FF}{{\mathbb F}}   \nc{\GG}{{\mathbb G}}
\nc{\LL}{{\mathbb L}}
\nc{\NN}{{\mathbb N}}   \nc{\PP}{{\mathbb P}}
\nc{\QQ}{{\mathbb Q}}   \nc{\RR}{{\mathbb R}}
\nc{\bbS}{{\mathbb S}}
\nc{\TT}{{\mathbb T}}   \nc{\VV}{{\mathbb V}}
\nc{\ZZ}{{\mathbb Z}}   \nc{\TP}{\widetilde{P}}
\nc{\m}{{\mathbbm m}}

\nc{\scra}{\mathscr A}
\nc{\scrm}{\mathscr M}
\nc{\scrp}{\mathscr P}
\nc{\scro}{\mathscr O}
\nc{\scrq}{\mathscr Q}

\nc{\ac}{\mathrm{ac}}

\nc{\cala}{{\mathcal A}}    \nc{\calC}{{\mathcal C}} 
\nc{\cald}{\mathcal{D}}     \nc{\cale}{{\mathcal E}} \nc{\calE}{{\mathcal E}}
\nc{\calf}{{\mathcal F}}    \nc{\calg}{{\mathcal G}}
\nc{\calh}{{\mathcal H}}    \nc{\cali}{{\mathcal I}} \nc{\calI}{{\mathcal I}}
\nc{\call}{{\mathcal L}}    \nc{\calm}{{\mathcal M}}
\nc{\caln}{{\mathcal N}}    \nc{\calO}{{\mathcal O}}
\nc{\calp}{{\mathcal P}} \nc{\calP}{{\mathcal P}}   \nc{\calr}{{\mathcal R}}
\nc{\calS}{{\mathcal S}}    \nc{\calt}{{\mathcal T}}
\nc{\calv}{{\mathcal V}}    \nc{\calw}{{\mathcal W}}
\nc{\calx}{{\mathcal X}}

\nc{\fraka}{{\mathfrak a}}
\nc{\frakb}{\mathfrak{b}}
\nc{\frakC}{\mathfrak{C}}
\nc{\frakg}{{\frak g}}
\nc{\frakh}{{\frak h}}
\nc{\frakl}{{\frak l}}
\nc{\frakL}{{\frak L}}
\nc{\fraks}{{\frak s}}
\nc{\frakB}{{\frak B}}
\nc{\frakm}{{\frak m}}
\nc{\frakM}{{\frak M}}
\nc{\frako}{{\frak O}}
\nc{\frakp}{{\frak p}}
\nc{\frakW}{{\frak W}}
\nc{\frakX}{{\frak X}}
\nc{\frakS}{{\frak S}}
\nc{\frakST}{{{\frak S}{\frak T}}}
\nc{\frakt}{{\mathfrak{T}}}\nc{\frakT}{{\mathfrak{T}}}
\nc{\frakLT}{{\frakL}{\frakt}}
\nc{\frakTp}{{\frakt^+}}
\nc{\frakA}{{\frak A}}
\nc{\frakx}{{\frakx}}
\nc{\red}{\color{red}}
\nc{\RB}{{{}_\lambda\mathfrak{RBA}_{\bbS}}}
\nc{\nsRB}{{{}_\lambda\mathfrak{RBA}}}
\nc{\Dif}{{}_\lambda\!\mathfrak{Dif}  }
\nc{\Difz}{{}_0\!\mathfrak{Dif}  }
\nc{\LM}{\mathfrak{LM}}
\nc{\RM}{\mathfrak{RM}}
\nc{\BM}{\mathfrak{BM}}
\nc{\AsAct}{\mathfrak{AA}}
\nc{\reDif}{_{\lambda}\mathfrak{rDif}}
\nc{\reDifz}{_{0}\mathfrak{rDif}}
\nc{\reRBLA}{_{\lambda}\mathfrak{rRBL}}
\nc{\reRBLAz}{_{0}\mathfrak{rRBL}}

\nc{\frakCT}{\mathfrak{CT}}
\nc{\frakCTP}{\mathfrak{CT}^{+}}

\nc{\lir}[1]{\textcolor{red}{\underline{Li:}#1 }}

\begin{document}

\title[Homotopy Rota-Baxter Lie algebras]{Deformations and homotopy theory of Rota-Baxter Lie algebras}

\author{Jun Chen, Kai Wang,  and Guodong Zhou}

\address{Jun Chen, School of Mathematics,
   Nanjing University,
   Nanjing 210093,
   P.R.China}
   
\email{mathcj@nju.edu.cn}

\address{Kai Wang,  School of Mathematical Sciences\\ University of Science and Technology of China\\ Hefei, Anhui Provience 230026, China}

   \email{wangkai17@ustc.edu.cn }

\address{Guodong Zhou,
  School of Mathematical Sciences, Key Laboratory of MEA (Ministry of Education), Shanghai Key Laboratory of PMMP,
  East China Normal University,
 Shanghai 200241,
   China}

\email{gdzhou@math.ecnu.edu.cn}

\date{\today}

\begin{abstract}

For     Rota-Baxter Lie algebras, a  homotopy cooperad is exhibited,  whose cobar construction is shown to be the minimal model of the operad of Rota-Baxter Lie algebras by using algebraic Morse theory.  The deformation complex of Rota-Baxter Lie algebras  as well as the $L_\infty$-algebra structure on this complex are deduced from the minimal model and the notion of homotopy Rota-Baxter Lie algebras is given as a consequence.  

\end{abstract}

\subjclass[2010]{
18M70   
 17B38  
16E40   
16S80   
16S70   
}

\keywords{cohomology,     homotopy cooperad,  homotopy Rota-Baxter Lie algebra, Kosuzl dual homotopy cooperad,  $L_\infty$-algebra,  minimal model, operad,  Rota-Baxter Lie algebra }

\maketitle

 \tableofcontents

\allowdisplaybreaks

\section*{Introduction}
A general philosophy of deformation theory of mathematical structures, going back to Gerstenhaber \cite{Ge63} and developed by Nijenhuis--Richardson, Deligne, Goldman--Millson and others, claims that the deformation theory of any given mathematical object is controlled by a certain differential graded (=dg) Lie algebra, or more generally an $L_\infty$-algebra, associated to the object; the underlying cochain complex is called the deformation complex. In characteristic zero this philosophy has been promoted to a theorem by Lurie \cite{Lurie} and Pridham \cite{Pridham} in the language of $\infty$-categories. It remains, however, an important and in general difficult problem to construct explicitly the $L_\infty$-algebra governing the deformation theory of a given algebraic structure.

A closely related problem is to find the correct homotopy version of a given algebraic structure, in the same way as $A_\infty$-algebras are the homotopy version of associative algebras and $L_\infty$-algebras \cite{Sta63, LS93, LM95} the homotopy version of Lie algebras. The optimal answer is a minimal model of the operad governing the structure. When the operad is Koszul, Koszul duality theory \cite{GK94, LV} produces such a minimal model via the cobar construction of the Koszul dual cooperad. When the operad is not Koszul, essential difficulties arise and only a few examples are known: G\'{a}lvez-Carrillo, Tonks and Vallette \cite{GCTV12} obtained a cofibrant resolution of the Batalin--Vilkovisky operad via inhomogeneous Koszul duality, and Drummond-Cole and Vallette \cite{DCV13} then extracted a minimal model from it by algebraic Morse theory; Dotsenko and Khoroshkin \cite{DK10, DK13} constructed resolutions for operads with a Gr\"{o}bner basis. These two problems are intimately connected: given a cofibrant resolution, in particular a minimal model, of the operad in question, the convolution construction produces an $L_\infty$-algebra whose Maurer--Cartan elements are exactly the algebraic structures under consideration \cite{KS, MV1, MV2}.

The algebraic structures studied in this paper are Rota-Baxter Lie algebras of arbitrary weight. Rota-Baxter algebras originated in the work of Baxter \cite{Bax} on probability theory and were further investigated by Rota \cite{Rota69, Rota95}, Cartier \cite{Cartier} and Atkinson \cite{Atkinson}, among others; we refer to \cite{Guo12} for the basic theory. They have by now found numerous connections with combinatorics, renormalization in quantum field theory \cite{CK}, multiple zeta values, Hopf algebras and the Yang--Baxter equation. Their Lie analogues, Rota-Baxter Lie algebras, appeared only recently and have already shown rich structures: Tang, Bai, Guo and Sheng \cite{TBGS19} developed the deformation and cohomology theory of $\mathcal{O}$-operators on Lie algebras; Guo, Lang and Sheng \cite{GLS21} introduced Rota-Baxter Lie groups and studied integration and geometrization; Lazarev, Sheng and Tang \cite{LST} established the deformation theory of relative Rota-Baxter Lie algebras of weight zero, determined the controlling $L_\infty$-algebra and introduced homotopy relative Rota-Baxter Lie algebras, which were further related to triangular $L_\infty$-bialgebras in \cite{LST2}; cohomologies of relative Rota-Baxter operators of weight one on Lie groups and Lie algebras were studied by Jiang, Sheng and Zhu \cite{JSZ21, JiangSheng21}. On the associative side, Das \cite{Das20, Das21} developed the cohomology and deformation theory of (weighted) Rota-Baxter operators, and in our previous work \cite{WZ24} the last two named authors constructed the minimal model of the operad of Rota-Baxter associative algebras of arbitrary weight and deduced from it the deformation complex, its $L_\infty$-structure and the notion of homotopy Rota-Baxter associative algebras. However, all these works deform the Rota-Baxter operator with the underlying bracket fixed, or concern weight zero. The goal of the present paper is to study \emph{simultaneous} deformations of the Lie bracket and the Rota-Baxter operator of \emph{arbitrary} weight, and to provide the homotopy theory of Rota-Baxter Lie algebras on the operadic level.

Our first main result is a minimal model of the operad $\RBLA$ of Rota-Baxter Lie algebras of weight $\lambda$. Since $\RBLA$ is not Koszul (in fact not even quadratic, see Section~\ref{Sect: RB Lie}), classical Koszul duality theory does not apply directly. Instead, we construct directly a coaugmented symmetric homotopy cooperad $\mathscr{S}(\RBLA^{\antish})$ (Section~\ref{Sect: RB Lie}), identify its Koszul dual homotopy cooperad $\RBLA^\antish:=\mathscr{S}\RBLA^\antish\ot_{\mathrm{H}}\calS^{-1}$, and show that its cobar construction is a minimal model, by exhibiting an explicit acyclic Morse matching on the free operad generated by the desuspension of the Koszul dual.

\begin{thmA}[{Proposition~\ref{Prop: homotopy cooperad structure} and Theorem~\ref{Thm: Minimal model}}]
The $\mathbb{S}$-module $\mathscr{S}(\RBLA^{\antish})$ carries a natural coaugmented symmetric homotopy cooperad structure, and the dg operad $\RBLA_\infty:=\Omega(\RBLA^{\antish})$, with explicitly computed differentials, is the minimal model of the operad $\RBLA$.
\end{thmA}

The proof uses algebraic Morse theory: we introduce the right path-permutation extension order on shuffle tree monomials, single out the effective tree monomials, and prove that the resulting matching is a Morse matching whose critical tree monomials are precisely the normal forms of Rota-Baxter Lie algebras.

Our second main result extracts the deformation theory from the minimal model. Applying the general convolution formalism to the Koszul dual homotopy cooperad, we obtain an $L_\infty$-algebra $\frakC_{\rmRBLA}(V)$ on the graded space of cochains of a graded space $V$, whose Maurer--Cartan elements are exactly homotopy Rota-Baxter Lie algebra structures on $V$ (Theorem~\ref{Thm: rb-L-infty} and Proposition~\ref{prop: RBLA 1-1 MC}); this gives the notion of homotopy Rota-Baxter Lie algebras of weight $\lambda$. Twisting this $L_\infty$-algebra by the Maurer--Cartan element corresponding to a given Rota-Baxter Lie algebra $(\frakg, \ell, T)$ produces its deformation complex. Concretely, we define the cochain complex $\rmC_{\RBLO}^{\bullet}(\frakg, M)$ of the Rota-Baxter operator $T$ with coefficients in a Rota-Baxter Lie representation $(M,\rho,T_M)$ as the Chevalley--Eilenberg complex of the descendent Lie algebra $\frakg_{\star}$ with coefficients in the representation $_{\rhd}M$, construct an explicit chain map $\Phi^\bullet$ from the Chevalley--Eilenberg complex of $\frakg$ to it (Proposition~\ref{prop: Phi is chain map}), and define the cochain complex $\rmC_{\rmRBLA}^{\bullet}(\frakg,M)$ of the Rota-Baxter Lie algebra as the negative shift of the mapping cone of $\Phi^\bullet$.

\begin{thmA}[{Propositions~\ref{prop: cochain complex of RBLA} and~\ref{prop: dgLa RBLO and twisting}}]
Twisting the $L_\infty$-algebra $\frakC_{\rmRBLA}(\frakg)$ by the Maurer--Cartan element corresponding to $(\frakg,\ell,T)$ recovers the suspension of the cochain complex $\rmC_{\rmRBLA}^{\bullet}(\frakg)$. Moreover, the graded space $\frakC_{\RBLO}(\frakg)$ carries a dg Lie algebra structure whose Maurer--Cartan elements are exactly the Rota-Baxter operators of weight $\lambda$ on $(\frakg,\ell)$, and twisting by such an element recovers the cochain complex $\rmC_{\RBLO}^{\bullet}(\frakg)$ of the Rota-Baxter operator.
\end{thmA}

Our third main result concerns the relationship with the associative theory. The commutator construction sends Rota-Baxter associative algebras to Rota-Baxter Lie algebras; we show that this classical functor deforms to the homotopy level.

\begin{thmA}[{Proposition~\ref{prop: comparison morphism}}]
There is an explicit morphism of dg operads $\Phi: \RBLA_\infty \to \RB_\infty$ from the minimal model of the operad of Rota-Baxter Lie algebras to the minimal model of the operad of Rota-Baxter associative algebras constructed in \cite{WZ24}, given on generators by anti-symmetrization.
\end{thmA}

The paper is organized as follows. In Section~\ref{sect: preliminaries} we recall the necessary preliminaries on algebraic Morse theory, $L_\infty$-algebras and their twisting procedure, and symmetric homotopy (co)operads. In Section~\ref{Sect: RB Lie} we construct the homotopy cooperad $\mathscr{S}(\RBLA^{\antish})$, determine its Koszul dual, make the differential of the cobar construction completely explicit, and prove Theorem~\ref{Thm: Minimal model} via algebraic Morse theory. In Section~\ref{sect: Linfty RBLA} we construct the $L_\infty$-algebra $\frakC_{\rmRBLA}(V)$, define homotopy Rota-Baxter Lie algebras, develop the cohomology theory of Rota-Baxter Lie algebras and of Rota-Baxter operators, and show that both arise from the twisting procedure. In Section~\ref{sect: comparison} we construct the comparison morphism $\Phi: \RBLA_\infty\to \RB_\infty$.

\medskip
\noindent\textbf{Conventions.} Throughout the paper, $\bfk$ is a field of characteristic $0$ and $\lambda\in\bfk$ is a fixed scalar. We use homological grading and the Koszul sign rule everywhere. Shuffles and the shuffle sets $\sh(n_1,\dots,n_p)$ are as recalled in Section~\ref{sect: notations}.

\section{Preliminaries}\label{sect: preliminaries}

In this section, we recall basic notions and methods concerning algebraic Morse theory, $L_\infty$-algebras and homotopy (co)operads.

\subsection{Notations}\label{sect: notations}

Let $\s_n$ be the \textbf{symmetric group} of order $n$. For $n_1, n_2, \cdots, n_p \geqslant 0,$ we denote $\sh(n_1, n_2, \cdots, n_p)$ to be the set of permutations $\sigma \in \s_n$ ($n=\sum_{i=1}^{p} n_i$) satisfies
$$\begin{array}{c}
 \sigma(1) < \sigma(2) < \cdots < \sigma(n_1); \\
\vdots \\
 \sigma(n_{p-1}+1) < \sigma(n_{p-1}+2) < \cdots < \sigma(n_p).
 \end{array}$$
 We call $\sh(n_1, n_2, \cdots, n_p)$ the set of \textbf{$(n_1, n_2, \cdots, n_p)$-shuffles}.

We shall use the \textbf{homological grading} and employ everywhere the  \textbf{Koszul rule} to determine signs.  For a graded space $V = \{V_n \}_{n\in \mathbb{Z}},$  its \textbf{symmetric algebra} $\s(V) $ is defined  to be the quotient algebra of $T(V) : = \oplus_{n\geqslant 1} V^{\ot n}$ modulo the ideal $I$ generated by
$$x\ot y =(-1)^{|x||y|} y \ot x$$
for homogeneous element $x,y \in V$ and define  $n$-th symmetric  power $\s^{n}(V) : = V^{\ot n} /( I \cap V^{\ot n}).$
 For homogeneous elements $v_{1}, \dots, v_{n} \in V$ and $\sigma \in \s_n,$ we define the \textbf{Koszul sign} $\varepsilon(\sigma; v_{1}, \dots, v_{n})$ as
$$v_{1} \ot v_{2} \ot \cdots \ot v_{n}=\varepsilon(\sigma; v_{1}, \dots, v_{n}) v_{{\sigma(1)}} \ot v_{{\sigma(2)}} \ot \cdots \ot v_{{\sigma(n)}}$$
in $\s^{n}(V)$ and introduce also
$$\chi(\sigma; v_{1}, \dots, v_{n})=\sgn(\sigma) \varepsilon(\sigma; v_{1}, \dots, v_{n}),$$
where $\sgn(\sigma)$ is the sign of the permutation $\sigma.$
Similarly, the \textbf{exterior  algebra} $\bigwedge(V)$ is defined to be the quotient algebra of $T(V)$ modulo the ideal $I'$ generated by
$$x\ot y =(-1)^{|x||y| + 1} y \ot x$$
for homogeneous element $x,y \in V$ and define  $n$-th exterior power $\wedge^{n}(V) : = V^{\ot n} / (I' \cap V^{\ot n}).$
For graded spaces $V_1, V_2, \cdots, V_n$ and $\sigma \in \s_n$, we define the \textbf{right action map}
$$r_{\sigma}:V_1\ot V_2 \ot \cdots \ot V_n \to V_1\ot V_2 \ot \cdots \ot V_n$$
by
$$ r_{\sigma}(v_1 \ot v_2 \ot \cdots \ot v_n) : = \varepsilon(\sigma; v_{1}, \dots, v_{n}) v_{{\sigma(1)}} \ot v_{{\sigma(2)}} \ot \cdots \ot v_{{\sigma(n)}}. $$
For a graded space $V=\{V_n \}_{n\in \mathbb{Z}},$ its \textbf{suspension} $sV$ is defined to be $(sV)_p=V_{p-1}, p \in \mathbb{Z}$ and its \textbf{desuspension} $s^{-1}V$ is $(s^{-1}V)_{p}=V_{p+1}.$

\bigskip

\subsection{Algebraic Morse theory}\label{subsect: Morse}\

In this subsection, we will quickly review the algebraic Morse theory. We will follow the notations in \cite{CLZ}, with some slight modifications.

Let $X_*$ be the following complex of vector spaces:
$$\cdots\to X_{n+1}\stackrel{d_{n+1}}{\to} X_n \stackrel{d_{n}}{\to} X_{n-1}\to\cdots$$
Suppose that for each $n\in \Z$,  there exists a decomposition into direct sums of subspaces
$$X_n=\oplus_{i\in I_n}X_{n, i}.$$
So $d_n:X_n\to X_{n-1}$ has a matrix presentation  $d_n=(d_{n, ji})$ with $i\in I_n, j\in I_{n-1}$ and
where $d_{n, ji}: X_{n, i}\to X_{n-1, j}$ is a linear map.
Such decomposition gives a weighted  quiver  $Q=Q_{X_*}$  as follows:
\vspace{1mm}
\begin{itemize}
\item[(Q1)] The  vertices are the pairs   $(n, i)$ with $ n\in \Z, i\in I_n$;
\item[(Q2)] if a map $d_{n, ji}$ with $i\in I_n, j\in I_{n-1}$ does not vanish, then draw an arrow from $(n, i)$ to $(n-1, j)$;
\item[(Q3)] for an arrow in (Q2), its  weight  is just the map  $d_{n, ji}$.
\end{itemize}
\vspace{1mm}
 A \textbf{partial  matching} is a full subquiver $\M$ of $Q$ such that
 \vspace{1mm}
 \begin{itemize}
 \item[(M1)]each vertex in $Q$  belongs to at most  one arrow of $\M$;

 \item[(M2)] each arrow in $\M$  has its weight  invertible as a linear map.
 \end{itemize}
\vspace{1mm}

 Given a  partial matching  $\M$, we define a new weighted quiver $Q^\M$ as follows:
 \begin{itemize}
 \item[(QM1)]Keep everything for all  arrows  which are not in $\M$ (they will be called \textbf{thick arrows});

 \item[(QM2)] For an  arrow in $\M$, replace it by  a new \textbf{dotted arrow} in the reverse direction and the weight of this new arrow is the negative  inverse of the weight of the original arrow.
 \end{itemize}

Given a partial matching $\M$ of $Q$, a vertex of $Q$ is called a \textbf{critical vertex} with respect to $\M$ if it is not incident to  any arrow in $\M$. A path $p$ in $\M$ is called a \textbf{zigzag path} if dotted arrows and  thick arrows  appear  alternately.
Let $m,n\in \mathbb{Z}, i \in I_n, j\in I_m$, $p$ be a path in $Q^{\M}$, denote by
\begin{itemize}
\item[(i)]  $ \V_{n}  $, the set of $(n,i)$ of the quiver $Q$;
\item[(ii)] $ \V_n^{\M}    $, the set of critical vertices of $\V_{n}$;
\item[(iii)] $ \U_n,$ the set of vertices which appear as the starting vertex of an arrow in $\M$;
\item[(iv)] $ \D_n,$ the set of vertices which appear as the ending  vertex of an arrow in $\M$;
\item[(v)] $ \varphi_p^\M,$ the composition of all maps appearing as the weights of all arrows in $p$;
\item[(vi)] $ \Path^\M((n, i), (m, j)),$ the set of all zigzag paths from $(n, i)$ to $(m, j)$  in $Q^\M$;
\item[(vii)] $  \Path^\M_1((n, i), (n, j)),$ the set of all zigzag paths in $\Path^\M((n, i), (m, j))$ which  begin  with  a dotted arrow;
\item[(viii)] $ \Path^\M_2((n, i), (n, j)),$ the set of all zigzag paths in $\Path^\M((n, i), (m, j))$ which  begin  with  a thick arrow.
\end{itemize}

A \textbf{Morse matching} is a partial  matching which satisfies the following local finiteness hypothesis:
 \begin{itemize}
 \item[(LFH)] Given  an arbitrary vertex  $(n, i)\in \D_n$, for each vertex  $(n, j)\in  \V_{n}$ and
     for each element $x\in X_{n, i}$, the sum
     $$\sum_{ p\in  \Path^\M_1 ((n, i),  (n, j) ) } \varphi_p^\M(x)$$
     exists (for instance, it may be a finite sum or it is convergent in a certain norm); moreover,
     the number of vertices $(n, j)\in \V_{n}$ such that  $$ \sum_{p\in  \Path^\M_1 ((n, i),  (n, j) )} \varphi^\M_p(x) \neq 0$$  is finite.

 \end{itemize}

There is a useful  sufficient condition for the Morse matching \cite[Remark3.5]{CLZ}.

\begin{lem}\label{Lem: criterion of morse matching}
If each zigzag path in $Q^{\M}$ which begins with a dotted arrow has finite length, then $\M$ is a Morse matching.
\end{lem}

Given a Morse matching $\M$, we can construct a new complex $( {X}_*^\M, d_*^\M)$ as follows:

The complex ${X}_*^\M$ has its $n$-th component $X_n^\M=\oplus_{(n, i)\in \V_n^\M} X_{n, i}$ and the differential
$d_n^\M: X_n^\M\to X_{n-1}^\M$ has the matrix presentation
$d_n^\M=(d_{n, ji}^\M)$ with $(n, i)\in \V_n^\M, (n-1, j)\in \V_{n-1}^\M$ and where
$d_{n, ji}^\M: X_{n, i}\to X_{n-1, j} $  is defined to be
$$d_{n, ji}^\M=\sum_{p\in  \Path^\M((n, i), (n-1, j))} \varphi^\M_p.$$

The main result of algebraic Morse theory \cite{CLZ} is as follows:

\begin{thm}\label{Thm: main result of algebraic Morse theory}
 Under the above setup, the following holds:
\begin{itemize}
\item[(i)]
 $({X}_*^\M, d^\M)$ is a complex.

\item[(ii)]   Define maps \begin{align*}
f_n: X_n^\M & \rightarrow X_n \\
 x \in X_{n,i} & \mapsto f_n(x):=x+ \sum_{(n, j)\in \U_n}\sum_{p\in  \Path^\M_2   ((n,i), (n, j))} \varphi^\M_p(x),
\end{align*}and

\begin{align*}
g_n: X_n & \rightarrow X^\M_n \\
 x \in X_{n,i} & \mapsto g_n(x):=\left\{\begin{array}{ll} \sum_{(n, j)\in \V^\M_n}\sum_{p\in   \Path^\M_1   ((n,i), (n, j)) } \varphi^\M_p(x),& (n, i)\in \D_n\\
  x,  &  (n,i) \in \V_n^\M \\
 0 & (n, i)\in \U_n\end{array}\right.
\end{align*}
 Then $f_*: {X}_*^\M   \rightarrow X_* $ and $ g_*: X_*   \rightarrow {X}_*^\M $ are chain maps which are homotopy equivalences:
$gf=\Id_{{X}_*^\M}$ and $fg\sim \Id_{X_*}$ via the homotopy
\begin{align*}
{\theta_n}: X_n & \rightarrow X_{n+1} \\
  x \in X_{n,i} & \mapsto {\theta_n}(x):= \left\{\begin{array}{ll} \sum_{(n+1, j)\in \U_{n+1}}\sum_{p\in \Path^\M((n,i), (n+1, j)) } \varphi^\M_p(x),& (n, i)\in \D_n\\
   0 & otherwise\end{array}\right.
\end{align*}

    \item[(iii)] We have a decomposition $$(X_*,d) \cong (X_* ^\M) \oplus (Y_*,d^Y)$$ where $(Y_*,d^Y)$ is a null homotopic  complex.
\end{itemize}
\end{thm}

 \subsection{$L_\infty$-algebras and Maurer-Cartan elements}\ \label{Subsect: DGLAs and Linfinity algebras}
In this subsection, we will recall some preliminaries on   $L_\infty$-algebras. For more background, we refer the reader to \cite{Sta92, LS93, LM95,Get09}.

\begin{defn}\label{Def:L-infty}
	Let $L=\bigoplus\limits_{i\in\mathbb{Z}}L_i$ be a graded space. Assume that $L$ is endowed with a family of graded linear operators $l_n:L^{\ot n}\rightarrow L, n\geqslant 1$ with  $|l_n|=n-2$ subject to  the following conditions:
	for arbitrary  $n\geqslant 1$,  $ \sigma\in \s_n$ and $x_1,\dots, x_n\in L$,
	\begin{enumerate}
		\item[(i)](generalised anti-symmetry) $$l_n(x_{\sigma(1)}\ot \cdots \ot x_{\sigma(n)})=\chi(\sigma; x_1,\dots,   x_n)\ l_n(x_1 \ot \cdots  \ot x_n);$$

		\item[(ii)](generalised Jacobi identity)
		$$\sum\limits_{i=1}^n\sum\limits_{\sigma\in \Sh(i,n-i)}\chi(\sigma; x_1, \dots, x_n)(-1)^{i(n-i)}l_{n-i+1}(l_i(x_{\sigma(1)}\ot \cdots \ot x_{\sigma(i)})\ot x_{\sigma(i+1)}\ot \cdots \ot x_{\sigma(n)})=0,$$
	\end{enumerate}
	Then $(L,\{l_n\}_{n\geqslant1})$ is called an $L_\infty$-algebra.
\end{defn}

\begin{remark} \label{Rem: L-infinity for small n}   Let us consider the generalised Jacobi identity for   $n\leqslant 3$ with the assumption of generalised anti-symmetry.
	
	\begin{enumerate}
		\item[(i)]  $n=1$,  $l_1\circ l_1=0$, that is,  $l_1$ is a differential,

		\item[(ii)]  $n=2$, $l_1\circ l_2=l_2\circ (l_1\ot\Id+\Id\ot l_1)$, that is   $l_1$ is a derivation for  $l_2$,

		\item[(iii)] $n=3$, for homogeneous elements $x_1, x_2, x_3\in L$
		$$\begin{array}{ll} &l_2(l_2(x_1\ot x_2)\ot x_3)+(-1)^{|x_1|(|x_2|+|x_3|)} l_2(l_2(x_2\ot x_3)\ot x_1)+
			(-1)^{|x_3|(|x_1|+|x_2|)} l_2(l_2(x_3\ot x_1)\ot x_2)
			\\
			=&-\Big(l_1(l_3(x_1\ot x_2\ot x_3))+ l_3(l_1 (x_1)\ot x_2\ot x_3 )+(-1)^{|x_1|} l_3(x_1\ot l_1 (x_2)\ot x_3 )+\\
			&(-1)^{|x_1|+|x_2|} l_3(x_1\ot x_2\ot l_1 (x_3) )\Big),\end{array}$$
		that is, $l_2$ satisfies the   Jacobi identity up to homotopy.
	\end{enumerate}
	
	In particular,    if all   $l_n=0$ with  $n\geqslant 3$, then $(L,l_1,l_2)$ is just a dg Lie algebra.
	
\end{remark}

One can also define Maurer-Cartan elements in $L_\infty$-algebras.
\begin{defn}\label{Def: MC element in L infinity algebra}
	Let $(L,\{l_n\}_{n\geqslant1})$ be an $L_\infty$-algebra. An element $\alpha\in L_{-1}$ is called a Maurer-Cartan element if it satisfies the Maurer-Cartan equation:
	\begin{eqnarray}\label{Eq: mc-equation}\sum_{n=1}^\infty\frac{1}{n!}(-1)^{\frac{n(n-1)}{2}} l_n(\alpha^{\ot n})=0,\end{eqnarray}
	whenever this infinite sum exists.
\end{defn}

\begin{prop}[{Twisting procedure \cite[Section 4]{Get09}}]\label{Prop: deformed-L-infty}
	Given a Maurer-Cartan element $\alpha$ in $L_\infty$-algebra $L$, one can introduce  a new $L_\infty$-structure $\{l_n^\alpha\}_{n\geqslant 1}$ on graded space $L$, where $l_n^{\alpha}: L^{\ot n}\rightarrow L$ is defined as :
	\begin{eqnarray}\label{Eq: twisted L infinity algebra} l^\alpha_n(x_1\ot \cdots\ot x_n)=\sum_{i=0}^\infty\frac{1}{i!}(-1)^{in+\frac{i(i-1)}{2}}l_{n+i}(\alpha^{\ot i}\ot x_1\ot \cdots\ot x_n),\ \forall x_1, \dots, x_n\in L,\end{eqnarray}
	whenever these infinite sums exist. The new $L_\infty$-algebra $(L, \{l_n^\alpha\}_{n\geqslant 1})$ is called the twisted $L_\infty$-algebra (by the Maurer-Cartan element $\alpha$).
\end{prop}

\bigskip

\subsection{Symmetric homotopy (co)operads}\label{sec: symmetric homotopy (co)operads}\label{Sect: homotopy cooperad}
In this subsection, we collect some basics on symmetric homotopy (co)operads and explain how to obtain $L_\infty$-structures from homotopy operads, in particular, from convolution homotopy operads. 

Recall that a \textbf{graded $\mathbb S$-module} $\calP=\{\calP(n)\}_{n\geqslant 1}$ consists of a collection of graded spaces such that each $\calP(n)$ is a right $\bfk[\mathbb S_n]$-module. The \textbf{suspension} of $\calP$, denoted by $s\calP$, is defined to be the graded $\mathbb S$-module $\{s\calP(n)\}_{n\geqslant 1}$. Similarly one have the \textbf{desuspension} $s^{-1}\calP$ of the graded $\mathbb S$-module $\calP$.

To give a clear account of operad theory, we need to resort to the language of trees. We consider only reduced planar rooted trees and will simply refer to them as trees from now on. For a tree $T$, let $\omega(T)$ denote its \textbf{weight} (number of internal vertices) and $\alpha(T)$ its \textbf{arity} (number of leaves). We will refer to internal vertices simply as vertices unless otherwise stated. For each vertex $v$, the arity $\alpha(v)$ is the number of children of $v$. Let $\mathfrak{T}$ be the set of all trees with weight $\geqslant 1$ and arity $\geqslant 1$. For integers $n,m\geqslant 1$, define 
$$
\frakT^{(n)}:=\{T\in\frakT\mid \omega(T)=n\},\quad \frakT(m):=\{T\in\frakT\mid \alpha(T)=m\},\quad \frakT^{n}(m):=\frakT^{(n)}\cap \frakT(m).
$$ 
Since the trees are planar, the inputs of each vertex are totally ordered from left to right. This order induces a total order on the set of all vertices of any given $T\in\frakT$, which we call the \textbf{planar order} (specifically, the order obtained by traversing the outer boundary of the tree clockwise). More formally,  for any two vertices $v$ and $v'$ with root paths $r=v_0\rightarrow v_1\rightarrow\dots\rightarrow v_k=v$ and $r=v'_0\rightarrow v'_1\rightarrow\dots\rightarrow v'_l=v'$, we set $v<v'$ if and only if there is some $i$ such that $v_j=v'_j$ for all $j<i$ and $v_i<v'_i$. 

\begin{exam}
    For the tree
    
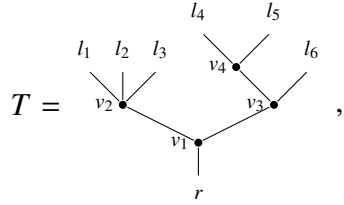
\begin{align*}
        \begin{tikzpicture}[scale=1,descr/.style={fill=white}]
            \tikzstyle{every node}=[thick,minimum size=3pt, inner sep=1pt]
            \node(T)at(-1.7,0.5)[label=left:{$T=$}]{};
            \node(dot)at(1.7,0.4)[label=right:{,}]{};
            \node(r)at(0,-0.5)[label=below:{\tiny$r$}]{};
            \node(v1)at(0,0)[fill=black,circle,label=left:{\tiny$v_1$}]{};
            \node(v2)at(-1,0.5)[fill=black,circle,label=left:{\tiny$v_2$}]{};
            \node(v3)at(1,0.5)[fill=black,circle,label=left:{\tiny$v_3$}]{};
            \node(v4)at(0.5,1)[fill=black,circle,label=left:{\tiny$v_4$}]{};
            \node(v2-1)at(-1.5,1)[label=above:{\tiny$l_1$}]{};
            \node(v2-2)at(-1,1)[label=above:{\tiny$l_2$}]{};
            \node(v2-3)at(-0.5,1)[label=above:{\tiny$l_3$}]{};
            \node(v3-1)at(1.5,1)[label=above:{\tiny$l_6$}]{};
            \node(v4-1)at(0,1.5)[label=above:{\tiny$l_4$}]{};
            \node(v4-2)at(1,1.5)[label=above:{\tiny$l_5$}]{};
            \draw(r)--(v1);\draw(v1)--(v2);\draw(v1)--(v3);
            \draw(v2)--(v2-1);\draw(v2)--(v2-2);\draw(v2)--(v2-3);
            \draw(v3)--(v3-1);\draw(v3)--(v4);
            \draw(v4)--(v4-1);\draw(v4)--(v4-2);
        \end{tikzpicture}
    \end{align*}
    we have $T\in\frakT^{(4)}(6)$. The set of all vertices of $T$ (including the root and leaves), in planar order, is $\{r<v_1<v_2<l_1<l_2<l_3<v_3<v_4<l_5<l_6<l_7\}$.
\end{exam}

A tree $T$ is called a \textbf{labelled tree} if there is a bijective \textbf{labelling map} $\Phi$ from $\{1,2,\dots, \alpha(T)\}$ to the set of the leaves of $T$. Let  $\frakLT$ be the set of all labelled trees. 
For a tree $T\in\frakLT$, each vertex is assigned the minimum label among its descendant leaves. A tree $T$ is called \textbf{trivially labelled} if its leaf labels increase according to the planar order, and a \textbf{shuffle tree} if the labels of the children of each vertex are increasing with respect to the planar order. The set of all shuffle trees is denoted by $\frakST$. We define 
$$
\frakST^{(n)}:=\frakT^{(n)}\cap\frakST,\quad \frakST(m):=\frakT(m)\cap\frakST.
$$
Unless specified otherwise, all trees considered hereafter are labelled trees. 

There is a right $\mathbb S_n$-action on $\frakLT(n)$. For $T\in\frakLT(n)$ and $\sigma\in\mathbb S_n$, the tree $T\cdot\sigma$ is obtained from $T$ by replacing the labelling map $\Phi$ with $\Phi\circ\sigma$, or equivalently, by replacing each label $i$ with $\sigma^{-1}(i)$.

Let $T_1, T_2\in\frakLT$. We say $T_1$ and $T_2$ are \textbf{isomorphic} if there exists a bijection $\theta$ between their vertex sets such that
\begin{enumerate}
    \item $\theta$ maps leaves to leaves, internal vertices to internal vertices, and the root to the root;
    \item for each vertex $v$ of $T_1$, $\theta$ maps the children of $v$ to the children of $\theta(v)$;
    \item for each leaf $l$ of $T_1$, the label of $l$ is equal to the label of $\theta(l)$. 
\end{enumerate}
Informally, $T_1$ and $T_2$ are isomorphic if $T_2$ can be obtained by rotating $T_1$. 
For an isomorphism $\theta: T_1\rightarrow T_2$ of trees, there is an associated family of permutations $(\tau; \sigma_1,\dots,\sigma_n)$. Let $\{v_1<\dots<v_n\}$ and $\{v'_1<\dots<v'_n\}$ denote the sets of internal vertices of $T_1$ and $T_2$, respectively, under the planar order. The permutation $\tau\in\mathbb S_n$ is defined by setting  $v'_i=\theta(v_{\tau(i)})$ for each $i=1,\dots,n$. Moreover, for each vertex $v_i$ of $T_1$, let $\{w_1<\dots<w_{\alpha(v_i)}\}$ be the children of $v_i$ under the planar order, and let $\{w'_1<\dots<w'_{\alpha(v_i)}\}$ be the children of $\theta(v_i)$. The permutation $\sigma_i$ is then determined by  $w'_j=\theta(w_{\sigma_i(j)})$ for each $j=1,\dots,\alpha(v_i)$.  Note that there is a natural bijection between the isomorphism classes of trees and the set of shuffle trees.

A trivially labelled tree $T'$ is called a subtree of $T$, denoted $T'\subset T$, if $T'$ is a divisor of $T$ as planar trees. Define the \textbf{quotient tree} $T/T'$ to be the tree obtained by replacing $T'$ in $T$ with a corolla (i.e., a tree of weight $1$) of the same arity. The pair $(T, T')$ naturally induces a permutation $\sigma=\sigma(T, T')$ as follows.
Suppose $\omega(T)=n$ and $\omega(T')=j$. Let $\{v_1 < v_2 < \dots < v_n\}$ be the set of internal vertices of $T$ under the planar order, and $v'$ be the $i$-th vertex of $T/T'$ corresponding to the corolla replacing $T'$. Define $\sigma \in \mathbb S_{n}$ to be the unique permutation such that, in the sequence
$$
(v_{\sigma(1)}, \dots, v_{\sigma(i-1)}, v_{\sigma(i)}, \dots, v_{\sigma(i+j-1)}, v_{\sigma(i+j)}, \dots, v_{\sigma(n)}),
$$
the set 
$\{v_{\sigma(i)} < \dots < v_{\sigma(i+j-1)}\}$ is precisely the set of internal vertices of $T'$ under the planar order, 
and $\{v_{\sigma(1)} < \dots < v_{\sigma(i-1)} < v' < v_{\sigma(i+j)} < \dots < v_{\sigma(n)}\}$ is the set of internal vertices of $T/T'$.

Let $\calP=\{\calP(n)\}_{n\geqslant 1}$ be a graded $\mathbb{S}$-module and $T\in \frakT^{(n)}$ be a tree whose internal vertices are $\{v_1 < v_2 < \dots < v_n\}$. Define the graded space of tree tensors as
$$
\calP^{\otimes T} : = \calP(\alpha(v_1)) \otimes \calP(\alpha(v_2)) \otimes \dots \otimes \calP(\alpha(v_n)),
$$
or equivalently, an element of $\calP^{\otimes T}$ is a linear combination of trees with each vertex $v_i$ decorated by an element of $\calP(\alpha(v_i))$. Let $B$ be a $\bfk$-basis of $\calP$.  An element $x\in\calP^{\otimes T}$ is called a \textbf{tree monomial} if each vertex of $x$ is decorated by an element of $B$.
Note that for $T\in \frakT(n)$ and $\sigma \in \mathbb S_n$, it is clear that $ \calP^{\otimes T}\cong\calP^{\otimes T \cdot \sigma}$ as graded spaces.

\begin{defn} \label{def:sym homotopy operads}
	A \textbf{symmetric homotopy operad} structure on a graded $\mathbb S$-module $\calP=\{\calP(n)\}_{n\geqslant 1}$ consists of a family of operations 
    $$
    \{m_T: \calP^{\otimes T}\rightarrow \calP(\alpha(T))\}_{T\in \frakLT}
    $$ 
    with degree $|m_T|=\omega(T)-2$, such that the following conditions hold:
    \begin{enumerate}
        \item For any $T\in \frakLT^{(n)}(m)$, $\sigma\in \mathbb S_{m}$, and  $x_1\otimes\dots \otimes x_n \in \calP^{\otimes T} \cong \calP^{\otimes T\cdot \sigma}$, we have
        $$
        m_{T \cdot \sigma}(x_1 \otimes x_2  \otimes \dots \otimes x_n )=m_T(x_1\otimes x_2 \otimes  \dots \otimes x_n) \cdot \sigma.
        $$
        \item Let $\theta: T_1\rightarrow T_2$ be an isomorphism of trees in $\frakLT^{(n)}$, and let $x_1\otimes \dots \otimes x_n \in \calP^{\otimes T_1}$. Then 
        $$
        m_{T_1}(x_1 \otimes \dots \otimes x_n ) = \chi(\tau; x_1,\dots, x_n) m_{T_2}(x_{\tau(1)}\cdot\sigma_{\tau(1)} \otimes \dots \otimes x_{\tau(n)}\cdot\sigma_{\tau(n)}),
        $$
        where $(\tau;\sigma_1,\dots,\sigma_n)$ is the family of permutations induced by $\theta$.
        \item For any $T\in \frakLT$, the relation
        $$
        \sum_{T'\subset{T}}(-1)^{i-1+jk}\sgn(\sigma(T,T'))\ m_{T/T'}\circ(\Id^{\otimes {i-1}}\otimes m_{T'}\otimes \Id^{\otimes k})\circ r_{\sigma(T,T')}=0
        $$
        holds,  where $i$ is the serial number of the vertex corresponding to $T'$ in $T/T'$, $j=\omega(T')$, $k=\omega(T)-i-j+1$, and $r_{\sigma(T,T')}$ denotes the right action of $\sigma:=\sigma(T,T')$, i.e., 
        $$
        r_\sigma(x_1\otimes \dots\otimes x_n)=\varepsilon(\sigma; x_1,\dots, x_{n})x_{\sigma(1)}\otimes \dots  \otimes x_{\sigma(n)}.
        $$   
    \end{enumerate}
\end{defn}

A \textbf{strict morphism} of homotopy operads is a morphism of graded $\mathbb S$-module compatible with all operations $\{m_T\}_{T\in \frakLT}$.

Let $\calI$ be the collection defined by $\calI(1)=\bfk$ and $\calI(n)=0$ for $n\ne1$. The collection $\calI$ is equipped with a symmetric homotopy operad structure in the natural way: namely, $m_T: \calI(1)\otimes \calI(1)\to \calI(1)$ is the identity when $T$ is the tree with two internal vertices and a unique leaf, and $m_T=0$ otherwise.

A symmetric homotopy operad $\calP$ is called \textbf{strictly unital} if there exists a strict morphism of symmetric homotopy operads $\eta: \calI\rightarrow \calP$ such that, for each $n\geqslant 1$, the compositions  
$$
\calP(n)\cong \calP(n)\otimes \calI(1)\xrightarrow{\Id\otimes \eta}\calP(n)\otimes \calP(1)\xrightarrow{m_{T_{1, i}}}\calP(n)
$$
and
$$
\calP(n)\cong \calI(1)\otimes \calP(n)\xrightarrow{\eta\otimes \Id}\calP(1)\otimes \calP(n)\xrightarrow{m_{T_2}} \calP(n)
$$
are the identity maps on $\calP(n)$. Here $T_{1,i}\in\frakLT^{(2)}(n)$ is the trivially labelled tree whose second vertex has arity $1$ and attaches to the $i$-th child of the first vertex, for $1\leqslant i\leqslant n$; and $T_2\in\frakLT^{(2)}(n)$ is the trivially labelled tree whose first vertex has arity $1$, as shown in the figure below. 

\begin{align*}
        \begin{tikzpicture}[scale=1,descr/.style={fill=white}]
			\tikzstyle{every node}=[thick,minimum size=3pt, inner sep=1pt]
            \node(r)at(0,-0.5){};
            \node(v1)at(0,0)[fill=black, circle,label=left:{\tiny$v_1$}]{};
            \node(v2)at(0,0.5)[fill=black, circle,label=left:{\tiny$v_2$}]{};
            \node(v1-1)at(-0.8,0.5)[label=above:{\tiny$1$},label=left:{$T_{1,i}=\ \ $}]{};
            \node(v1-2)at(0.8,0.5)[label=above:{\tiny$n$}]{};
            \node(v2-1)at(0,1)[label=above:{\tiny$i$}]{};
            \node(dot)at(0.8,0.3)[label=right:{\ \ ,}]{};
            \draw(r)--(v1);\draw(v1)--(v2);
            \draw(v1)--(v1-1);\draw(v1)--(v1-2);
            \draw(v2)--(v2-1);
            \draw[dotted](-0.4,0.3)--(0.4,0.3);
		\end{tikzpicture}\ \qquad \ 
        \begin{tikzpicture}[scale=1,descr/.style={fill=white}]
			\tikzstyle{every node}=[thick,minimum size=3pt, inner sep=1pt]
            \node(r)at(0,-0.5){};
            \node(v1)at(0,0)[fill=black, circle,label=left:{\tiny$v_1$}]{};
            \node(v2)at(0,0.5)[fill=black, circle,label=left:{\tiny$v_2$}]{};
            \node(v2-1)at(-0.8,1)[label=above:{\tiny$1$}]{};
            \node(v2-2)at(0.8,1)[label=above:{\tiny$n$}]{};
            \node(T2)at(-0.8,0.5)[label=left:{$T_{2}= $}]{};
            \node(dot)at(0.8,0.3)[label=right:{\ \ .}]{};
            \draw(r)--(v1);\draw(v1)--(v2);
            \draw(v2)--(v2-1);\draw(v2)--(v2-2);
            \draw[dotted](-0.4,0.8)--(0.4,0.8);
		\end{tikzpicture}
\end{align*}
A symmetric homotopy operad $\calP$ is called \textbf{augmented} if there exists a strict morphism of symmetric homotopy operads $\varepsilon: \calP\rightarrow \calI$ such that $\varepsilon \circ \eta=\Id_{\calI}$.

If a homotopy operad $\calP$ satisfies $m_T=0$ for all $T\in \frakLT$ with $\omega(T)\geqslant 3$, then $\calP$ is precisely a nonunital dg operad in the sense of Markl \cite{Mar08}, where $\{m_T\mid T\in\frakST^{(1)}\}$ and $\{m_T\mid T\in\frakST^{(2)}\}$ give its differentials and partial compositions, respectively.

There is a natural $L_\infty$-algebra associated with any symmetric homotopy operad $\calP=\{\calP(n)\}_{n\geqslant 1}$.
Let $\calP^{\prod}:= \prod_{n=1}^\infty\calP(n)$. For each $n\geqslant 1$, define operations 
$$
m_n:=\sum\limits_{T\in \frakST^{(n)}}m_T: (\calP^{\prod})^{\otimes n}\to \calP^{\prod}
$$ 
and let $l_n$ be the anti-symmetrization of $m_n$, given by 
$$
l_n(x_1\otimes \dots \otimes x_n)=\sum\limits_{\sigma\in \mathbb S_n}\chi(\sigma;x_1,\dots, x_n)m_n(x_{\sigma(1)}\otimes\dots \otimes x_{\sigma(n)}).
$$

\begin{prop}[\cite{MV1, Van02}]\label{prop:homotopy operad->L_infty}
	Let $\calP$ be a symmetric homotopy operad. Then $(\calP^{\prod}, \{l_n\}_{n\geqslant 1})$ forms an $L_\infty$-algebra.
\end{prop}

Let $\calP(n)^{\mathbb S_n}\subset \calP(n)$ denote the invariant subspace under the right $\mathbb S_n$-action, and define $\calP_{\mathbb S}^{\prod}:=\prod_{n=1}^\infty\calP(n)^{\mathbb S_n}\subset \calP^{\prod}$. By restricting the operations $\{l_n\}_{n\geqslant 1}$ from $\calP^{\prod}$ to $\calP_{\mathbb S}^{\prod}$, we obtain the following:
\begin{prop}\label{prop:Linfty on subspace of sym homotopy operad}
    Let $\calP$ be a symmetric homotopy operad. Then $(\calP_{\mathbb S}^{\prod}, \{l_n\}_{n\geqslant 1})$ is an $L_\infty$-algebra. In particular, if $\calP$ is a symmetric dg operad, then $\calP_{\mathbb S}^{\prod}$ reduces to a dg Lie algebra.
\end{prop}

\begin{defn}
    Given a symmetric operad $\calP$, for each tree $T\in\frakLT$, the \textbf{composition along $T$ in $\calP$}
\begin{align*}
    m_\calP^T:\calP^{\otimes T}\longrightarrow \calP(\alpha(T))
\end{align*}
is defined inductively as follows:
\begin{enumerate}
    \item For the trivially labelled tree $T\in\frakLT^{(1)}(m)$ with $m\geqslant 1$, let $\sigma\in\mathbb S_m$ and $x\in\calP(m)$. We define $m_\calP^{T\cdot\sigma}(x)=x\cdot\sigma$;
    \item For $T\in\frakLT^{(2)}$, we define $m_\calP^T=m_T$;
    \item For $T\in\frakLT^{(n)}$ with $n\geqslant 3$, let $\{v_1<\dots<v_n\}$ be the set of internal vertices of $T$. Let $T'$ denote the subtree whose set of internal vertices is $\{v_1<\dots<v_{n-1}\}$. We then define $m_\calP^T=m_{T/T'}\circ(m_\calP^{T'}\otimes\Id)$, where $m_\calP^{T'}$ is given by induction.
\end{enumerate}
\end{defn}

\begin{defn}\label{Def: brace operation of shuffle tree}
Let $\mathcal P$ be a (nonunital) symmetric operad. For $f\in\mathcal P(m)$ and $g_j\in\mathcal P(r_j)$, $j=1,\dots,n$, with $1\leqslant n\leqslant m$, the \textbf{shuffle brace operation} of $(f; g_1,\dots,g_n)$ is defined as
\begin{align}\label{eq:shuffle brace operation}
f\{g_1,\dots,g_n\}=\sum_{T\cdot\sigma}m_{\mathcal{P}}^T(f\otimes g_1\otimes\dots\otimes g_n)\cdot\sigma,
\end{align}
where the sum is taken over all shuffle trees $T\cdot\sigma$ of the form
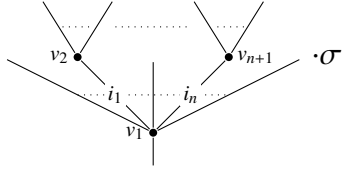
\begin{align*}
        \begin{tikzpicture}[scale=1,descr/.style={fill=white}]
			\tikzstyle{every node}=[thick,minimum size=3pt, inner sep=1pt]
            \node(r)at(0,-0.5){};
            \node(v1)at(0,0)[fill=black, circle,label=left:{\tiny$v_1$}]{};
            \node(v2)at(-1,1)[fill=black, circle,label=left:{\tiny$v_2$}]{};
            \node(v3)at(1,1)[fill=black, circle,label=right:{\tiny$v_{n+1}$}]{};
            \node(v1-1)at(-2,1){};\node(v1-2)at(0,1){};\node(v1-3)at(2,1){};
            \node(v2-1)at(-1.5,1.8){};\node(v2-2)at(-0.5,1.8){};
            \node(v3-1)at(0.5,1.8){};\node(v3-2)at(1.5,1.8){};
            \node(sigma)at(2,1)[minimum size=0pt, label=right:{$\cdot\sigma$}]{};
            \draw(r)--(v1);
            \draw(v1)--(v1-1);\draw(v1)--(v1-2);\draw(v1)--(v1-3);
            \draw(v2)--(v2-1);\draw(v2)--(v2-2);
            \draw(v3)--(v3-1);\draw(v3)--(v3-2);
            \draw[dotted](-1,0.5)--(1,0.5);
            \draw[dotted](-1.2,1.4)--(-0.8,1.4);
            \draw[dotted](1.2,1.4)--(0.8,1.4);
            \draw[dotted](-0.5,1.4)--(0.5,1.4);
            \path[-,font=\scriptsize] 
            (v1) edge node[descr]{{\tiny$i_1$}} (v2)
            edge node[descr]{{\tiny$i_n$}} (v3);
		\end{tikzpicture}
\end{align*}
Here, $T$ is a trivially labelled tree with $\alpha(v_1)=m$, $\alpha(v_{j+1})=r_j$ for $j=1,\dots n$, and $1\leqslant i_1<\dots< i_n\leqslant m$. Furthermore, for $f\in\mathcal P(m)$ and $g\in\mathcal P(n)$, we define a bracket operation
$$
[f,g]:=f\{g\}-(-1)^{|f||g|}g\{f\}\in \mathcal P(m+n-1).
$$
\end{defn}

\begin{remark}\label{rmk: brace operation and NR bracket}
	\begin{enumerate}
	    \item If $\sigma$ is the identity permutation throughout the summation in Equation \eqref{eq:shuffle brace operation}, then the shuffle brace operation reduce to the classical \textbf{brace operation}.
        \item The operation $l_2$ in the dg Lie algebra $\calP^{\prod}$ coincides exactly with the bracket operation defined above.
        \item Let $\mathcal P=\End_{sV}$. Then the bracket operation on $(\End_{sV}^{\prod})_{\mathbb{S}}=\prod_{n\geqslant 1}\Hom((sV)^{\odot n}, sV)$ is precisely the \textbf{Richardson-Nijenhuis bracket}, which endows $(\End_{sV}^{\prod})_{\mathbb{S}}$ with a Lie algebra structure. 
	\end{enumerate}
\end{remark}

\begin{lem}\label{lem:shuffle brace operation on End_sV}
Let $\mathcal P=(\End_{sV})_{\mathbb S}$. For $f\in\Hom((sV)^{\odot k},sV)$, $g_i\in\Hom((sV)^{\odot r_i},sV)$ with $i=1,\dots,l$ and $1\leqslant l\leqslant k$, we have
\begin{align*}
    f\{g_1,\dots,g_l\}=\sum_{\substack{\sigma\in\Sh(r_1,\dots,r_l,k-l)\\ \sigma(1)<\sigma(r_1+1)<\dots<\sigma(r_1+\dots+r_{l-1}+1)}}f(g_1\otimes \dots \otimes g_l\otimes\Id^{\otimes k-l})r_\sigma.
\end{align*}
\end{lem}

Dualizing the definition of symmetric homotopy operads, one obtains the notion of symmetric homotopy cooperads.

\begin{defn}
    Let $\calC=\{\calC(n)\}_{n\geqslant 1}$ be a graded $\mathbb S$-module. A \textbf{symmetric homotopy cooperad} structure on $\calC$ consists of a family of operations $\{\Delta_T: \calC(\alpha(T))\rightarrow \calC^{\otimes T }\}_{T\in\frakLT}$ with degree $|\Delta_T|=\omega(T)-2$, such that for any $c\in \calC$, $\Delta_T(c)=0$ for all but finitely many $T\in\frakLT$. Furthermore, this family of operations $\{\Delta_T\}_{T\in\frakLT}$ is required to satisfy the following conditions:
    \begin{enumerate}
        \item For any $T\in\frakLT(n)$ and $\sigma\in\mathbb S_n$, we have $\Delta_T(c)=\Delta_{T\cdot\sigma}(c\cdot\sigma)$.
        \item Let $\theta:T_1\rightarrow T_2$ be an isomorphism of trees in $\frakLT^{(n)}(m)$, and let $c\in\calC(m)$. Suppose that $\Delta_{T_1}(c)=x_1\otimes\dots\otimes x_n$. Then
        $$
        \Delta_{T_2}(c)=\chi(\tau;x_1,\dots,x_n)x_{\tau(1)}\cdot\sigma_{\tau(1)}\otimes\dots\otimes x_{\tau(n)}\cdot\sigma_{\tau(n)},
        $$
        where $(\tau;\sigma_1,\dots,\sigma_n)$ is the family of permutations induced by $\theta$.
        \item For any $T\in \frakLT$, the relation 
        $$
        \sum_{T'\subset T}\sgn\big(\sigma(T,T')^{-1}\big)(-1)^{i-1+jk} r_{\sigma(T,T')^{-1}}\circ (\Id^{\otimes i-1}\otimes \Delta_{T'}\otimes \Id^{\otimes k})\circ  \Delta_{T/T'}=0
        $$
        holds, where the indices $i,j,k$ are defined as in Definition \ref{def:sym homotopy operads}.
    \end{enumerate}
\end{defn}

The graded collection $\calI$ carries a natural symmetric homotopy cooperad structure, where $\Delta_T: \calI(1)\to  \calI(1)\otimes \calI(1)$ is given by the identity map when $T$ is the unique tree in $\frakLT^{(2)}(1)$, and $\Delta_T$ vanishes otherwise.
A symmetric homotopy cooperad $\calC$ is called \textbf{strictly counital} if there exists a strict morphism of symmetric homotopy cooperad $\varepsilon: \calC\rightarrow \calI$ such that the compositions
$$
\calC(n)\xrightarrow{\Delta_{T_{1, i}}}\calC(n)\otimes \calC(1)\xrightarrow{\Id\otimes \varepsilon}\calC(n)\otimes \calI(1)\cong \calC(n)
$$ 
and 
$$
\calC(n)\xrightarrow{\Delta_{T_2}}\calC(1)\otimes \calC(n)\xrightarrow{\varepsilon\otimes \Id}\calI(1)\otimes \calC(n)\cong \calC(n)
$$ 
are the identity maps on $\calC(n)$, where $T_{1,i}$ and $T_2$ are the trees defined previously. Moreover, for any $T\in\frakLT^{(n)}$ with $n\geqslant 3$, the composition $(\Id^{\otimes i-1}\otimes \varepsilon \otimes\Id^{\otimes n-i})\circ\Delta_T$ is required to vanish for all $1\leqslant i\leqslant n$.

A symmetric homotopy cooperad $\calC$ is called \textbf{coaugmented} if there exists a strict morphism of symmetric homotopy cooperads $\eta:\calI\rightarrow \calC$ such that $\varepsilon\circ \eta=\Id_{\calI}$. For a coaugmented symmetric homotopy cooperad $\calC$, the graded $\mathbb S$-module $\overline{\calC}=\operatorname{ker}(\varepsilon)$, endowed with the operations $\{\overline{\Delta}_T\}_{T\in\frakLT}$ obtained by restricting $\Delta_T$ to $\overline{\calC}$, naturally inherits a symmetric homotopy cooperad structure. 

A symmetric homotopy cooperad $\calE = \{\calE(n)\}_{n\geqslant 1}$ for which the operations $\{\Delta_T\}$ vanish for all trees with $\omega(T) \geqslant 3$ reduces precisely to a noncounital symmetric dg cooperad in the sense of Markl \cite{Mar08}.

For a (noncounital) symmetric dg cooperad $\calE$ and a tree $T\in\frakLT$, one can define the \textbf{cocomposition along $T$}, denoted by  $\Delta^{T}_\calE:\calE(\alpha(T))\to \calE^{\otimes T}$, in a manner dual to the composition $m^T_\calP$ along $T$ for a dg operad $\calP$.

\begin{prop-def}
    Let $\calC$ be a symmetric homotopy cooperad and $\calE$ be a symmetric dg cooperad. Then the graded collection $\calC\otimes \calE$, defined by $(\calC\otimes\calE)(n) := \calC(n)\otimes \calE(n)$ for $n\geqslant 1$, carries a natural symmetric homotopy cooperad structure given as follows:
	\begin{itemize}
		\item For any $c\in\calC(n)$, $e\in\calE(n)$, and $\sigma\in\mathbb S_n$, the right $\mathbb S_n$-action on $(\calC\otimes\calE)(n)$ is defined by 
        $$
        (c\otimes e)\sigma:=(c\cdot\sigma)\otimes (e\cdot\sigma).
        $$
        \item For any tree $T\in \frakLT^{(1)}(n)$, the operation $\Delta_T^H$ on $c\otimes e \in \calC(n)\otimes \calE(n)$ is defined by 
        $$
        \Delta_T^H(c\otimes e):=\Delta_T^\calC(c)\otimes e+(-1)^{|c|}c\otimes d_{\calE}(e).
        $$ 
		\item For any tree $T\in \frakLT^{(n)}(m)$ with $n\geqslant 2$, and for any $c\in\calC(m), e\in\calE(m)$, the operation $\Delta_T^H$ is defined by 
        $$
        \Delta_T^H(c\otimes e):=(-1)^{\sum\limits_{k=1}^{n-1}\sum\limits_{j=k+1}^n|e_k||c_j|}(c_1\otimes e_1)\otimes \dots \otimes (c_n\otimes e_n)\in (\calC\otimes \calE )^{\otimes T},
        $$ 
        where $c_1\otimes \dots\otimes c_n=\Delta_T^\calC(c)\in \calC^{\otimes T}$  and $e_1\otimes \dots \otimes e_n=\Delta^T_\calE(e)\in \calE^{\otimes T}$, with $\Delta^T_\calE$ denoting the cocomposition along $T$ in $\calE$.
	\end{itemize}
	This new symmetric homotopy cooperad structure is called the \textbf{Hadamard product} of $\calC$ and $\calE$, and is denoted by $\calC\otimes_{\rmH}\calE$.
\end{prop-def}

Define $\calS=\mathrm{End}_{\bfk s}^c$ to be the symmetric graded cooperad whose underlying graded collection is given by $\calS(n)=\Hom((\bfk s)^{\otimes n},\bfk s)\cong\bfk \delta_n$ for $n\geqslant 1$, where $\delta_n$ is the map that sends $s^{\otimes n}$ to $s$. The $\mathbb S$-module structure of $\calS$ is given by 
$$
\delta_n\cdot\sigma:=\sgn(\sigma)\delta_n,
$$ 
for $n\geqslant 1$ and $\sigma\in\mathbb S_n$. The cooperad structure is defined by
$$
\Delta_T(\delta_n):=(-1)^{(j-1)(i-1)}\sgn(\sigma)\delta_{n-i+1}\otimes \delta_i\in (\calS)^{\otimes T}
$$ 
for any tree $T=\overline T\cdot\sigma\in\frakLT^{(2)}$ of the form
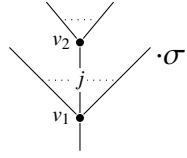
\begin{align*}
        \begin{tikzpicture}[scale=1,descr/.style={fill=white}]
			\tikzstyle{every node}=[thick,minimum size=3pt, inner sep=1pt]
            \node(r)at(0,-0.5){};
            \node(v1)at(0,0)[fill=black, circle,label=left:{\tiny$v_1$}]{};
            \node(v2)at(0,1)[fill=black, circle,label=left:{\tiny$v_2$}]{};
            \node(v1-1)at(-1,1){};\node(v1-2)at(1,1){};
            \node(v2-1)at(-0.5,1.6){};\node(v2-2)at(0.5,1.6){};
            \node(sigma)at(0.9,0.8)[minimum size=0pt, label=right:{$\cdot\sigma$}]{};
            \draw(r)--(v1);
            \draw(v1)--(v1-1);\draw(v1)--(v1-2);
            \draw(v2)--(v2-1);\draw(v2)--(v2-2);
            \draw[dotted](-0.5,0.5)--(0.5,0.5);
            \draw[dotted](-0.2,1.3)--(0.2,1.3);
            \path[-,font=\scriptsize] 
            (v1) edge node[descr]{{\tiny$j$}} (v2);
		\end{tikzpicture}
\end{align*}
where $\overline T$ is the trivially labelled tree with $\alpha(v_1)=n-i+1, \alpha(v_2)=i$, and $j$ denotes the incoming edge index of $v_2$ at $v_1$.

We also define $\calS^{-1}$ to be the graded symmetric cooperad whose underlying graded collection is given by $\calS^{-1}(n)=\Hom((\bfk s^{-1})^{\otimes n},s^{-1})\cong\bfk\varepsilon_n$ for $n\geqslant 1$, where $\varepsilon_n$ is the map that sends $(s^{-1})^{\otimes n}$ to $s^{-1}$. The $\mathbb S$-module structure of $\calS^{-1}$ is given by 
$$
\varepsilon_n\cdot\sigma:=\sgn(\sigma)\varepsilon_n,
$$ 
for $n\geqslant 1$ and $\sigma\in\mathbb S_n$. The cooperad structure is defined by 
$$
\Delta_T(\varepsilon_n)=(-1)^{(i-1)(n-i+1-j)}\sgn(\sigma)\varepsilon_{n-i+1}\otimes \varepsilon_i\in (\calS^{-1})^{\otimes T},
$$  
where $T$ is the same tree as pictured above.
It is easy to see that $\calS\otimes_{\rmH}\calS^{-1}\cong \calS^{-1}\otimes_{\rmH}\calS=:\mathbf{Com}^\vee$. Notice that for any symmetric homotopy cooperad $\calC$, we have $\calC\otimes_\rmH \mathbf{Com}^\vee\cong\calC\cong \mathbf{Com}^\vee\otimes_\rmH \calC.$

\begin{defn}
    Let $\calC$ be a symmetric homotopy cooperad. The \textbf{operadic suspension} (resp. \textbf{desuspension}) of $\calC$ is the symmetric homotopy cooperad defined by $\calC\otimes _{\rmH} \calS$ (resp. $\calC\otimes_{\rmH}\calS^{-1}$), and is denoted by $\mathscr{S}\calC$ (resp. $\mathscr{S}^{-1}\calC$).
\end{defn}

\begin{defn}
	Let $\calC=\{\calC(n)\}_{n\geqslant 1}$ be a coaugmented symmetric homotopy cooperad. The \textbf{cobar construction} of $\calC$, denoted by $\Omega\calC$, is the free graded symmetric operad generated by the graded $\mathbb S$-module $s^{-1}\overline{\calC}$, endowed with the differential $\partial$ induced by the lifting of the linear map $\partial: s^{-1}\overline{\calC}\to \Omega\calC$ given by
	$$
    \partial(s^{-1}f)=-\sum_{T\in \frakST(n)} (s^{-1})^{\otimes\omega(T)} \circ \overline\Delta_T(f),
    $$
    for any $f\in \overline{\calC}(n)$.
\end{defn}

This provides an alternative definition for symmetric homotopy operads. In fact, a graded $\mathbb S$-module $\overline{\calC}=\{\overline{\calC}(n)\}_{n\geqslant 1}$ carries a symmetric homotopy cooperad structure if and only if the free graded symmetric operad generated by $s^{-1}\overline{\calC}$ (also called the cobar construction of $\calC=\overline{C}\oplus \calI$) can be  endowed with a differential making it a symmetric dg operad.

Given a symmetric homotopy cooperad and a symmetric dg cooperad, one can construct a new symmetric homotopy operad. 
\begin{prop}\label{prop:Hom(C,P)-homotopy operad}
	Let $\calC$ be a symmetric homotopy cooperad and $\calP$ be a symmetric dg operad. Then the graded collection $\mathbf{Hom}(\calC,\calP)=\{\Hom(\calC(n), \calP(n))\}_{n\geqslant 1}$ has a natural symmetric homotopy operad structure defined as follows:
	\begin{itemize}
		\item For any $f\in\mathbf{\Hom}(\calC,\calP)(n)$, $\sigma\in\mathbb S_n$, and $x\in\calC(n)$, the right $\mathbb S_n$-action on $\mathbf{\Hom}(\calC,\calP)(n)$ is given by
        $$
        (f\cdot\sigma)(x):=f(x\cdot\sigma^{-1})\cdot\sigma.
        $$
        \item For any $T\cdot\sigma\in \frakLT^{(1)}(n)$ with $T$ being a trivially labelled tree, and for any $f\in\mathbf{\Hom}(\calC,\calP)(n)$, $c\in\calC(n)$, the operation $m_{T\cdot\sigma}$ is given by
        $$
        m_{T\cdot\sigma}(f)(c):=m_{T\cdot\sigma}^\calP\circ f(c\cdot\sigma^{-1})-(-1)^{|f|}\big(f\circ \Delta_{T\cdot\sigma}^\calC(c)\big)\cdot\sigma.
        $$
		\item For any $T\in \frakLT^{(n)}$ with $n\geqslant 2$, and for any $f_1\otimes\dots\otimes f_n\in\mathbf{\Hom}(\calC,\calP)^{\otimes T}$,  the operation $m_T$ is defined by
        $$
        m_T(f_1\otimes \dots \otimes f_n)=(-1)^{\frac{n(n-1)}{2}+1+n(\sum_{i=1}^n |f_i|)}m_{\calP}^T\circ(f_1\otimes \dots\otimes f_n)\circ \Delta_T^\calC,
        $$  
        where $m_\calP^T$ denotes the composition along $T$ in $\calP$.
	\end{itemize}
\end{prop}

\begin{prop}\label{prop:Linfinity give MC}
	Let $\calC$ be a coaugmented symmetric homotopy cooperad and $\calP$ be a unital symmetric dg operad. Then there is a natural bijection:
	$$
    \Hom_{udgOp}(\Omega\calC, \calP)\cong \calm\calC\left(\mathbf{Hom}(\overline{\calC},\calP)_{\mathbb S}^{\prod}\right),
    $$
	where the left-hand side is the set of morphisms of unital symmetric dg operads from $\Omega \calC$ to $\calP$ and the right-hand side is the set of Maurer-Cartan elements in the $L_\infty$-algebra $\mathbf{Hom}(\overline{\calC},\calP)_{\mathbb S}^{\prod}$.
\end{prop}

At the end of this section, we recall the notions of minimal models and Koszul dual symmetric homotopy cooperads of symmetric operads. We then explain how these concepts are related to deformation complexes and the resulting $L_\infty$-algebra structures that arise on them.

Recall that a symmetric dg operad is called \textbf{quasi-free} if its underlying graded symmetric operad is free.
\begin{defn} \label{def:minimal model} 
    A \textbf{minimal model} for a symmetric operad $\mathcal{P}$ is a quasi-free symmetric dg operad $(\mathcal{F}(M),\partial)$ together with a surjective quasi-isomorphism of symmetric operads $(\mathcal{F}(M), \partial)\overset{\sim}{\twoheadrightarrow}\mathcal{P}$, such that the symmetric dg operad $(\mathcal{F}(M), \partial)$ satisfies the following  minimality conditions:
	\begin{enumerate}
		\item \label{it:min1} The differential $\partial$ is decomposable, i.e., $\partial$ maps $M$ into $\mathcal{F}(M)^{\geqslant 2}$, the subspace of $\mathcal{F}(M)$ consisting of elements of tree weight $\geqslant 2$; 
		\item \label{it:min2} The generating $\mathbb S$-module $M$ admits a decomposition $M=\bigoplus_{i\geqslant 1}M_{(i)}$ such that $\partial(M_{(k+1)})\subset \mathcal{F}\left(\bigoplus_{i=1}^kM_{(i)}\right)$ for any $k\geqslant 1$. 
\end{enumerate}
\end{defn}

\begin{thm}[\cite{DCV13}]\label{thm:uniqueness of minimal model}
    If a symmetric operad $\mathcal{P}$ admits a minimal model,  then it is unique up to isomorphism.
\end{thm}

For a symmetric operad $\calP$, suppose that its minimal model $\calP_\infty $ exists. Since $\calP_\infty$ is a quasi-free symmetric dg operad, it can be realized as the cobar construction $\Omega \calC$ of a coaugmented symmetric homotopy cooperad $\calC$. We call $\calC$ the \textbf{Koszul dual symmetric homotopy cooperad} of $\calP$, and denote it by $\calP^\ac$.

Let $V$ be a chain complex, and let $\End_V$ denote the symmetric dg endomorphism operad of $V$.
The \textbf{deformation complex} of $\calP$ on $V$ is defined as the underlying complex of the $L_{\infty}$-algebra $\mathbf{Hom}(\overline{\calP^\ac}, \End_V)^{\prod}_{\mathbb S}$.
A \textbf{homotopy $\calP$-structure} on the space $V$ is defined to be a morphism in $\Hom_{udgOp}(\calP_\infty, \End_V)$.
By Proposition \ref{prop:Linfinity give MC}, there is a bijection between the set of homotopy $\calP$-structures on $V$ and the set of Maurer-Cartan elements in the $L_\infty$-algebra $\mathbf{Hom}(\overline{\calP^\ac}, \End_V)^{\prod}_{\mathbb S}$.
\begin{remark}\label{rmk: sym (co)operads to nonsym (co)operads}
    If we forget all $\mathbb S_n$-actions and the labels of the leaves of the trees, the theory of symmetric homotopy (co)operads reduces to that of non-symmetric homotopy (co)operads.
\end{remark}

\bigskip

\section{The minimal model of   Rota-Baxter Lie algebras with weight}\label{Sect: RB Lie}

\subsection{The Koszul dual homotopy cooperad for the operad of Rota-Baxter Lie algebras}
In this section, we will study the operad of Rota-Baxter Lie algebras. Explicitly, we will going to construct a symmetric homotopy cooperad which can be considered as the Koszul dual of the operad of Rota-Baxter Lie algebras, and we will prove that the cobar construction of this symmetric homotopy cooperad is exactly the minimal model for the operad of Rota-Baxter Lie algebras.

Firstly, let's recall some basics on Rota-Baxter Lie algebras.
	\begin{defn}\label{Def: Rota-Baxter Lie algebra}
		Let $(A, \ell=[,])$ be a Lie algebra over field $\bfk$ and
		$\lambda\in \bfk$. A linear
		operator $T: A\rightarrow A$ is said to be a Rota-Baxter operator of
		weight $\lambda$ if it satisfies
		\begin{eqnarray}\label{Eq: Rota-Baxter relation}
			[T(a), T(b)]=T\big([ a, T(b)] + [ T(a), b]+\lambda\  [a,
			b]\big)\end{eqnarray}	
		for any $a,b \in A$, or in terms of maps
		\begin{eqnarray}\label{Eq: Rota-Baxter relation in terms of maps}
			\ell\circ (T \ot T)=T\circ \ell \circ (\Id\ot T +T \ot \Id)+\lambda\ T\circ \ell.\end{eqnarray} In this case,  $(A,\ell,T)$ is called a Rota-Baxter Lie
		algebra of weight $\lambda$. Denote by $\rmRBLA$ the category of Rota-Baxter Lie algebras of weight $\lambda$ with obvious morphisms.
	\end{defn}

%
%
%

By the definition of Rota-Baxter Lie algebras, we can see that the  operad for Rota-Baxter Lie algebras of weight $\lambda$, denoted by $\RBLA$,  is the symmetric free operad generated by a unary operator $T$ and a binary operator $\ell$ with  $\ell(12)=-\ell$ modulo the operadic relation generated by
\begin{equation}\label{Eq: RBLA relation 1}
\begin{tikzpicture}[scale=0.5,descr/.style={fill=white},baseline=(current bounding box.center)]
\tikzstyle{every node}=[thick,minimum size=3pt, inner sep=1pt]
\node(r) at (0,-0.7)[minimum size=0pt,circle]{};
\node(v0) at (0,0)[fill=black, circle,label=left:{\tiny $\ell$}]{};
\node(v1-1) at (-1,1)[fill=black, circle,label=left:{\tiny $\ell$}]{};
\node(v1-2) at(1,1)[label=mid:{\tiny $ 3$}]{};
\node(v2-1) at (-2,2)[label=mid:{\tiny $ 1$}]{};
\node(v2-2) at (0,2)[label=mid:{\tiny $ 2$}]{};
\node(s1) at (2,0.5)[label=mid:{ $ +$}]{};
\node(r') at (5,-0.7)[minimum size=0pt,circle]{};
\node(v0') at (5,0)[fill=black, circle,label=left:{\tiny $\ell$}]{};
\node(v1-1') at (4,1)[fill=black, circle,label=left:{\tiny $\ell$}]{};
\node(v1-2') at(6,1)[label=mid:{\tiny $ 1$}]{};
\node(v2-1') at (3,2)[label=mid:{\tiny $ 2$}]{};
\node(v2-2') at (5,2)[label=mid:{\tiny $ 3$}]{};
\draw(r) -- (v0);
\draw(v0)--(v1-1);
\draw(v0)--(v1-2);
\draw(v1-1)--(v2-1);
\draw(v1-1)--(v2-2);
\draw(r') -- (v0');
\draw(v0')--(v1-1');
\draw(v0')--(v1-2');
\draw(v1-1')--(v2-1');
\draw(v1-1')--(v2-2');
\node(s2) at (7,0.5)[label=mid:{ $ +$}]{};
\node(r'') at (10,-0.7)[minimum size=0pt,circle]{};
\node(v0'') at (10,0)[fill=black, circle,label=left:{\tiny $\ell$}]{};
\node(v1-1'') at (9,1)[fill=black, circle,label=left:{\tiny $\ell$}]{};
\node(v1-2'') at(11,1)[label=mid:{\tiny $ 2$}]{};
\node(v2-1'') at (8,2)[label=mid:{\tiny $ 3$}]{};
\node(v2-2'') at (10,2)[label=mid:{\tiny $ 1$}]{};
\draw(r'') -- (v0'');
\draw(v0'')--(v1-1'');
\draw(v0'')--(v1-2'');
\draw(v1-1'')--(v2-1'');
\draw(v1-1'')--(v2-2'');
		\end{tikzpicture}\end{equation}
and
\begin{equation}\label{Eq: RBLA relation 2}
\begin{tikzpicture}[scale=0.5,descr/.style={fill=white},baseline=(current bounding box.center)]
\tikzstyle{every node}=[thick,minimum size=3pt, inner sep=1pt]
\node(r) at (0,-0.7)[minimum size=0pt,circle]{};
\node(v0) at (0,0)[fill=black, circle,label=left:{\tiny $\ell$}]{};
\node(v1-1) at (-1,1)[fill=black, circle,label=left:{\tiny $T$}]{};
\node(v1-2) at (1,1)[fill=black, circle,label=left:{\tiny $T$}]{};
\node(v2-1) at (-1,2)[label=mid:{\tiny $ 1$}]{};
\node(v2-2) at (1,2)[label=mid:{\tiny $ 2$}]{};
\draw(r) -- (v0);
\draw(v0) -- (v1-1);
\draw(v0) -- (v1-2);
\draw(v1-1)--(v2-1);
\draw(v1-2)--(v2-2);
\node(s1) at (2,0.5)[label=mid:{\tiny $ -$}]{};
\node(r') at (4,-0.7)[minimum size=0pt,circle]{};
\node(v0') at (4,0)[fill=black, circle,label=left:{\tiny $T$}]{};
\node(v1') at (4,1)[fill=black, circle,label=left:{\tiny $\ell$}]{};
\node(v2-1') at (3,2)[fill=black, circle,label=left:{\tiny $T$}]{};
\node(v2-2') at (5,2)[label=mid:{\tiny $ 2$}]{};
\node(v3') at (3,3)[label=mid:{\tiny $ 1$}]{};
\draw(r') -- (v0');
\draw(v0') -- (v1');
\draw(v1')--(v2-1');
\draw(v1')--(v2-2');
\draw(v2-1')--(v3');
\node(s2) at (6,0.5)[label=mid:{\tiny $ -$}]{};
\node(r'') at (8,-0.7)[minimum size=0pt,circle]{};
\node(v0'') at (8,0)[fill=black, circle,label=left:{\tiny $T$}]{};
\node(v1'') at (8,1)[fill=black, circle,label=left:{\tiny $\ell$}]{};
\node(v2-2'') at (7,2)[label=mid:{\tiny $ 1$}]{};
\node(v2-1'') at (9,2)[fill=black, circle,label=left:{\tiny $T$}]{};
\node(v3'') at (9,3)[label=mid:{\tiny $ 2$}]{};
\draw(r'') -- (v0'');
\draw(v0'') -- (v1'');
\draw(v1'')--(v2-1'');
\draw(v1'')--(v2-2'');
\draw(v2-1'')--(v3'');
\node(s3) at (10,0.5)[label=mid:{\tiny $ -$ \large$\lambda$}]{};
\node(r''') at (12,-0.7)[minimum size=0pt,circle]{};
\node(v0''') at (12,0)[fill=black, circle,label=left:{\tiny $T$}]{};
\node(v1''') at (12,1)[fill=black, circle,label=left:{\tiny $\ell$}]{};
\node(v2-2''') at (11,2)[label=mid:{\tiny $ 1$}]{};
\node(v3''') at (13,2)[label=mid:{\tiny $ 2$}]{};
\draw(r''') -- (v0''');
\draw(v0''') -- (v1''');
\draw(v1''')--(v3''');
\draw(v1''')--(v2-2''');
\node(stop) at (13.5,0) {$.$};
\end{tikzpicture}
\end{equation}


It can be easily seen that the symmetric operad $\RBLA$ is not a Koszul operad, even not a quadratic operad. Usually, for a Koszul operad $\calp$, the minimal model of $\calp$ is exactly $\Omega \calp^{\antish}$, the cobar construction of its Koszul dual cooperad $\calp^\antish$. But since  the operad $\RBLA$ is not Koszul, its Koszul dual $(\RBLA^{\lambda})^\antish$ should be a symmetric homotopy cooperad, rather than a symmetric cooperad.

\smallskip

Now, let's construct the symmetric homotopy cooperad $(\RBLA^{\lambda})^\antish$ explicitly and we will prove that its cobar construction is exactly the minimal model for the symmetric operad $\RBLA$.

 Firstly, define $\mathscr{S}(\RBLA^\antish)$ to be the $\mathbb{S}$-module,  whose $n$-arity component is $\mathscr{S}(\RBLA^\antish)(n)=\bfk u_n\oplus \bfk v_n,$ $|u_n|=0$ and $ |v_n|=1,$ with trivial $\s_n$-action. Then we are going to construct a coaugmented symmetric homotopy cooperad structure on the $\mathbb{S}$-module  $\mathscr{S}(\RBLA^\antish)$. Consider the subset of $\frakt$ consisting of trees in the following list:
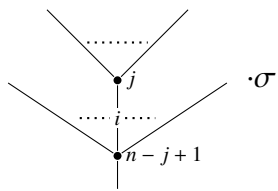
\begin{itemize}[keyvals]
	
	\item[(I)]
	
	\begin{eqnarray*}
		\begin{tikzpicture}[scale=1,descr/.style={fill=white}]
			\tikzstyle{every node}=[thick,minimum size=3pt, inner sep=1pt]
			\node(r) at (0,-0.5)[minimum size=0pt,circle]{};
			\node(v0) at (0,0)[fill=black, circle,label=right:{\tiny $ n-j+1$}]{};
			\node(v1-1) at (-1.5,1)[minimum size=0pt,circle]{};
			\node(v1-2) at(0,1)[fill=black,circle,label=right:{\tiny $\tiny j$}]{};
			\node(v1-3) at(1.5,1)[minimum size=0pt,circle]{};
			\node(v2-1)at (-1,2){};
			\node(v2-2) at(1,2){};
            \node(v3-1) at (1.9,1){$\cdot \sigma$};
            \draw(r) -- (v0);
			\draw(v0)--(v1-1);
			\draw(v0)--(v1-3);
			\draw(v1-2)--(v2-1);
			\draw(v1-2)--(v2-2);
			\draw[thick,dotted](-0.4,1.5)--(0.4,1.5);
			\draw[thick,dotted](-0.5,0.5)--(-0.1,0.5);
			\draw[thick,dotted](0.1,0.5)--(0.5,0.5);
			\path[-,font=\scriptsize]
			(v0) edge node[descr]{{\tiny$i$}} (v1-2);
		\end{tikzpicture}
	\end{eqnarray*}
	with $1\leq j \leq n,  1\leq i\leq   n-j+1$ and $\sigma \in \s_n.$
	\item[(II)]
	
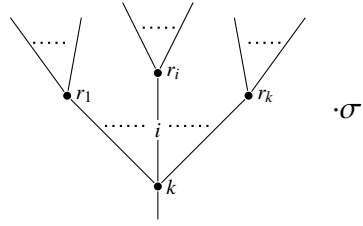
\begin{eqnarray*}
		\begin{tikzpicture}[scale=1,descr/.style={fill=white}]
			\tikzstyle{every node}=[thick,minimum size=3pt, inner sep=1pt]
            \node(r) at (0,-2) [minimum size=0pt,circle]{};
			\node(v0) at (0,-1.5)[circle, fill=black,label=right:\tiny$k$]{};
			\node(v1) at(-1.2,-0.3)[circle, fill=black,label=right:\tiny$r_1$]{};
			\node(v1-1) at(-2,0.8){};
			\node(v1-2) at (-1,0.8){};
			\node(v3) at (1.2,-0.3)[circle, fill=black,label=right:\tiny$r_k$]{};
			\node(v3-1) at (1,0.8){};
			\node(v3-2) at (2,0.8){};
			\node(v2-1) at (0,0) [circle,fill=black,label=right:\tiny$r_i$]{};
			\node(v2-1-1) at (-0.5,1){};
			\node(v2-1-2) at (0.5, 1){};
            \node(v4-1) at (2.5,-0.5){$\cdot \sigma$};
            \draw(r) -- (v0);
			\draw [thick,dotted ] (-0.7,-0.7)--(-0.15,-0.7);
			\draw [thick,dotted ] (0.15,-0.7)--(0.7,-0.7);
			\draw [thick,dotted ] (-0.2,0.6)--(0.25,0.6);
			\draw [thick,dotted ] (-1.65,0.4)--(-1.2, 0.4);
			\draw [thick,dotted ] (1.2,0.4)--(1.65,0.4);
			\draw        (v0)--(v1);
			\draw         (v0)--(v3);
			\draw         (v1)--(v1-1);
			\draw          (v1)--(v1-2);
			\draw        (v2-1)--(v2-1-1);
			\draw        (v2-1)--(v2-1-2);
			\draw        (v3)--(v3-1);
			\draw        (v3)--(v3-2);
			\path[-,font=\scriptsize]
			(v0) edge node[descr]{{\tiny$i$}} (v2-1);
		\end{tikzpicture}
	\end{eqnarray*}
	with $k\geqslant 2$, $r_i\geqslant 1$ for all $1\leq i\leq k ,$ $n=r_1+\dots +r_k$ and $\sigma\in \s_n.$
	\item[(III)]
	
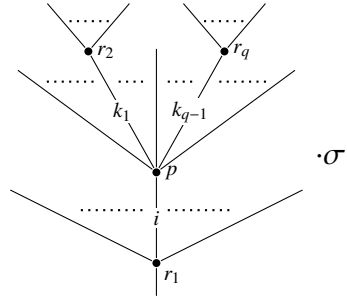
\begin{eqnarray*}
		\begin{tikzpicture}[scale=1,descr/.style={fill=white}]
			\tikzstyle{every node}=[thick,minimum size=3pt, inner sep=1pt]
            \node(r) at (0,-0.5){};
			\node(v0) at (0,0)[circle,inner sep=0.5pt,fill=black,label=below right :{\tiny $r_1$}]{};
			\node(v1-1) at (-2,1){};
			\node(v1-2) at(0,1.2)[circle,fill=black,label=right:{\tiny $p$}]{};
			\node(v1-3) at(2,1){};
			\node(v2-1) at(-1.9,2.6){};
			\node(v2-2) at (-0.9, 2.8)[circle,fill=black,label=right:{\tiny $r_2$}]{};
			\node(v2-3) at (0,2.9){};
			\node(v2-4) at(0.9,2.8)[circle,fill=black,label=right:{\tiny $r_q$}]{};
			\node(v2-5) at(1.9,2.6){};
			\node(v3-1) at (-1.5,3.5){};
			\node(v3-2) at (-0.3,3.5){};
			\node(v3-3) at (0.3,3.5){};
			\node(v3-4) at(1.5,3.5){};
\draw(r)--(v0);
            \node(v4-1) at (2.3, 1.4){$\cdot \sigma$};
			\draw(v0)--(v1-1);
			\draw(v0)--(v1-3);
			\path[-,font=\scriptsize]
			(v0) edge node[descr]{{\tiny$i$}} (v1-2);
			\draw(v1-2)--(v2-1);
			\draw(v1-2)--(v2-3);
			\draw(v1-2)--(v2-5);
			\path[-,font=\scriptsize]
			(v1-2) edge node[descr]{{\tiny$k_1$}} (v2-2)
			edge node[descr]{{\tiny$k_{q-1}$}} (v2-4);
			\draw(v2-2)--(v3-1);
			\draw(v2-2)--(v3-2);
			\draw(v2-4)--(v3-3);
			\draw(v2-4)--(v3-4);
			\draw[thick,dotted](-1,0.7)--(-0.15,0.7);
			\draw[thick,dotted](0.15,0.7)--(1,0.7);
			\draw[thick,dotted](-0.5,2.4)--(-0.15,2.4);
			\draw[thick,dotted](0.15,2.4)--(0.5,2.4);
			\draw[thick,dotted](-1.45,2.4)--(-0.8,2.4);
			\draw[thick,dotted](1.45,2.4)--(0.8,2.4);
			\draw[thick,dotted](-1.15,3.2)--(-0.6,3.2);
			\draw[thick,dotted](1.15,3.2)--(0.6,3.2);
		\end{tikzpicture}
	\end{eqnarray*}
	with $1\leqslant q\leqslant p$, $1\leqslant k_1<\dots<k_{q-1}\leqslant p$, $1\leqslant i\leqslant r_1,$  $r_j\geqslant 1$ for all $1\leqslant j\leqslant q,$ $n=r_1+\cdots+r_q+ p-q$ and $\sigma \in \s_n$.
\end{itemize}

Now, we define a family of operations $\{\Delta_T: \mathscr{S}(\RBLA^\antish)(\alpha(T)) \rightarrow (\mathscr{S}(\RBLA^\antish))^{\ot T}\}_{T\in \frakt}$ as follows:

\begin{itemize}
	\item[(i)] for element $u_n\in\mathscr{S}(\RBLA^\antish)(n)$ and $T$ of type $\mathrm{(I)}$, define
	\begin{eqnarray*}
		\begin{tikzpicture}[scale=1,descr/.style={fill=white}]
			\tikzstyle{every node}=[thick,minimum size=5pt, inner sep=1pt]
			\node(r) at (0,-0.5)[minimum size=0pt,rectangle]{};
			\node(v-2) at(-2,0.5)[minimum size=0pt, label=left:{\Large$\Delta_T(u_n)=$}]{};
			\node(v0) at (0,0)[draw,rectangle]{{\small $u_{n-j+1}$}};
			\node(v1-1) at (-1.5,1){};
			\node(v1-2) at(0,1)[draw,rectangle]{\small$u_j$};
			\node(v1-3) at(1.5,1){};
			\node(v2-1)at (-1,2){};
			\node(v2-2) at(1,2){};
\node(v3-1) at (2,0.5){$\cdot\sigma$};
\draw(r) --(v0);
			\draw(v0)--(v1-1);
			\draw(v0)--(v1-3);
			\draw(v1-2)--(v2-1);
			\draw(v1-2)--(v2-2);
			\draw[thick,dotted](-0.4,1.5)--(0.45,1.5);
			\draw[thick,dotted](-0.55,0.5)--(-0.1,0.5);
			\draw[thick,dotted](0.1,0.5)--(0.55,0.5);
			\path[-,font=\scriptsize]
			(v0) edge node[descr]{{\tiny$i$}} (v1-2);
		\end{tikzpicture}
	\end{eqnarray*}
	\item[(ii)] for the element $v_n\in\mathscr{S}(\RBLA^\antish)(n)$ and tree $T$ of type $\mathrm{(I)}$ with $j=1$, define
	\begin{eqnarray*}
		\begin{tikzpicture}[scale=1,descr/.style={fill=white}]
			\tikzstyle{every node}=[thick,minimum size=5pt, inner sep=1pt]
			\node(r) at (0,-0.5)[minimum size=0pt,rectangle]{};
			\node(v-1) at(-2,0.5)[minimum size=0pt, label=left:{\Large$\Delta_T(v_n)=$}]{};
			\node(v0) at (0,0)[draw,rectangle]{\small$v_n$};
			\node(v1-1) at (-1.3,1){};
			\node(v1-2) at(0,1)[draw,rectangle]{\small$u_1$};
			\node(v1-3) at(1.3,1){};
			\node(v2-1)at (0,1.8){};
\node(v3-1) at (1.8,0.5){$\cdot \sigma$};
\draw(r) --(v0);
			\draw(v0)--(v1-1);
			\draw(v0)--(v1-3);
			\draw(v1-2)--(v2-1);
			\draw[thick,dotted](-0.55,0.5)--(-0.1,0.5);
			\draw[thick,dotted](0.1,0.5)--(0.55,0.5);
			\path[-,font=\scriptsize]
			(v0) edge node[descr]{{\tiny$i$}} (v1-2);
		\end{tikzpicture}
	\end{eqnarray*}
	for tree $T$ of type $\mathrm{(I)}$ with $2\leqslant j\leqslant n-1$, define
	\begin{eqnarray*}
		\begin{tikzpicture}[scale=1,descr/.style={fill=white}]
			\tikzstyle{every node}=[thick,minimum size=5pt, inner sep=1pt]
			\node(r) at (0,-0.5)[minimum size=0pt,rectangle]{};
			\node(v-2) at(-3,0.5)[minimum size=0pt, label=left:{\Large$\Delta_T(v_n)=$}]{};
			\node(v-1) at(-1.5,0.5)[minimum size=0pt,label=left:{\Large$\lambda^{j-1}$}]{};
			\node(v0) at (0,0)[draw, rectangle]{$v_{n-j+1}$};
			\node(v1-1) at (-1.5,1){};
			\node(v1-2) at(0,1)[draw,rectangle]{\small $u_j$};
			\node(v1-3) at(1.5,1){};
			\node(v2-1)at (-1,2){};
			\node(v2-2) at(1,2){};
\node(v3-1) at (2,0.5){$\cdot \sigma$};
\draw(r) --(v0);
			\draw(v0)--(v1-1);
			\draw(v0)--(v1-3);
			\draw(v1-2)--(v2-1);
			\draw(v1-2)--(v2-2);
			\draw[thick,dotted](-0.4,1.5)--(0.45,1.5);
			\draw[thick,dotted](-0.55,0.5)--(-0.1,0.5);
			\draw[thick,dotted](0.1,0.5)--(0.55,0.5);
			\path[-,font=\scriptsize]
			(v0) edge node[descr]{{\tiny$i$}} (v1-2);
		\end{tikzpicture}
	\end{eqnarray*}
	and for tree $T$ of type $(\mathrm{I})$ with $j=n$, define
	\begin{eqnarray*}
		\begin{tikzpicture}[scale=1,descr/.style={fill=white}]
			\tikzstyle{every node}=[thick,minimum size=5pt, inner sep=1pt]
			\node(r) at (0,-0.7)[minimum size=0pt,rectangle]{};
			\node(va) at(-3,0.5)[minimum size=0pt, label=left:{\Large$\Delta_T(v_n)=$}]{};
			\node(vb) at(-1.5,0.5)[minimum size=0pt,label=left:{\Large$\lambda^{n-1}$}]{};
\begin{scope}[yshift=-0.2cm]
			\node(vc) at (0,0)[draw, rectangle]{\small $v_1$};
			\node(v1) at(0,1)[draw,rectangle]{\small $u_n$};
			\node(v2-1)at (-1,2){};
			\node(v2-2) at(1,2){};
\node(3-1) at  (1.3,0.7){$\cdot \sigma$};
\end{scope}
\begin{scope}[xshift=1cm,yshift=-0.2cm]
            \node(r') at (3.5,-0.5)[minimum size=0pt,rectangle]{};
			\node(vd) at(1.5,0.8)[minimum size=0, label=right:$+$]{};
			\node(ve) at (3.5,0)[draw, rectangle]{\small $u_1$};
			\node(ve1) at (3.5,1)[draw,rectangle]{\small $v_n$};
			\node(ve2-1) at(2.5,2){};
			\node(ve2-2) at(4.5,2){};
			\draw[thick,dotted](3.1,1.5)--(3.95,1.5);
            \node(ve3-1) at (4.8, 0.7){$\cdot\sigma$};
\end{scope}
            \draw(r) --(vc);
            \draw(r') --(ve);
			\draw(v1)--(v2-1);
			\draw(v1)--(v2-2);
			\draw[thick,dotted](-0.4,1.5)--(0.45,1.5);
			\draw(vc)--(v1);
			\draw(ve)--(ve1);
			\draw(ve1)--(ve2-1);
			\draw(ve1)--(ve2-2);
		\end{tikzpicture}
	\end{eqnarray*}
	
	\item[(iii)] for tree $T$ of type $\mathrm{(II)}$ with $r_1+\dots+r_k=n$, define
	
	\begin{eqnarray*}
		\begin{tikzpicture}[scale=1,descr/.style={fill=white}]
			\tikzstyle{every node}=[thick,minimum size=5pt, inner sep=1pt]
\node(r) at (0,-2){};
			\node(v-2) at (-4,-0.5)[minimum size=0pt, label=left:{\Large$\Delta_T(v_n)=$}]{};
			\node(v-1) at(-4,-0.5)[minimum size=0pt,label=right:{\Large$(-1)^{\frac{k(k-1)}{2}}$}]{};
			\node(v0) at (0,-1.5)[rectangle, draw]{\small $u_k$};
			\node(v1) at(-1.2,-0.3)[rectangle,draw]{\small $v_{r_1}$};
			\node(v1-1) at(-2,0.8){};
			\node(v1-2) at (-1,0.8){};
			\node(v3) at (1.2,-0.3)[rectangle, draw]{\small $v_{r_k}$};
			\node(v3-1) at (1,0.8){};
			\node(v3-2) at (2,0.8){};
			\node(v2-1) at (0,0) [rectangle,draw]{\small $v_{r_i}$};
			\node(v2-1-1) at (-0.5,1){};
			\node(v2-1-2) at (0.5, 1){};
\node(v4-1) at (2.5,-0.5){$\cdot\sigma$};
\draw(r)--(v0);
			\draw [thick,dotted ] (-0.75,-0.5)--(-0.2,-0.5);
			\draw [thick,dotted ] (0.3,-0.5)--(0.85,-0.5);
			\draw [thick,dotted ] (-0.2,0.5)--(0.25,0.5);
			\draw [thick,dotted ] (-1.65,0.5)--(-1.2, 0.5);
			\draw [thick,dotted ] (1.2,0.5)--(1.65,0.5);
			\draw        (v0)--(v1);
			\draw         (v0)--(v3);
			\draw         (v1)--(v1-1);
			\draw          (v1)--(v1-2);
			\draw        (v2-1)--(v2-1-1);
			\draw        (v2-1)--(v2-1-2);
			\draw        (v3)--(v3-1);
			\draw        (v3)--(v3-2);
			\path[-,font=\scriptsize]
			(v0) edge node[descr]{{\tiny$i$}} (v2-1);
		\end{tikzpicture}
	\end{eqnarray*}
	\item[(iv)] for tree $T$ of type $\mathrm{(III)}$ with $r_1+\dots+r_q+p-q=n$, define
	\begin{eqnarray*}
		\begin{tikzpicture}[scale=1,descr/.style={fill=white}]
			\tikzstyle{every node}=[minimum size=4pt, inner sep=1pt]
\node(r) at (0,-0.5){};
			\node(v-2) at (-6,1.2)[minimum size=0pt, label=left:{\Large$\Delta_T(v_n)=$}]{};
			\node(v-1) at(-6,1.3)[minimum size=0pt,label=right:{\Large$(-1)^\frac{q(q-1)}{2}\lambda^{p-q}$}]{};
			\node(v0) at (0,0)[rectangle,draw]{\small $v_{r_1}$};
			\node(v1-1) at (-2,1){};
			\node(v1-2) at(0,1.2)[rectangle,draw]{\small $u_p$};
			\node(v1-3) at(2,1){};
			\node(v2-1) at(-1.9,2.6){};
			\node(v2-2) at (-0.9, 2.8)[rectangle,draw]{\small$v_{r_2}$};
			\node(v2-3) at (0,2.9){};
			\node(v2-4) at(0.9,2.8)[rectangle,draw]{\small $v_{r_q}$};
			\node(v2-5) at(1.9,2.6){};
			\node(v3-1) at (-1.5,3.5){};
			\node(v3-2) at (-0.3,3.5){};
			\node(v3-3) at (0.3,3.5){};
			\node(v3-4) at(1.5,3.5){};
\node(v4-1) at (2.5,1.2){$\cdot\sigma$};
\draw(r)--(v0);
			\draw(v0)--(v1-1);
			\draw(v0)--(v1-3);
			\path[-,font=\scriptsize]
			(v0) edge node[descr]{{\tiny$i$}} (v1-2);
			\draw(v1-2)--(v2-1);
			\draw(v1-2)--(v2-3);
			\draw(v1-2)--(v2-5);
			\path[-,font=\scriptsize]
			(v1-2) edge node[descr]{{\tiny$k_1$}} (v2-2)
			edge node[descr]{{\tiny$k_{q-1}$}} (v2-4);
			\draw(v2-2)--(v3-1);
			\draw(v2-2)--(v3-2);
			\draw(v2-4)--(v3-3);
			\draw(v2-4)--(v3-4);
			\draw[thick,dotted](-1,0.7)--(-0.15,0.7);
			\draw[thick,dotted](0.15,0.7)--(1,0.7);
			\draw[thick,dotted](-0.5,2.4)--(-0.15,2.4);
			\draw[thick,dotted](0.15,2.4)--(0.5,2.4);
			\draw[thick,dotted](-1.45,2.4)--(-0.8,2.4);
			\draw[thick,dotted](1.45,2.4)--(0.8,2.4);
			\draw[thick,dotted](-1.1,3.2)--(-0.65,3.2);
			\draw[thick,dotted](1.1,3.2)--(0.65,3.2);
		\end{tikzpicture}
	\end{eqnarray*}
	\item[(v)] all other components of $\Delta_T, T\in \frakt$ vanish.
\end{itemize}

\begin{prop}\label{Prop: homotopy cooperad structure}
The $\mathbb{S}$-module $\mathscr{S}(\RBLA^\antish)$ endowed with operations $\{\Delta_T\}_{T\in \frakt}$ introduced above forms a coaugmented symmetric homotopy cooperad, whose strict counit is the natural projection $\varepsilon:\mathscr{S}(\RBLA^\antish) \twoheadrightarrow \bfk u_1\cong \cali$ and the coaugmentation is just the natural embedding $\eta:\cali\cong \bfk u_1\hookrightarrow \mathscr{S}(\RBLA^\antish)$. 	
\end{prop}

\begin{proof}
 One needs to show that the induced derivation $\partial$ on the cobar construction of $\mathscr{S}(\RBLA^\antish),$ i.e., the free symmetric operad generated by $s^{-1}\overline{\mathscr{S}(\RBLA^\antish)},$ is a differential, that is, $\partial^2=0$.

Denote $s^{-1}u_n, n\geqslant 2$ (resp.  $s^{-1}v_n, n\geqslant 1$) by $x_n$ (resp. $y_n$) which are the generators of $\Omega(  \mathscr{S}(\RBLA^\antish) ).$
Notice that $|x_n|=-1$ and $|y_n|=0$.
By the definition of cobar construction of coaugmented symmetric homotopy cooperads, the action of differential $\partial$ on generators $x_n, y_n$ is given by the following formulas: for $n\geqslant 2,$
\begin{eqnarray}\label{Eq: trees of differential of Lie bracket}
 		\hspace{-2cm} \begin{tikzpicture}[scale=0.4,baseline=(current bounding box.center) ]
 			\tikzstyle{every node}=[minimum size=3pt, inner sep=1pt]
 			\node(a) at (-2,0){\begin{Large}$\partial \ \ \ $\end{Large}};
 \node(r) at (0,-2){};
 			\node(b0)at (0,-1)[rectangle,draw]{\tiny{$x_n$}};
 			\node (b1) at (-2,1)  [minimum size=0pt,label=above:{\tiny{$1$}}]{};
 			\node (b2) at (0,1)  [minimum size=0pt,label=above:{}]{};
 			\node (b3) at (2,1)  [minimum size=0pt,label=above:{\tiny{$n$}}]{};
 \draw(r) -- (b0);
 			\draw        (b0)--(b1);
 			\draw        (b0)--(b2);
 			\draw        (b0)--(b3);
 			\draw [dotted,thick] (-1.05,0.25)--(-0.1,0.25);
  			\draw [dotted,thick] (1.05,0.25)--(0.1,0.25);
 		\end{tikzpicture}
 		&&
 		\begin{tikzpicture}[scale=0.4,descr/.style={fill=white},baseline=(current bounding box.center) ]
 			\tikzstyle{every node}=[minimum size=5pt, inner sep=1pt]
 			\node(a) at (0,0.25)[minimum size=0pt, label=right:{\large $=-\sum\limits_{\substack{2\leqslant j\leqslant n-1 \\ 1\leqslant i\leqslant n-j+1}} \sum\limits_{\sigma}$}]{};
 \node(r) at (10,-2){};
 			\node(r1) at(10,-1)[rectangle,draw]{\tiny{$x_{n-j+1}$}};
 			\node(r2) at(10,1.5)[rectangle,draw]{\tiny{$x_{j}$}};
 \node  at(13, 0.5){\large $\cdot \sigma$};
 \draw(r) -- (r1);
 			\draw(r1)--(8,1);
 			\draw (r1)--(r2);
 			\draw(r1)--(12,1);
 			\draw(r2)--(8.5,3);
 			\path[-,font=\scriptsize]
 			(r1) edge node[descr]{{\tiny{$i$}}} (r2);
 			\draw(r2)--(11.5,3);
 			\draw [dotted,thick] (9.1,0.25)--(9.8,0.25);
  			\draw [dotted,thick] (10.2,0.25)--(10.9,0.25);
 			\draw [dotted,thick] (9.15,2.5)--(10.85,2.5);
 		\end{tikzpicture}
 \end{eqnarray}
and for $n\geqslant 1,$
 \begin{eqnarray}	\nonumber	\label{Eq: trees of differential of RB operator} 
	\begin{tikzpicture}[scale=0.4,baseline=(current bounding box.center)]
 			\tikzstyle{every node}=[minimum size=3pt, inner sep=1pt]
 			\node(a) at (-2,0){\begin{Large}$\partial \ \ \ $\end{Large}};
 \node(r) at (0,-2){};
 			\node(b0)at (0,-1)[rectangle,draw]{\tiny{$y_n$}};
 			\node (b1) at (-2,1)  [minimum size=0pt,label=above:{\tiny{$1$}}]{};
 			\node (b2) at (0,1)  [minimum size=0pt,label=above:{}]{};
 			\node (b3) at (2,1)  [minimum size=0pt,label=above:{\tiny{$n$}}]{};
 \draw(r)--(b0);
 			\draw        (b0)--(b1);
 			\draw        (b0)--(b2);
 			\draw        (b0)--(b3);
 			\draw [dotted,thick] (-1.15,0.25)--(-0.05,0.25);
  			\draw [dotted,thick] (1.15,0.25)--(0.05,0.25);
 		\end{tikzpicture}&&
\hspace{-0.5cm} \begin{tikzpicture}[scale=0.5,descr/.style={fill=white},baseline=(current bounding box.center)]
\tikzstyle{every node}=[minimum size=3pt, inner sep=1pt]
 			\node(b) at(13,0.5)[minimum size=0pt,label=right:{\large$\ \ \ \ =-\sum\limits_{k=2}^{n}\sum\limits_{\substack{r_1+\cdots+r_k=n \\ r_1, \cdots, r_k \geqslant 1}} \sum\limits_{\sigma}$}]{};
 \node(r) at (25,-2){};
 			\node(l1) at(25,-1)[rectangle, draw]{\tiny{$x_k$}};
 			\node(l21) at(22,1.5)[rectangle, draw]{\tiny{$y_{r_1}$}};
 			\node(l22) at(24,1.5)[rectangle,draw]{\tiny{$y_{r_2}$}};
 			\node(l23) at(28,1.5)[rectangle,draw]{\tiny{$y_{r_k}$}};
 \node(sgn) at (30,0.5){\large$\cdot \sigma$};
 			\draw[dotted,thick](24.8,0.2)--(26.1,0.2);
 \draw(r) -- (l1);
 			\draw(l21)--(21,2.5);
 			\draw(l21)--(22.5,2.5);
 			\draw(l22)--(23,2.5);
 			\draw(l22)--(24.5,2.5);
 			\draw(l23)--(27,2.5);
 			\draw(l23)--(29,2.5);
 			\draw(l1)--(l21);
 			\draw(l1)--(l22);
 			\draw(l1)--(l23);
 			\draw[dotted,thick] (21.5,2.25)--(22.2,2.25);
 			\draw[dotted,thick] (23.5,2.25)--(24.2,2.25);
 			\draw[dotted,thick] (27.5,2.25)--(28.6,2.25);
 		\end{tikzpicture}
	\end{eqnarray}
 $$\hspace{4cm}	\begin{tikzpicture}[scale=0.5,descr/.style={fill=white},baseline=(current bounding box.center)]
\tikzstyle{every node}=[minimum size=3pt, inner sep=1pt]
 			\node(b) at(-11,0)[minimum size=0pt,label=right:{\large$\ \ \ \  +\sum\limits_{\substack{2\leqslant p \leqslant n\\ 1 \leqslant q\leqslant p}}\sum\limits_{\substack{r_1+\cdots+r_q + p -q =n \\ r_1, \cdots, r_q \geqslant 1\\ 1\leqslant i \leqslant r_1\\  1\leqslant k_1 < \cdots < k_{q-1} \leqslant p}} \sum\limits_{\sigma} \lambda^{p-q}$}]{};
 \node(r) at (5,-2){};
 \node(a0) at (5,-1)[rectangle, draw]{\tiny{$y_{r_1}$}};
 \node(a1-1) at (1.5,1){};
 \node(a1-2) at (5,1)[rectangle, draw]{\tiny{$x_p$}};
 \node(a1-3) at (8.5,1){};
 \node(a2-1) at (1.5,3){};
 \node(a2-2) at (3, 3.5)[rectangle, draw]{\tiny $y_{r_2}$};
  \node(a2-3) at (5, 3){};
  \node(a2-4) at (7, 3.5)[rectangle, draw]{\tiny $y_{r_q}$};
    \node(a2-5) at (8.5, 3){};
    \node(a3-1) at (2,4.5){};
    \node(a3-2) at (4,4.5){};
        \node(a3-3) at (6,4.5){};
    \node(a3-4) at (8,4.5){};
\draw(r)--(a0);
\draw(a0)--(a1-1);
\draw(a0)--(a1-2);
\draw(a0)--(a1-3);
\draw(a1-2)--(a2-1);
\draw(a1-2)--(a2-2);
\draw(a1-2)--(a2-3);
\draw(a1-2)--(a2-4);
\draw(a1-2)--(a2-5);
\draw(a2-2) --(a3-1);
 \draw(a2-2) --(a3-2);
 \draw(a2-4) --(a3-3);
 \draw(a2-4) --(a3-4);
 \node(sgn) at (10,0.5){\large$\cdot \sigma$};
 			\draw[dotted,thick](3.5,0)--(4.85,0);
  			\draw[dotted,thick](5.15,0)--(6.5,0);
    		\draw[dotted,thick](3.1,2.25)--(3.6,2.25);
    	    \draw[dotted,thick](4.4,2.25)--(4.9,2.25);
            \draw[dotted,thick](5.6,2.25)--(5.1,2.25);
            \draw[dotted,thick](6.4,2.25)--(6.9,2.25);
            \draw[dotted,thick] (2.55,4.2)--(3.45,4.2);
            \draw[dotted,thick] (6.55,4.2)--(7.45,4.2);
 	\path[-,font=\scriptsize]
 			(a0) edge node[descr]{{\tiny{$i$}}} (a1-2);
  	\path[-,font=\scriptsize]
 			(a1-2) edge node[descr]{{\tiny{$k_1$}}} (a2-2);
    \path[-,font=\scriptsize]
 			(a1-2) edge node[descr]{{\tiny{$k_{q-1}$}}} (a2-4);
 		\end{tikzpicture}$$
where all  $\sum_\sigma$ denote the summations over permutations  $\sigma \in \s_n$ that make the tree monomials into shuffle tree monomials.

We just need to see that $\partial^{2}$ holds on generators $x_n, n\geqslant 2$ and $y_n, n\geqslant 1$ which can be checked by direct computations.
Firstly, note that the suboperad in $\Omega(  \mathscr{S}(\RBLA^\antish) )$ generated by the  $\s$-module $\{x_n\}_{n\geqslant 2}$ is exactly the operadic suspension of the dg operad that governs $L_{\infty}$-algebras \cite{LM95}. %
Equation \eqref{Eq: trees of differential of Lie bracket} precisely represents the formulas for the differential on generators in this  symmetric dg  operad. Therefore, the equation $\partial^2(x_n) = 0$ holds  true.

We also have
 \begin{eqnarray*}		
&& \begin{tikzpicture}[scale=0.5,baseline=(current bounding box.center)]
 			\tikzstyle{every node}=[minimum size=2pt, inner sep=0.8pt]
 			\node(a) at (-2,0){\begin{Large}$\partial^2 \ \ \ $\end{Large}};
 \node(r) at (0,-2){};
 			\node(b0)at (0,-1)[rectangle,draw]{\tiny{$y_n$}};
 			\node (b1) at (-2,1)  [minimum size=0pt,label=above:{\tiny{$1$}}]{};
 			\node (b2) at (0,1)  [minimum size=0pt,label=above:{}]{};
 			\node (b3) at (2,1)  [minimum size=0pt,label=above:{\tiny{$n$}}]{};
 \draw(r)--(b0);
 			\draw        (b0)--(b1);
 			\draw        (b0)--(b2);
 			\draw        (b0)--(b3);
 			\draw [dotted,thick] (-1.15,0.25)--(-0.05,0.25);
  			\draw [dotted,thick] (1.15,0.25)--(0.05,0.25);
 		\end{tikzpicture}\\
 &&	 \begin{tikzpicture}[scale=0.5,descr/.style={fill=white},baseline=(b)]
\tikzstyle{every node}=[minimum size=3pt, inner sep=0.8pt]
 			\node(b) at(16,0)[minimum size=0pt,label=right:{$ =\partial ( -\sum\limits_{\substack{2\leqslant k\leqslant n \\ r_1+\cdots+r_k=n \\ r_1, \cdots, r_k \geqslant 1}} \sum\limits_{\sigma}$}]{};
 \node(r) at (25,-2){};
 			\node(l1) at(25,-1)[rectangle, draw]{\tiny{$x_k$}};
 			\node(l21) at(23,0.5)[rectangle, draw]{\tiny{$y_{r_1}$}};
 			\node(l22) at(24,0.5)[rectangle,draw]{\tiny{$y_{r_2}$}};
 			\node(l23) at(27,0.5)[rectangle,draw]{\tiny{$y_{r_k}$}};
 \node(sgn) at (28,0.5){$\cdot \sigma$};
 			\draw[dotted,thick](24.7,-0.2)--(26,-0.2);
 \draw(r) -- (l1);
 			\draw(l21)--(22.2,1.5);
 			\draw(l21)--(23.2,1.5);
 			\draw(l22)--(23.7,1.5);
 			\draw(l22)--(24.7,1.5);
 			\draw(l23)--(26.5,1.5);
 			\draw(l23)--(27.5,1.5);
 			\draw(l1)--(l21);
 			\draw(l1)--(l22);
 			\draw(l1)--(l23);
 			\draw[dotted,thick] (22.5,1.25)--(23,1.25);
 			\draw[dotted,thick] (23.9,1.25)--(24.4,1.25);
 			\draw[dotted,thick] (26.8,1.25)--(27.3,1.25);
 		\end{tikzpicture} \hspace{0.3cm}
 	\begin{tikzpicture}[scale=0.5,descr/.style={fill=white},baseline=(b)]
\tikzstyle{every node}=[minimum size=3pt, inner sep=0.8pt]
 			\node(b) at(-7,-0.2)[minimum size=0pt,label=right:{$ +\sum\limits_{\substack{2\leqslant p \leqslant n\\ 1\leqslant q\leqslant p}}\sum\limits_{\substack{r_1+\cdots+r_q + p -q =n \\ r_1, \cdots, r_q \geqslant 1\\ 1\leqslant i \leqslant r_1\\  1\leqslant k_1 < \cdots < k_{q-1} \leqslant p}} \sum\limits_{\sigma} \lambda^{p-q}$}]{};
 \node(r) at (5,-2){};
 \node(a0) at (5,-1)[rectangle, draw]{\tiny{$y_{r_1}$}};
 \node(a1-1) at (2.5,0.5){};
 \node(a1-2) at (5,0.5)[rectangle, draw]{\tiny{$x_p$}};
 \node(a1-3) at (7.5,0.5){};
 \node(a2-1) at (2,2){};
 \node(a2-2) at (3.5, 2.5)[rectangle, draw]{\tiny $y_{r_2}$};
  \node(a2-3) at (5, 2){};
  \node(a2-4) at (6.65, 2.5)[rectangle, draw]{\tiny $y_{r_q}$};
    \node(a2-5) at (8, 2){};
    \node(a3-1) at (2.5,3.5){};
    \node(a3-2) at (4,3.5){};
        \node(a3-3) at (6,3.5){};
    \node(a3-4) at (7.5,3.5){};
\draw(r)--(a0);
\draw(a0)--(a1-1);
\draw(a0)--(a1-2);
\draw(a0)--(a1-3);
\draw(a1-2)--(a2-1);
\draw(a1-2)--(a2-2);
\draw(a1-2)--(a2-3);
\draw(a1-2)--(a2-4);
\draw(a1-2)--(a2-5);
\draw(a2-2) --(a3-1);
 \draw(a2-2) --(a3-2);
 \draw(a2-4) --(a3-3);
 \draw(a2-4) --(a3-4);
 \node(sgn) at (9,0.6){$\cdot \sigma )$};
 			\draw[dotted,thick](4,-0.25)--(4.85,-0.25);
  			\draw[dotted,thick](5.15,-0.25)--(6,-0.25);
    		\draw[dotted,thick](3.3,1.5)--(3.6,1.5);
    	    \draw[dotted,thick](4.4,1.5)--(4.9,1.5);
            \draw[dotted,thick](5.6,1.5)--(5.1,1.5);
            \draw[dotted,thick](6.4,1.5)--(6.9,1.5);
            \draw[dotted,thick] (3.1,3.1)--(3.55,3.1);
            \draw[dotted,thick] (6.45,3.1)--(7,3.1);
 	\path[-,font=\scriptsize]
 			(a0) edge node[descr]{{\tiny{$i$}}} (a1-2);
  	\path[-,font=\scriptsize]
 			(a1-2) edge node[descr]{{\tiny{$k_1$}}} (a2-2);
    \path[-,font=\scriptsize]
 			(a1-2) edge node[descr]{{\tiny{$k_{q-1}$}}} (a2-4);
 		\end{tikzpicture}\\
 &&	 \begin{tikzpicture}[scale=0.5,descr/.style={fill=white},baseline=(b)]
\tikzstyle{every node}=[minimum size=3pt, inner sep=0.8pt]
 			\node(b) at(16,0)[minimum size=0pt,label=right:{$ = -\sum\limits_{\substack{2\leqslant k\leqslant n \\ r_1+\cdots+r_k=n \\ r_1, \cdots, r_k \geqslant 1}} \sum\limits_{\sigma}$}]{};
 \node(r) at (25,-2){};
 			\node(l1) at(25,-1)[rectangle, draw]{\tiny{$\partial x_k$}};
 			\node(l21) at(23,0.5)[rectangle, draw]{\tiny{$y_{r_1}$}};
 			\node(l22) at(24,0.5)[rectangle,draw]{\tiny{$y_{r_2}$}};
 			\node(l23) at(27,0.5)[rectangle,draw]{\tiny{$y_{r_k}$}};
 \node(sgn) at (28,0.5){$\cdot \sigma$};
 			\draw[dotted,thick](24.7,-0.2)--(26,-0.2);
 \draw(r) -- (l1);
 			\draw(l21)--(22.2,1.5);
 			\draw(l21)--(23.2,1.5);
 			\draw(l22)--(23.7,1.5);
 			\draw(l22)--(24.7,1.5);
 			\draw(l23)--(26.5,1.5);
 			\draw(l23)--(27.5,1.5);
 			\draw(l1)--(l21);
 			\draw(l1)--(l22);
 			\draw(l1)--(l23);
 			\draw[dotted,thick] (22.5,1.25)--(23,1.25);
 			\draw[dotted,thick] (23.9,1.25)--(24.4,1.25);
 			\draw[dotted,thick] (26.8,1.25)--(27.3,1.25);
 		\end{tikzpicture}
\hspace{0.5cm} \begin{tikzpicture}[scale=0.5,descr/.style={fill=white},baseline=(b)]
\tikzstyle{every node}=[minimum size=3pt, inner sep=0.8pt]
 			\node(b) at(15,0)[minimum size=0pt,label=right:{$  + \sum\limits_{\substack{2\leqslant k\leqslant n \\ r_1+\cdots+r_k=n \\ r_1, \cdots, r_k \geqslant 1}} \sum\limits_{l=1}^{k}\sum\limits_{\sigma}$}]{};
 \node(r) at (25,-2){};
 			\node(l1) at(25,-1)[rectangle, draw]{\tiny{$ x_k$}};
 			\node(l21) at(23,0.5)[rectangle, draw]{\tiny{$y_{r_1}$}};
 			\node(l22) at(25,0.5)[rectangle,draw]{\tiny{$\partial y_{r_l}$}};
 			\node(l23) at(27,0.5)[rectangle,draw]{\tiny{$y_{r_k}$}};
 \node(sgn) at (28,0.5){$\cdot \sigma$};
 			\draw[dotted,thick](24.1,-0.2)--(24.85,-0.2);
  			\draw[dotted,thick](25.2,-0.2)--(25.85,-0.2);
 \draw(r) -- (l1);
 			\draw(l21)--(22.2,1.5);
 			\draw(l21)--(23.2,1.5);
 			\draw(l22)--(24.5,1.5);
 			\draw(l22)--(25.55,1.5);
 			\draw(l23)--(26.5,1.5);
 			\draw(l23)--(27.5,1.5);
 			\draw(l1)--(l21);
 			\draw(l1)--(l22);
 			\draw(l1)--(l23);
 			\draw[dotted,thick] (22.5,1.25)--(23,1.25);
 			\draw[dotted,thick] (24.7,1.25)--(25.2,1.25);
 			\draw[dotted,thick] (26.8,1.25)--(27.3,1.25);
 	\path[-,font=\scriptsize]
 			(l1) edge node[descr]{{\tiny{$i$}}} (l22);
 		\end{tikzpicture} \\
 && 	\begin{tikzpicture}[scale=0.5,descr/.style={fill=white},baseline=(b)]
\tikzstyle{every node}=[minimum size=3pt, inner sep=0.8pt]
 			\node(b) at(-7,-0.2)[minimum size=0pt,label=right:{$ +\sum\limits_{\substack{2\leqslant p \leqslant n\\ 1\leqslant q\leqslant p}}\sum\limits_{\substack{r_1+\cdots+r_q + p -q =n \\ r_1, \cdots, r_q \geqslant 1\\ 1\leqslant i \leqslant r_1\\  1\leqslant k_1 < \cdots < k_{q-1} \leqslant p}} \sum\limits_{\sigma} \lambda^{p-q}$}]{};
 \node(r) at (5,-2){};
 \node(a0) at (5,-1)[rectangle, draw]{\tiny{$\partial y_{r_1}$}};
 \node(a1-1) at (2.5,0.5){};
 \node(a1-2) at (5,0.5)[rectangle, draw]{\tiny{$x_p$}};
 \node(a1-3) at (7.5,0.5){};
 \node(a2-1) at (2,2){};
 \node(a2-2) at (3.5, 2.5)[rectangle, draw]{\tiny $y_{r_2}$};
  \node(a2-3) at (5, 2){};
  \node(a2-4) at (6.65, 2.5)[rectangle, draw]{\tiny $y_{r_q}$};
    \node(a2-5) at (8, 2){};
    \node(a3-1) at (2.5,3.5){};
    \node(a3-2) at (4,3.5){};
        \node(a3-3) at (6,3.5){};
    \node(a3-4) at (7.5,3.5){};
\draw(r)--(a0);
\draw(a0)--(a1-1);
\draw(a0)--(a1-2);
\draw(a0)--(a1-3);
\draw(a1-2)--(a2-1);
\draw(a1-2)--(a2-2);
\draw(a1-2)--(a2-3);
\draw(a1-2)--(a2-4);
\draw(a1-2)--(a2-5);
\draw(a2-2) --(a3-1);
 \draw(a2-2) --(a3-2);
 \draw(a2-4) --(a3-3);
 \draw(a2-4) --(a3-4);
 \node(sgn) at (8,0.6){$\cdot \sigma $};
 			\draw[dotted,thick](4,-0.25)--(4.85,-0.25);
  			\draw[dotted,thick](5.15,-0.25)--(6,-0.25);
    		\draw[dotted,thick](3.3,1.5)--(3.6,1.5);
    	    \draw[dotted,thick](4.4,1.5)--(4.9,1.5);
            \draw[dotted,thick](5.6,1.5)--(5.1,1.5);
            \draw[dotted,thick](6.4,1.5)--(6.9,1.5);
            \draw[dotted,thick] (3.1,3.1)--(3.55,3.1);
            \draw[dotted,thick] (6.45,3.1)--(7,3.1);
 	\path[-,font=\scriptsize]
 			(a0) edge node[descr]{{\tiny{$i$}}} (a1-2);
  	\path[-,font=\scriptsize]
 			(a1-2) edge node[descr]{{\tiny{$k_1$}}} (a2-2);
    \path[-,font=\scriptsize]
 			(a1-2) edge node[descr]{{\tiny{$k_{q-1}$}}} (a2-4);
 		\end{tikzpicture}
\hspace{0.3cm} \begin{tikzpicture}[scale=0.5,descr/.style={fill=white},baseline=(b)]
\tikzstyle{every node}=[minimum size=3pt, inner sep=0.8pt]
 			\node(b) at(-7,-0.2)[minimum size=0pt,label=right:{$ +\sum\limits_{\substack{2\leqslant p \leqslant n\\ 1\leqslant q\leqslant p}}\sum\limits_{\substack{r_1+\cdots+r_q + p -q =n \\ r_1, \cdots, r_q \geqslant 1\\ 1\leqslant i \leqslant r_1\\  1\leqslant k_1 < \cdots < k_{q-1} \leqslant p}} \sum\limits_{\sigma} \lambda^{p-q}$}]{};
 \node(r) at (5,-2){};
 \node(a0) at (5,-1)[rectangle, draw]{\tiny{$ y_{r_1}$}};
 \node(a1-1) at (2.5,0.5){};
 \node(a1-2) at (5,0.5)[rectangle, draw]{\tiny{$\partial x_p$}};
 \node(a1-3) at (7.5,0.5){};
 \node(a2-1) at (2,2){};
 \node(a2-2) at (3.5, 2.5)[rectangle, draw]{\tiny $y_{r_2}$};
  \node(a2-3) at (5, 2){};
  \node(a2-4) at (6.65, 2.5)[rectangle, draw]{\tiny $y_{r_q}$};
    \node(a2-5) at (8, 2){};
    \node(a3-1) at (2.5,3.5){};
    \node(a3-2) at (4,3.5){};
        \node(a3-3) at (6,3.5){};
    \node(a3-4) at (7.5,3.5){};
\draw(r)--(a0);
\draw(a0)--(a1-1);
\draw(a0)--(a1-2);
\draw(a0)--(a1-3);
\draw(a1-2)--(a2-1);
\draw(a1-2)--(a2-2);
\draw(a1-2)--(a2-3);
\draw(a1-2)--(a2-4);
\draw(a1-2)--(a2-5);
\draw(a2-2) --(a3-1);
 \draw(a2-2) --(a3-2);
 \draw(a2-4) --(a3-3);
 \draw(a2-4) --(a3-4);
 \node(sgn) at (8,0.6){$\cdot \sigma $};
 			\draw[dotted,thick](4,-0.25)--(4.85,-0.25);
  			\draw[dotted,thick](5.15,-0.25)--(6,-0.25);
    		\draw[dotted,thick](3.3,1.5)--(3.6,1.5);
    	    \draw[dotted,thick](4.4,1.5)--(4.9,1.5);
            \draw[dotted,thick](5.6,1.5)--(5.1,1.5);
            \draw[dotted,thick](6.4,1.5)--(6.9,1.5);
            \draw[dotted,thick] (3.1,3.1)--(3.55,3.1);
            \draw[dotted,thick] (6.45,3.1)--(7,3.1);
 	\path[-,font=\scriptsize]
 			(a0) edge node[descr]{{\tiny{$i$}}} (a1-2);
  	\path[-,font=\scriptsize]
 			(a1-2) edge node[descr]{{\tiny{$k_1$}}} (a2-2);
    \path[-,font=\scriptsize]
 			(a1-2) edge node[descr]{{\tiny{$k_{q-1}$}}} (a2-4);
 		\end{tikzpicture}\\
 && 	\begin{tikzpicture}[scale=0.5,descr/.style={fill=white},baseline=(b)]
\tikzstyle{every node}=[minimum size=3pt, inner sep=0.8pt]
 			\node(b) at(-8,-0.2)[minimum size=0pt,label=right:{$ - \sum\limits_{\substack{2\leqslant p \leqslant n\\ 2\leqslant q\leqslant p\\ 2\leqslant l \leqslant q}}\sum\limits_{\substack{r_1+\cdots+r_q + p -q =n \\ r_1, \cdots, r_q \geqslant 1\\ 1\leqslant i \leqslant r_1\\  1\leqslant k_1 < \cdots < k_{q-1} \leqslant p}} \sum\limits_{\sigma} \lambda^{p-q}$}]{};
 \node(r) at (5,-2){};
 \node(a0) at (5,-1)[rectangle, draw]{\tiny{$y_{r_1}$}};
 \node(a1-1) at (1.5,0.5){};
 \node(a1-2) at (5,0.5)[rectangle, draw]{\tiny{$x_p$}};
 \node(a1-3) at (8.5,0.5){};
 \node(a2-1) at (1,2){};
 \node(a2-2) at (2.5, 2.5)[rectangle, draw]{\tiny $y_{r_2}$};
 \node(b2-1) at (4.2,2.1){};
 \node(b2-2) at (5.8,2.1){};
  \node(a2-3) at (5, 2.5)[rectangle, draw]{\tiny $\partial y_{r_l}$};
  \node(a2-4) at (7.5, 2.5)[rectangle, draw]{\tiny $y_{r_q}$};
    \node(a2-5) at (9, 2){};
    \node(a3-1) at (1.5,3.5){};
    \node(a3-2) at (3.5,3.5){};
        \node(a3-3) at (6.5,3.5){};
    \node(a3-4) at (8.5,3.5){};
    \node(b3-1) at (4,3.5){};
    \node(b3-2) at (6,3.5){};
\draw(r)--(a0);
\draw(a0)--(a1-1);
\draw(a0)--(a1-2);
\draw(a0)--(a1-3);
\draw(a1-2)--(a2-1);
\draw(a1-2)--(a2-2);
\draw(a1-2)--(a2-3);
\draw(a1-2)--(a2-4);
\draw(a1-2)--(a2-5);
\draw(a2-2) --(a3-1);
 \draw(a2-2) --(a3-2);
 \draw(a2-4) --(a3-3);
 \draw(a2-4) --(a3-4);
 \draw(b2-1) -- (a1-2);
 \draw(b2-2) -- (a1-2);
 \draw(b3-1) -- (a2-3);
 \draw(b3-2) -- (a2-3);
 \node(sgn) at (9,0.6){$\cdot \sigma $};
 			\draw[dotted,thick](3.6,-0.25)--(4.85,-0.25);
  			\draw[dotted,thick](5.15,-0.25)--(6.4,-0.25);
    		\draw[dotted,thick](2.9,1.5)--(3.4,1.5);
    	    \draw[dotted,thick](3.85,1.5)--(5.7,1.5);
            \draw[dotted,thick](6.4,1.5)--(7.3,1.5);
            \draw[dotted,thick] (2.1,3.1)--(2.95,3.1);
            \draw[dotted,thick] (4.6,3.1)--(5.45,3.1);
            \draw[dotted,thick] (7.1,3.1)--(7.95,3.1);
 	\path[-,font=\scriptsize]
 			(a0) edge node[descr]{{\tiny{$i$}}} (a1-2);
  	\path[-,font=\scriptsize]
 			(a1-2) edge node[descr]{{\tiny{$k_1$}}} (a2-2);
    \path[-,font=\scriptsize]
 			(a1-2) edge node[descr]{{\tiny{$k_{q-1}$}}} (a2-4);
 		\end{tikzpicture}\\
 &&	 \begin{tikzpicture}[scale=0.5,descr/.style={fill=white},baseline=(b)]
\tikzstyle{every node}=[minimum size=3pt, inner sep=0.8pt]
 			\node(b) at(14,0)[minimum size=0pt,label=right:{$ = \sum\limits_{\substack{2\leqslant k\leqslant n \\ r_1+\cdots+r_k=n \\ r_1, \cdots, r_k \geqslant 1}} \sum\limits_{\substack{2\leqslant s \leqslant k-1 \\ 1\leqslant t \leqslant k-s-1}}\sum\limits_{\sigma}$}]{};
 \node(r) at (25,-2){};
 			\node(l1) at(25,-1)[rectangle, draw]{\tiny{$x_{k-s-1}$}};
            \node(l2') at (25,0.5)[rectangle, draw]{\tiny{$ x_{s}$}};
 			\node(l21) at(23,0.5)[rectangle, draw]{\tiny{$y_{r_1}$}};
 			\node(l22) at(24,2)[rectangle,draw]{\tiny{$y_{r_t}$}};
            \node(l22') at (25.7,2)[rectangle,draw]{\tiny{$y_{r_{t+s+1}}$}};
 			\node(l23) at(27,0.5)[rectangle,draw]{\tiny{$y_{r_k}$}};
 \node(sgn) at (28,0.5){$\cdot \sigma$};
 \draw[dotted,thick](24.1,-0.2)--(24.75,-0.2);
  \draw[dotted,thick](25.25,-0.2)--(25.9,-0.2);
 \draw(r) -- (l1);
 			\draw(l21)--(22.2,1.5);
 			\draw(l21)--(23.2,1.5);
 			\draw(l22)--(23.3,3);
 			\draw(l22)--(24.5,3);
            \draw(l22') -- ( 25.3,3);
            \draw(l22') -- ( 26.3,3);
 			\draw(l23)--(26.7,1.5);
 			\draw(l23)--(27.7,1.5);
 			\draw(l1)--(l21);
 			\draw(l1)--(l2');
 			\draw(l1)--(l23);
            \draw(l2') -- (l22);
            \draw(l2') -- (l22');
 			\draw[dotted,thick] (22.5,1.25)--(23,1.25);
 			\draw[dotted,thick] (24.7,1.25)--(25.25,1.25);
 			\draw[dotted,thick] (26.9,1.25)--(27.4,1.25);
  			\draw[dotted,thick] (25.6,2.75)--(26.1,2.75);
    			\draw[dotted,thick] (23.6,2.75)--(24.1,2.75);
  \path[-,font=\scriptsize]
 			(l1) edge node[descr]{{\tiny{$t$}}} (l2');
 		\end{tikzpicture}
 \begin{tikzpicture}[scale=0.5,descr/.style={fill=white},baseline=(b)]
\tikzstyle{every node}=[minimum size=3pt, inner sep=0.8pt]
 			\node(b) at(14,0)[minimum size=0pt,label=right:{$ - \sum\limits_{\substack{2\leqslant k\leqslant n \\ r_1+\cdots+r_k=n \\ r_1, \cdots, r_k \geqslant 1}} \sum\limits_{\substack{1\leqslant l \leqslant k\\ 2\leqslant u \leqslant r_l \\ v_1 + \cdots v_u =r_l
\\ v_1, \cdots, v_u \geqslant 1}}\sum\limits_{\sigma}$}]{};
 \node(r) at (25,-2){};
 			\node(l1) at(25,-1)[rectangle, draw]{\tiny{$x_{k}$}};
            \node(l2') at (25,0.5)[rectangle, draw]{\tiny{$ x_{u}$}};
 			\node(l21) at(22,0.5)[rectangle, draw]{\tiny{$y_{r_1}$}};
 \node(l21') at (23.9,0.5)[rectangle, draw]{\tiny{$y_{r_{l-1}}$}};
 			\node(l22) at(24,2)[rectangle,draw]{\tiny{$y_{v_1}$}};
            \node(l22') at (25.7,2)[rectangle,draw]{\tiny{$y_{v_{u}}$}};
 			\node(l23) at(28,0.5)[rectangle,draw]{\tiny{$y_{r_k}$}};
 \node(l23') at (26.1,0.5)[rectangle,draw]{\tiny{$y_{r_{l+1}}$}};
 \node(sgn) at (29,0.5){$\cdot \sigma$};
 \draw[dotted,thick](23.7,-0.2)--(24.15,-0.2);
  \draw[dotted,thick](25.85,-0.2)--(26.4,-0.2);
 \draw(r) -- (l1);
 			\draw(l21)--(21.2,1.5);
 			\draw(l21)--(22.2,1.5);
 \draw(l21') --(23.1,1.5);
  \draw(l21') --(24,1.5);
 			\draw(l22)--(23.3,3);
 			\draw(l22)--(24.5,3);
            \draw(l22') -- ( 25.3,3);
            \draw(l22') -- ( 26.3,3);
 			\draw(l23)--(27.7,1.5);
 			\draw(l23)--(28.7,1.5);
 \draw(l23') -- (25.9,1.5);
  \draw(l23') -- (26.9,1.5);
   \draw[dotted,thick](26.05,1.25)--(26.5,1.25);
  \draw[dotted,thick](27.85,1.25)--(28.3,1.25);
 			\draw(l1)--(l21);
 			\draw(l1)--(l2');
 			\draw(l1)--(l23);
            \draw(l2') -- (l22);
            \draw(l2') -- (l22');
            \draw(l21') -- (l1);
            \draw(l23')--(l1);
 			\draw[dotted,thick] (21.6,1.25)--(22.1,1.25);
  			\draw[dotted,thick] (23.4,1.25)--(23.9,1.25);
 			\draw[dotted,thick] (24.7,1.25)--(25.25,1.25);
  			\draw[dotted,thick] (25.6,2.75)--(26.1,2.75);
    			\draw[dotted,thick] (23.6,2.75)--(24.1,2.75);
  \path[-,font=\scriptsize]
 			(l1) edge node[descr]{{\tiny{$l$}}} (l2');
 		\end{tikzpicture} \\
 &&  \begin{tikzpicture}[scale=0.5,descr/.style={fill=white},baseline=(b)]
\tikzstyle{every node}=[minimum size=3pt, inner sep=0.8pt]
 			\node(b) at(9,0)[minimum size=0pt,label=right:{$ + \sum\limits_{\substack{2\leqslant k\leqslant n \\ r_1+\cdots+r_k=n \\ r_1, \cdots, r_k \geqslant 1}} \sum\limits_{\substack{1\leqslant l \leqslant k\\ 2\leqslant s \leqslant r_l \\ 1\leqslant t \leqslant s}} \sum\limits_{\substack{ v_1 + \cdots + v_t  + s - t = r_l\\ v_1, \cdots,  v_t \geqslant 1 \\ 1\leqslant i \leqslant v_1 \\ 1\leqslant j_1 < \cdots < j_{t-1} \leqslant s }} \sum\limits_{\sigma} \lambda^{s-t}$}]{};
 \node(r) at (25,-2){};
 			\node(l1) at(25,-1)[rectangle, draw]{\tiny{$x_{k}$}};
            \node(l2') at (25,0.5)[rectangle, draw]{\tiny{$ y_{v_1}$}};
 			\node(l21) at(22,0.5)[rectangle, draw]{\tiny{$y_{r_1}$}};
 \node(l21') at (23.9,0.5)[rectangle, draw]{\tiny{$y_{r_{l-1}}$}};
 			\node(l22) at(23.9,2){};
            \node(l22') at (25.8,2){};
            \node(l22'') at (24.95,2.25)[rectangle, draw]{\tiny{$x_{s}$}};
            \node(l31) at (22,3.5){};
            \node(l32) at (23.5,4)[rectangle, draw]{\tiny{$y_{v_2}$}};
            \node(l33) at (24.95, 3.5){};
            \node(l34) at (26.4,4)[rectangle, draw]{\tiny{$y_{v_{t}}$}};
            \node(l35) at (27.9,3.5){};
            \node(l41) at (23,5){};
            \node(l42) at (24,5){};
            \node(l43) at (25.9,5){};
            \node(l44) at (26.9,5){};
 			\node(l23) at(28,0.5)[rectangle,draw]{\tiny{$y_{r_k}$}};
 \node(l23') at (26.1,0.5)[rectangle,draw]{\tiny{$y_{r_{l+1}}$}};
 \node(sgn) at (29,0.5){$\cdot \sigma$};
 \draw[dotted,thick](23.7,-0.2)--(24.15,-0.2);
  \draw[dotted,thick](25.85,-0.2)--(26.4,-0.2);
 \draw(r) -- (l1);
 			\draw(l21)--(21.2,1.5);
 			\draw(l21)--(22.2,1.5);
 \draw(l21') --(23.1,1.5);
  \draw(l21') --(24,1.5);
 			\draw(l23)--(27.7,1.5);
 			\draw(l23)--(28.7,1.5);
 \draw(l23') -- (25.9,1.5);
  \draw(l23') -- (26.9,1.5);
   \draw[dotted,thick](26.05,1.25)--(26.5,1.25);
  \draw[dotted,thick](27.85,1.25)--(28.3,1.25);
  \draw(l2')--(l22'');
  \draw(l22'')--(l31);
  \draw(l22'')--(l32);
  \draw(l22'')--(l33);
  \draw(l22'')--(l34);
  \draw(l22'')--(l35);
  \draw(l32)--(l41);
  \draw(l32)--(l42);
  \draw(l34)--(l43);
  \draw(l34)--(l44);
 			\draw(l1)--(l21);
 			\draw(l1)--(l2');
 			\draw(l1)--(l23);
            \draw(l2') -- (l22);
            \draw(l2') -- (l22');
            \draw(l21') -- (l1);
            \draw(l23')--(l1);
 			\draw[dotted,thick] (21.6,1.25)--(22.1,1.25);
  			\draw[dotted,thick] (23.4,1.25)--(23.9,1.25);
 			\draw[dotted,thick] (24.5,1.35)--(24.75,1.35);
  			\draw[dotted,thick] (25.15,1.35)--(25.4,1.35);
  \draw[dotted,thick] (23.4,3.1)--(23.9,3.1);
  \draw[dotted,thick] (24.2,3.1)--(25.95,3.1);
  \draw[dotted,thick] (26.65,3.1)--(26.15,3.1);
  \draw[dotted,thick](23.35,4.6) -- (23.65,4.6);
    \draw[dotted,thick](26.55,4.6) -- (26.25,4.6);
  \path[-,font=\scriptsize]
 			(l1) edge node[descr]{{\tiny{$l$}}} (l2');
   \path[-,font=\scriptsize]
 			(l2') edge node[descr]{{\tiny{$i$}}} (l22'');
    \path[-,font=\scriptsize]
 			(l22'') edge node[descr]{{\tiny{$j_1$}}} (l32);
     \path[-,font=\scriptsize]
 			(l22'') edge node[descr]{{\tiny{$j_{t-1}$}}} (l34);
 		\end{tikzpicture} \\
 &&  \begin{tikzpicture}[scale=0.5,descr/.style={fill=white},baseline=(b)]
\tikzstyle{every node}=[minimum size=3pt, inner sep=0.8pt]
 			\node(b) at(-11,-0.2)[minimum size=0pt,label=right:{$ - \sum\limits_{\substack{2\leqslant p \leqslant n\\ 1\leqslant q\leqslant p}}\sum\limits_{\substack{r_1+\cdots+r_q + p -q =n \\ r_1, \cdots, r_q \geqslant 1\\ 1\leqslant i \leqslant r_1\\  1\leqslant k_1 < \cdots < k_{q-1} \leqslant p}} \sum\limits_{\substack{ 2\leqslant s \leqslant r_1\\ v_1+\cdots + v_s = r_1\\ v_1, \cdots, v_s \geqslant 1 \\ 1\leqslant j \leqslant s \\ 1\leqslant l \leqslant v_j }}\sum\limits_{\sigma} \lambda^{p-q}$}]{};
 \node(r) at (5,-3.5){};
 \node(a-1) at (5,-2.5)[rectangle, draw]{\tiny{$x_s$}};
 \node(a0-1) at (2,-1)[rectangle, draw]{\tiny{$ y_{v_1}$}};
  \node(a0-2) at (8,-1)[rectangle, draw]{\tiny{$ y_{v_s}$}};
 \node(a0) at (5,-1)[rectangle, draw]{\tiny{$ y_{v_j}$}};
 \node(a1-4') at (1.3,0){};
 \node(a1-5') at (2.2,0){};
  \node(a1-6') at (7.8,0){};
 \node(a1-7') at (8.7,0){};
 \node(a1-1) at (2.5,0.5){};
 \node(a1-2) at (5,0.5)[rectangle, draw]{\tiny{$x_p$}};
 \node(a1-3) at (7.5,0.5){};
 \node(a2-1) at (2,2){};
 \node(a2-2) at (3.5, 2.5)[rectangle, draw]{\tiny $y_{r_2}$};
  \node(a2-3) at (5, 2){};
  \node(a2-4) at (6.65, 2.5)[rectangle, draw]{\tiny $y_{r_q}$};
    \node(a2-5) at (8, 2){};
    \node(a3-1) at (2.5,3.5){};
    \node(a3-2) at (4,3.5){};
        \node(a3-3) at (6,3.5){};
    \node(a3-4) at (7.5,3.5){};
\draw(r)--(a-1);
\draw(a-1)--(a0-1);
\draw(a-1)--(a0-2);
\draw(a-1)--(a0);
\draw(a1-4') -- (a0-1);
\draw(a1-5') -- (a0-1);
\draw(a1-6') -- (a0-2);
\draw(a1-7') -- (a0-2);
\draw(a0)--(a1-1);
\draw(a0)--(a1-2);
\draw(a0)--(a1-3);
\draw(a1-2)--(a2-1);
\draw(a1-2)--(a2-2);
\draw(a1-2)--(a2-3);
\draw(a1-2)--(a2-4);
\draw(a1-2)--(a2-5);
\draw(a2-2) --(a3-1);
 \draw(a2-2) --(a3-2);
 \draw(a2-4) --(a3-3);
 \draw(a2-4) --(a3-4);
 \node(sgn) at (9,0.6){$\cdot \sigma $};
 			\draw[dotted,thick](4,-0.25)--(4.85,-0.25);
  			\draw[dotted,thick](5.15,-0.25)--(6,-0.25);
    		\draw[dotted,thick](3.3,1.5)--(3.6,1.5);
    	    \draw[dotted,thick](4.4,1.5)--(4.9,1.5);
            \draw[dotted,thick](5.6,1.5)--(5.1,1.5);
            \draw[dotted,thick](6.4,1.5)--(6.9,1.5);
            \draw[dotted,thick] (3.1,3.1)--(3.55,3.1);
            \draw[dotted,thick] (6.45,3.1)--(7,3.1);
            \draw[dotted,thick] (3.9,-1.8)--(4.8,-1.8);
            \draw[dotted,thick] (5.2,-1.8)--(6.1,-1.8);
            \draw[dotted,thick] (1.7,-0.4)--(2,-0.4);
            \draw[dotted,thick] (8.3,-0.4)--(8,-0.4);
 	\path[-,font=\scriptsize]
 			(a0) edge node[descr]{{\tiny{$l$}}} (a1-2);
  	\path[-,font=\scriptsize]
 			(a1-2) edge node[descr]{{\tiny{$k_1$}}} (a2-2);
    \path[-,font=\scriptsize]
 			(a1-2) edge node[descr]{{\tiny{$k_{q-1}$}}} (a2-4);
  	\path[-,font=\scriptsize]
 			(a0) edge node[descr]{{\tiny{$j$}}} (a-1);
 		\end{tikzpicture}\\
 && 	\begin{tikzpicture}[scale=0.5,descr/.style={fill=white},baseline=(b)]
\tikzstyle{every node}=[minimum size=3pt, inner sep=0.8pt]
 			\node(b) at(-19,-0.2)[minimum size=0pt,label=right:{$ +\sum\limits_{\substack{2\leqslant p \leqslant n\\ 1\leqslant q\leqslant p}}\sum\limits_{\substack{r_1+\cdots+r_q + p -q =n \\ r_1, \cdots, r_q \geqslant 1\\  1\leqslant k_1 < \cdots < k_{q-1} \leqslant p}} \sum\limits_{\substack{2\leqslant  s \leqslant r_1 \\ 1\leqslant t\leqslant s}} \sum\limits_{\substack{v_1+ \cdots v_t + s- t = r_1\\ v_1, \cdots, v_t \geqslant 1\\ 1\leqslant j_1 < \cdots < j_{t-1} \leqslant s\\ 1\leqslant i_1 < i_2 \leqslant v_1  }}\sum\limits_{\sigma} \lambda^{p-q+s -t}$}]{};
 \node(r) at (1,-4){};
 \node(b1) at (0,-1.5)[rectangle, draw]{\tiny{$ x_{s}$}};
 \node(b2-1) at (-3.5,0){};
 \node(b2-2) at (-2,0.5)[rectangle, draw]{\tiny{$ y_{v_2}$}};
  \node(b2-3) at (0,0){};
   \node(b2-4) at (2,0.5)[rectangle, draw]{\tiny{$ y_{v_t}$}};
     \node(b2-5) at (3.5,0){};
     \node(b3-1) at (-2.5,1.5){};
     \node(b3-2) at (-1,1.5){};
     \node(b3-3) at (1,1.5){};
     \node(b3-4) at (2.5,1.5){};
 \node(a0) at (1,-3)[rectangle, draw]{\tiny{$ y_{v_1}$}};
 \node(a1-1) at (-2,-1.5){};
 \node(a1-2) at (5,0.5)[rectangle, draw]{\tiny{$x_p$}};
 \node(a1-3) at (4.5,-1.5){};
 \node(a2-1) at (2,2){};
 \node(a2-2) at (3.5, 2.5)[rectangle, draw]{\tiny $y_{r_2}$};
  \node(a2-3) at (5, 2){};
  \node(a2-4) at (6.65, 2.5)[rectangle, draw]{\tiny $y_{r_q}$};
    \node(a2-5) at (8, 2){};
    \node(a3-1) at (2.5,3.5){};
    \node(a3-2) at (4,3.5){};
        \node(a3-3) at (6,3.5){};
    \node(a3-4) at (7.5,3.5){};
    \draw(b3-1) --(b2-2);
     \draw(b3-2) --(b2-2);
         \draw(b3-3) --(b2-4);
     \draw(b3-4) --(b2-4);
\draw(r)--(a0);
\draw(a0)--(b1);
\draw(b1) --(b2-1);
\draw(b1) --(b2-2);
\draw(b1) --(b2-3);
\draw(b1) --(b2-4);
\draw(b1) --(b2-5);
\draw(a0)--(a1-1);
\draw(a0)--(a1-2);
\draw(a0)--(a1-3);
\draw(a1-2)--(a2-1);
\draw(a1-2)--(a2-2);
\draw(a1-2)--(a2-3);
\draw(a1-2)--(a2-4);
\draw(a1-2)--(a2-5);
\draw(a2-2) --(a3-1);
 \draw(a2-2) --(a3-2);
 \draw(a2-4) --(a3-3);
 \draw(a2-4) --(a3-4);
 \node(sgn) at (8,0.6){$\cdot \sigma $};
    		\draw[dotted,thick](3.3,1.5)--(3.6,1.5);
    	    \draw[dotted,thick](4.4,1.5)--(4.9,1.5);
            \draw[dotted,thick](5.6,1.5)--(5.1,1.5);
            \draw[dotted,thick](6.4,1.5)--(6.9,1.5);
            \draw[dotted,thick] (3.1,3.1)--(3.55,3.1);
            \draw[dotted,thick] (6.45,3.1)--(7,3.1);
            \draw[dotted,thick] (-0.3,-2.2)--(0.2,-2.2);
            \draw[dotted,thick] (0.8,-2.2)--(2.7,-2.2);
            \draw[dotted,thick] (-2,-0.5)--(-1.3,-0.5);
            \draw[dotted,thick] (-0.8,-0.5)--(0.85,-0.5);
            \draw[dotted,thick] (1.3,-0.5)--(2,-0.5);
            \draw[dotted,thick] (-2.2,1.1)--(-1.5,1.1);
            \draw[dotted,thick] (2.2,1.1)--(1.5,1.1);
 	\path[-,font=\scriptsize]
 			(a0) edge node[descr]{{\tiny{$i_2$}}} (a1-2);
  	\path[-,font=\scriptsize]
 			(a0) edge node[descr]{{\tiny{$i_1$}}} (b1);
  	\path[-,font=\scriptsize]
 			(a1-2) edge node[descr]{{\tiny{$k_1$}}} (a2-2);
    \path[-,font=\scriptsize]
 			(a1-2) edge node[descr]{{\tiny{$k_{q-1}$}}} (a2-4);
     \path[-,font=\scriptsize]
 			(b2-2) edge node[descr]{{\tiny{$j_{1}$}}} (b1);
      \path[-,font=\scriptsize]
 			(b2-4) edge node[descr]{{\tiny{$j_{t-1}$}}} (b1);
 		\end{tikzpicture}\\
 && 	\begin{tikzpicture}[scale=0.5,descr/.style={fill=white},baseline=(b)]
\tikzstyle{every node}=[minimum size=3pt, inner sep=0.8pt]
 			\node(b) at(-24,-0.2)[minimum size=0pt,label=right:{$-\sum\limits_{\substack{2\leqslant p \leqslant n\\ 1\leqslant q\leqslant p}}\sum\limits_{\substack{r_1+\cdots+r_q + p -q =n \\ r_1, \cdots, r_q \geqslant 1\\  1\leqslant k_1 < \cdots < k_{q-1} \leqslant p}} \sum\limits_{\substack{2\leqslant s \leqslant r_1 \\ 1\leqslant t\leqslant s}} \sum\limits_{\substack{v_1+ \cdots v_t + s- t = r_1\\ v_1, \cdots, v_t \geqslant 1\\ 1\leqslant j_1 < \cdots < j_{t-1} \leqslant s\\ 1\leqslant i_1 < i_2 \leqslant v_1  }}\sum\limits_{\sigma} \lambda^{p-q + s - t}$}]{};
 \node(r) at (-1,-4){};
 \node(b1) at (0,-1.5)[rectangle, draw]{\tiny{$ x_{s}$}};
 \node(b2-1) at (3.5,0){};
 \node(b2-2) at (2,0.5)[rectangle, draw]{\tiny{$ y_{v_t}$}};
  \node(b2-3) at (0,0){};
   \node(b2-4) at (-2,0.5)[rectangle, draw]{\tiny{$ y_{v_2}$}};
     \node(b2-5) at (-3.5,0){};
     \node(b3-1) at (2.5,1.5){};
     \node(b3-2) at (1,1.5){};
     \node(b3-3) at (-1,1.5){};
     \node(b3-4) at (-2.5,1.5){};
 \node(a0) at (-1,-3)[rectangle, draw]{\tiny{$ y_{v_1}$}};
 \node(a1-1) at (2,-1.5){};
 \node(a1-2) at (-5,0.5)[rectangle, draw]{\tiny{$x_p$}};
 \node(a1-3) at (-4.5,-1.5){};
 \node(a2-1) at (-2,2){};
 \node(a2-2) at (-3.5, 2.5)[rectangle, draw]{\tiny $y_{r_q}$};
  \node(a2-3) at (-5, 2){};
  \node(a2-4) at (-6.65, 2.5)[rectangle, draw]{\tiny $y_{r_2}$};
    \node(a2-5) at (-8, 2){};
    \node(a3-1) at (-2.5,3.5){};
    \node(a3-2) at (-4,3.5){};
        \node(a3-3) at (-6,3.5){};
    \node(a3-4) at (-7.5,3.5){};
    \draw(b3-1) --(b2-2);
     \draw(b3-2) --(b2-2);
         \draw(b3-3) --(b2-4);
     \draw(b3-4) --(b2-4);
\draw(r)--(a0);
\draw(a0)--(b1);
\draw(b1) --(b2-1);
\draw(b1) --(b2-2);
\draw(b1) --(b2-3);
\draw(b1) --(b2-4);
\draw(b1) --(b2-5);
\draw(a0)--(a1-1);
\draw(a0)--(a1-2);
\draw(a0)--(a1-3);
\draw(a1-2)--(a2-1);
\draw(a1-2)--(a2-2);
\draw(a1-2)--(a2-3);
\draw(a1-2)--(a2-4);
\draw(a1-2)--(a2-5);
\draw(a2-2) --(a3-1);
 \draw(a2-2) --(a3-2);
 \draw(a2-4) --(a3-3);
 \draw(a2-4) --(a3-4);
 \node(sgn) at (4,0.6){$\cdot \sigma $};
    		\draw[dotted,thick](-3.3,1.5)--(-3.6,1.5);
    	    \draw[dotted,thick](-4.4,1.5)--(-4.9,1.5);
            \draw[dotted,thick](-5.6,1.5)--(-5.1,1.5);
            \draw[dotted,thick](-6.4,1.5)--(-6.9,1.5);
            \draw[dotted,thick] (-3.1,3.1)--(-3.55,3.1);
            \draw[dotted,thick] (-6.45,3.1)--(-7,3.1);
            \draw[dotted,thick] (0.3,-2.2)--(-0.2,-2.2);
            \draw[dotted,thick] (-0.8,-2.2)--(-2.7,-2.2);
            \draw[dotted,thick] (2,-0.5)--(1.3,-0.5);
            \draw[dotted,thick] (0.8,-0.5)--(-0.85,-0.5);
            \draw[dotted,thick] (-1.3,-0.5)--(-2,-0.5);
            \draw[dotted,thick] (2.2,1.1)--(1.5,1.1);
            \draw[dotted,thick] (-2.2,1.1)--(-1.5,1.1);
 	\path[-,font=\scriptsize]
 			(a0) edge node[descr]{{\tiny{$i_1$}}} (a1-2);
  	\path[-,font=\scriptsize]
 			(a0) edge node[descr]{{\tiny{$i_2$}}} (b1);
  	\path[-,font=\scriptsize]
 			(a1-2) edge node[descr]{{\tiny{$k_{q-1}$}}} (a2-2);
    \path[-,font=\scriptsize]
 			(a1-2) edge node[descr]{{\tiny{$k_{1}$}}} (a2-4);
     \path[-,font=\scriptsize]
 			(b2-2) edge node[descr]{{\tiny{$j_{t-1}$}}} (b1);
      \path[-,font=\scriptsize]
 			(b2-4) edge node[descr]{{\tiny{$j_{1}$}}} (b1);
 		\end{tikzpicture}\\
 &&  \begin{tikzpicture}[scale=0.5,descr/.style={fill=white},baseline=(b)]
\tikzstyle{every node}=[minimum size=3pt, inner sep=0.8pt]
 			\node(b) at(-18,-0.2)[minimum size=0pt,label=right:{$ + \sum\limits_{\substack{2\leqslant p \leqslant n\\ 1\leqslant q\leqslant p}}\sum\limits_{\substack{r_1+\cdots+r_q + p -q =n \\ r_1, \cdots, r_q \geqslant 1\\ 1\leqslant k_1 < \cdots < k_{q-1} \leqslant p}}
 \sum\limits_{\substack{2\leqslant s \leqslant r_1\\ 1 \leqslant t\leqslant s }}
 \sum\limits_{\substack{  v_1+\cdots + v_t +s-t = r_1\\ v_1, \cdots, v_t \geqslant 1 \\ 1 \leqslant j_1 < \cdots < j_{t-1} \leqslant s }}
 \sum\limits_{\substack{1\leqslant i \leqslant v_1 \\ 1\leqslant l \leqslant s \\ l\neq  j_1, \cdots, j_{t-1}}}
 \sum\limits_{\sigma} \lambda^{p-q+s-t}$}]{};
 \node(r) at (5,-4.5){};
 \node(a-1) at (5,-3.5)[rectangle, draw]{\tiny{$y_{v_1}$}};
 \node(a0-1) at (2,-2){};
  \node(a0-2) at (8,-2){};
 \node(a0) at (5,-2)[rectangle, draw]{\tiny{$ x_s$}};
 \node(a1-1) at (0.5,0){};
 \node(a1-2) at (5,0.5)[rectangle, draw]{\tiny{$x_p$}};
 \node(a1-3) at (9.5,0){};
 \node(b1-1) at (2.5,0.5)[rectangle, draw]{\tiny{$y_{v_2}$}};
 \node(b1-2) at (7.5,0.5)[rectangle, draw]{\tiny{$y_{v_{t}}$}};
 \node(b2-1) at (1.4,1.5){};
 \node(b2-2) at (2.6,1.5){};
 \node(b2-3) at (7.4,1.5){};
 \node(b2-4) at (8.6,1.5){};
 \node(a2-1) at (2,2){};
 \node(a2-2) at (3.5, 2.5)[rectangle, draw]{\tiny $y_{r_2}$};
  \node(a2-3) at (5, 2){};
  \node(a2-4) at (6.65, 2.5)[rectangle, draw]{\tiny $y_{r_q}$};
    \node(a2-5) at (8, 2){};
    \node(a3-1) at (2.5,3.5){};
    \node(a3-2) at (4,3.5){};
        \node(a3-3) at (6,3.5){};
    \node(a3-4) at (7.5,3.5){};
    \draw(b2-1) -- (b1-1);
    \draw(b2-2) -- (b1-1);
    \draw(b2-3) -- (b1-2);
    \draw(b2-4) -- (b1-2);
\draw(r)--(a-1);
\draw(a-1)--(a0-1);
\draw(a-1)--(a0-2);
\draw(a-1)--(a0);
\draw(a0)--(a1-1);
\draw(a0)--(a1-2);
\draw(a0)--(a1-3);
\draw(a1-2)--(a2-1);
\draw(a1-2)--(a2-2);
\draw(a1-2)--(a2-3);
\draw(a1-2)--(a2-4);
\draw(a1-2)--(a2-5);
\draw(a2-2) --(a3-1);
 \draw(a2-2) --(a3-2);
 \draw(a2-4) --(a3-3);
 \draw(a2-4) --(a3-4);
 \node(sgn) at (10,0.6){$\cdot \sigma $};
 			\draw[dotted,thick](4.3,-0.75)--(4.85,-0.75);
  			\draw[dotted,thick](5.15,-0.75)--(5.7,-0.75);
    		\draw[dotted,thick](3.3,1.5)--(3.6,1.5);
    	    \draw[dotted,thick](4.4,1.5)--(4.9,1.5);
            \draw[dotted,thick](5.6,1.5)--(5.1,1.5);
            \draw[dotted,thick](6.4,1.5)--(6.9,1.5);
            \draw[dotted,thick] (3.1,3.1)--(3.55,3.1);
            \draw[dotted,thick] (6.45,3.1)--(7,3.1);
            \draw[dotted,thick] (3.9,-2.8)--(4.8,-2.8);
            \draw[dotted,thick] (5.2,-2.8)--(6.1,-2.8);
            \draw[dotted,thick] (2.7,-0.75)--(3.8,-0.75);
            \draw[dotted,thick] (6.2,-0.75)--(7.3,-0.75);
             \draw[dotted,thick](2,1.2)--(2.5,1.2);
             \draw[dotted,thick](8,1.2)--(7.5,1.2);
 	\path[-,font=\scriptsize]
 			(a0) edge node[descr]{{\tiny{$l$}}} (a1-2);
  	\path[-,font=\scriptsize]
 			(a1-2) edge node[descr]{{\tiny{$k_1$}}} (a2-2);
    \path[-,font=\scriptsize]
 			(a1-2) edge node[descr]{{\tiny{$k_{q-1}$}}} (a2-4);
  	\path[-,font=\scriptsize]
 			(a0) edge node[descr]{{\tiny{$i$}}} (a-1);
   	\path[-,font=\scriptsize]
 			(a0) edge node[descr]{{\tiny{$j_1$}}} (b1-1);
    	\path[-,font=\scriptsize]
 			(a0) edge node[descr]{{\tiny{$j_{t-1}$}}} (b1-2);
 		\end{tikzpicture}\\
 &&  \begin{tikzpicture}[scale=0.5,descr/.style={fill=white},baseline=(b)]
\tikzstyle{every node}=[minimum size=3pt, inner sep=0.8pt]
 			\node(b) at(-16,-0.2)[minimum size=0pt,label=right:{$+ \sum\limits_{\substack{2\leqslant p \leqslant n\\ 1\leqslant q\leqslant p}}\sum\limits_{\substack{r_1+\cdots+r_q + p -q =n \\ r_1, \cdots, r_q \geqslant 1\\ 1\leqslant k_1 < \cdots < k_{q-1} \leqslant p}}
 \sum\limits_{\substack{2\leqslant s \leqslant r_1\\ 2 \leqslant t\leqslant s \\ 1\leqslant u \leqslant t-1  }}
 \sum\limits_{\substack{  v_1+\cdots + v_t + s - t = r_1\\ v_1, \cdots, v_t \geqslant 1 \\ 1 \leqslant j_1 < \cdots < j_{t-1} \leqslant s \\ 1\leqslant i_1 \leqslant v_1 \\ 1\leqslant i_2 \leqslant v_{u+1}}}
 \sum\limits_{\sigma} \lambda^{p-q+s-t}$}]{};
 \node(r) at (5,-4.5){};
 \node(a-1) at (5,-3.5)[rectangle, draw]{\tiny{$y_{v_1}$}};
 \node(a0-1) at (2,-2){};
  \node(a0-2) at (8,-2){};
 \node(a0) at (5,-2)[rectangle, draw]{\tiny{$ x_s$}};
 \node(a1-1) at (0.5,0){};
 \node(a1-2) at (5,2.5)[rectangle, draw]{\tiny{$x_p$}};
 \node(a1-2') at (5,0.5)[rectangle, draw]{\tiny{$y_{v_{u+1}}$}};
 \node(a2-1') at (3,2){};
 \node(a2-2') at (7,2){};
 \node(a1-3) at (9.5,0){};
 \node(b1-1) at (2.5,0.5)[rectangle, draw]{\tiny{$y_{v_2}$}};
 \node(b1-2) at (7.5,0.5)[rectangle, draw]{\tiny{$y_{v_{t}}$}};
 \node(b2-1) at (1.4,1.5){};
 \node(b2-2) at (2.6,1.5){};
 \node(b2-3) at (7.4,1.5){};
 \node(b2-4) at (8.6,1.5){};
 \node(a2-1) at (2,4){};
 \node(a2-2) at (3.5, 4.5)[rectangle, draw]{\tiny $y_{r_2}$};
  \node(a2-3) at (5, 4){};
  \node(a2-4) at (6.65, 4.5)[rectangle, draw]{\tiny $y_{r_q}$};
    \node(a2-5) at (8, 4){};
    \node(a3-1) at (2.5,5.5){};
    \node(a3-2) at (4,5.5){};
        \node(a3-3) at (6,5.5){};
    \node(a3-4) at (7.5,5.5){};
    \draw(a2-1') -- (a1-2');
    \draw(a2-2') -- (a1-2');
    \draw(b2-1) -- (b1-1);
    \draw(b2-2) -- (b1-1);
    \draw(b2-3) -- (b1-2);
    \draw(b2-4) -- (b1-2);
\draw(r)--(a-1);
\draw(a-1)--(a0-1);
\draw(a-1)--(a0-2);
\draw(a-1)--(a0);
\draw(a0)--(a1-1);
\draw(a0)--(a1-2');
\draw(a0)--(a1-3);
\draw(a1-2) -- (a1-2');
\draw(a1-2)--(a2-1);
\draw(a1-2)--(a2-2);
\draw(a1-2)--(a2-3);
\draw(a1-2)--(a2-4);
\draw(a1-2)--(a2-5);
\draw(a2-2) --(a3-1);
 \draw(a2-2) --(a3-2);
 \draw(a2-4) --(a3-3);
 \draw(a2-4) --(a3-4);
 \node(sgn) at (10,0.6){$\cdot \sigma $};
 			\draw[dotted,thick](4.3,-0.75)--(4.65,-0.75);
  			\draw[dotted,thick](5.45,-0.75)--(5.7,-0.75);
    		\draw[dotted,thick](3.3,3.5)--(3.6,3.5);
    	    \draw[dotted,thick](4.4,3.5)--(4.9,3.5);
            \draw[dotted,thick](5.6,3.5)--(5.1,3.5);
            \draw[dotted,thick](6.4,3.5)--(6.9,3.5);
            \draw[dotted,thick] (3.1,5.1)--(3.55,5.1);
            \draw[dotted,thick] (6.45,5.1)--(7,5.1);
            \draw[dotted,thick] (3.9,-2.8)--(4.8,-2.8);
            \draw[dotted,thick] (5.25,-2.8)--(6.1,-2.8);
            \draw[dotted,thick] (2.7,-0.75)--(3.8,-0.75);
            \draw[dotted,thick] (6.2,-0.75)--(7.3,-0.75);
             \draw[dotted,thick](2,1.2)--(2.5,1.2);
             \draw[dotted,thick](8,1.2)--(7.5,1.2);
             \draw[dotted,thick](4,1.5)--(4.7,1.5);
             \draw[dotted,thick](5.3,1.5)--(6,1.5);
 	\path[-,font=\scriptsize]
 			(a0) edge node[descr]{{\tiny{$j_u$}}} (a1-2');
  	\path[-,font=\scriptsize]
 			(a1-2) edge node[descr]{{\tiny{$i_2$}}} (a1-2');
  	\path[-,font=\scriptsize]
 			(a1-2) edge node[descr]{{\tiny{$k_1$}}} (a2-2);
    \path[-,font=\scriptsize]
 			(a1-2) edge node[descr]{{\tiny{$k_{q-1}$}}} (a2-4);
  	\path[-,font=\scriptsize]
 			(a0) edge node[descr]{{\tiny{$i_1$}}} (a-1);
   	\path[-,font=\scriptsize]
 			(a0) edge node[descr]{{\tiny{$j_1$}}} (b1-1);
    	\path[-,font=\scriptsize]
 			(a0) edge node[descr]{{\tiny{$j_{t-1}$}}} (b1-2);
 		\end{tikzpicture} \\
 && \begin{tikzpicture}[scale=0.5,descr/.style={fill=white},baseline=(b)]
\tikzstyle{every node}=[minimum size=3pt, inner sep=0.8pt]
 			\node(b) at(-15,-0.2)[minimum size=0pt,label=right:{$- \sum\limits_{\substack{2\leqslant p \leqslant n\\ 1\leqslant q\leqslant p\\ 1\leqslant t \leqslant q\\ 2\leqslant s \leqslant p -1 }}\sum\limits_{\substack{r_1+\cdots+r_{t} + v_1 + \cdots + v_{q-t} + p -q =n \\ r_1, \cdots, r_t, v_1, \cdots, v_{q-t} \geqslant 1\\ 1\leqslant k_1 < \cdots < k_{t-1} \leqslant p-s-1\\ 1\leqslant j_1 < \cdots < j_{q-t} \leqslant s}}
 \sum\limits_{\substack{ 1\leqslant i_1 \leqslant r_1\\ 1\leqslant i_2 \leqslant p-s-1\\ i_2 \neq k_1, \cdots, k_{t-1} }}
 \sum\limits_{\sigma} \lambda^{p-q}$}]{};
 \node(r) at (5,-4.5){};
 \node(a-1) at (5,-3.5)[rectangle, draw]{\tiny{$y_{r_1}$}};
 \node(a0-1) at (2,-2){};
  \node(a0-2) at (8,-2){};
 \node(a0) at (5,-2)[rectangle, draw]{\tiny{$ x_{p-s-1}$}};
 \node(a1-1) at (0.5,0){};
 \node(a1-2) at (5,0.5)[rectangle, draw]{\tiny{$x_s$}};
 \node(a1-3) at (9.5,0){};
 \node(b1-1) at (2.5,0.5)[rectangle, draw]{\tiny{$y_{r_2}$}};
 \node(b1-2) at (7.5,0.5)[rectangle, draw]{\tiny{$y_{r_{t}}$}};
 \node(b2-1) at (1.4,1.5){};
 \node(b2-2) at (2.6,1.5){};
 \node(b2-3) at (7.4,1.5){};
 \node(b2-4) at (8.6,1.5){};
 \node(a2-1) at (2,2){};
 \node(a2-2) at (3.5, 2.5)[rectangle, draw]{\tiny $y_{v_1}$};
  \node(a2-3) at (5, 2){};
  \node(a2-4) at (6.65, 2.5)[rectangle, draw]{\tiny $y_{v_{q-t}}$};
    \node(a2-5) at (8, 2){};
    \node(a3-1) at (2.5,3.5){};
    \node(a3-2) at (4,3.5){};
        \node(a3-3) at (6,3.5){};
    \node(a3-4) at (7.5,3.5){};
    \draw(b2-1) -- (b1-1);
    \draw(b2-2) -- (b1-1);
    \draw(b2-3) -- (b1-2);
    \draw(b2-4) -- (b1-2);
\draw(r)--(a-1);
\draw(a-1)--(a0-1);
\draw(a-1)--(a0-2);
\draw(a-1)--(a0);
\draw(a0)--(a1-1);
\draw(a0)--(a1-2);
\draw(a0)--(a1-3);
\draw(a1-2)--(a2-1);
\draw(a1-2)--(a2-2);
\draw(a1-2)--(a2-3);
\draw(a1-2)--(a2-4);
\draw(a1-2)--(a2-5);
\draw(a2-2) --(a3-1);
 \draw(a2-2) --(a3-2);
 \draw(a2-4) --(a3-3);
 \draw(a2-4) --(a3-4);
 \node(sgn) at (10,0.6){$\cdot \sigma $};
 			\draw[dotted,thick](4.3,-0.75)--(4.85,-0.75);
  			\draw[dotted,thick](5.15,-0.75)--(5.7,-0.75);
    		\draw[dotted,thick](3.3,1.5)--(3.6,1.5);
    	    \draw[dotted,thick](4.4,1.5)--(4.9,1.5);
            \draw[dotted,thick](5.6,1.5)--(5.1,1.5);
            \draw[dotted,thick](6.4,1.5)--(6.9,1.5);
            \draw[dotted,thick] (3.1,3.1)--(3.55,3.1);
            \draw[dotted,thick] (6.45,3.1)--(7,3.1);
            \draw[dotted,thick] (3.9,-2.8)--(4.8,-2.8);
            \draw[dotted,thick] (5.35,-2.8)--(6.1,-2.8);
            \draw[dotted,thick] (2.7,-0.75)--(3.8,-0.75);
            \draw[dotted,thick] (6.2,-0.75)--(7.3,-0.75);
             \draw[dotted,thick](2,1.2)--(2.5,1.2);
             \draw[dotted,thick](8,1.2)--(7.5,1.2);
 	\path[-,font=\scriptsize]
 			(a0) edge node[descr]{{\tiny{$i_2$}}} (a1-2);
  	\path[-,font=\scriptsize]
 			(a1-2) edge node[descr]{{\tiny{$j_1$}}} (a2-2);
    \path[-,font=\scriptsize]
 			(a1-2) edge node[descr]{{\tiny{$j_{q-t}$}}} (a2-4);
  	\path[-,font=\scriptsize]
 			(a0) edge node[descr]{{\tiny{$i_1$}}} (a-1);
   	\path[-,font=\scriptsize]
 			(a0) edge node[descr]{{\tiny{$k_1$}}} (b1-1);
    	\path[-,font=\scriptsize]
 			(a0) edge node[descr]{{\tiny{$k_{t-1}$}}} (b1-2);
 		\end{tikzpicture} \\
  && \begin{tikzpicture}[scale=0.5,descr/.style={fill=white},baseline=(b)]
\tikzstyle{every node}=[minimum size=3pt, inner sep=0.8pt]
 			\node(b) at(-12,-0.2)[minimum size=0pt,label=right:{$ + \sum\limits_{\substack{2\leqslant p \leqslant n\\ 2\leqslant q\leqslant p\\ 2\leqslant l \leqslant q}}\sum\limits_{\substack{r_1+\cdots+r_q + p -q =n \\ r_1, \cdots, r_q \geqslant 1\\ 1\leqslant i \leqslant r_1\\  1\leqslant k_1 < \cdots < k_{q-1} \leqslant p}}
 \sum\limits_{\substack{2\leqslant s \leqslant r_{l+1}\\ v_1+ \cdots + v_s = r_{l+1} \\ v_1, \cdots, v_s \geqslant 1}} \sum\limits_{\sigma} \lambda^{p-q}$}]{};
 \node(r) at (5,-2){};
 \node(a0) at (5,-1)[rectangle, draw]{\tiny{$y_{r_1}$}};
 \node(a1-1) at (1.5,0.5){};
 \node(a1-2) at (5,0.5)[rectangle, draw]{\tiny{$x_p$}};
 \node(a1-3) at (8.5,0.5){};
 \node(a2-1) at (1,2){};
 \node(a2-2) at (2.5, 2.5)[rectangle, draw]{\tiny $y_{r_2}$};
 \node(b2-1) at (4.2,2.1){};
 \node(b2-2) at (5.8,2.1){};
  \node(a2-3) at (5, 2.5)[rectangle, draw]{\tiny $x_{s}$};
  \node(a2-4) at (7.5, 2.5)[rectangle, draw]{\tiny $y_{r_q}$};
    \node(a2-5) at (9, 2){};
    \node(a3-1) at (1.5,3.5){};
    \node(a3-2) at (3.1,3.5){};
        \node(a3-3) at (6.9,3.5){};
    \node(a3-4) at (8.5,3.5){};
    \node(b3-1) at (3.5,4)[rectangle, draw]{\tiny $y_{v_1}$};
    \node(b3-2) at (6.5,4)[rectangle, draw]{\tiny $y_{v_{s}}$};
\draw(b3-1) -- (2.5,5);
\draw(b3-1)--(4.5,5);
\draw(b3-2) -- (5.5,5);
\draw(b3-2) --(7.5,5);
\draw(r)--(a0);
\draw(a0)--(a1-1);
\draw(a0)--(a1-2);
\draw(a0)--(a1-3);
\draw(a1-2)--(a2-1);
\draw(a1-2)--(a2-2);
\draw(a1-2)--(a2-3);
\draw(a1-2)--(a2-4);
\draw(a1-2)--(a2-5);
\draw(a2-2) --(a3-1);
 \draw(a2-2) --(a3-2);
 \draw(a2-4) --(a3-3);
 \draw(a2-4) --(a3-4);
 \draw(b2-1) -- (a1-2);
 \draw(b2-2) -- (a1-2);
 \draw(b3-1) -- (a2-3);
 \draw(b3-2) -- (a2-3);
 \node(sgn) at (9,0.6){$\cdot \sigma $};
 			\draw[dotted,thick](3.6,-0.25)--(4.85,-0.25);
  			\draw[dotted,thick](5.15,-0.25)--(6.4,-0.25);
    		\draw[dotted,thick](2.9,1.5)--(3.4,1.5);
    	    \draw[dotted,thick](3.85,1.5)--(4.7,1.5);
        	\draw[dotted,thick](5.1,1.5)--(5.6,1.5);
            \draw[dotted,thick](6.4,1.5)--(7.3,1.5);
            \draw[dotted,thick] (2.1,3.1)--(2.75,3.1);
            \draw[dotted,thick] (4.4,3.3)--(5.65,3.3);
            \draw[dotted,thick] (7.3,3.1)--(7.95,3.1);
            \draw[dotted,thick] (3.1,4.7)--(4.05,4.7);
            \draw[dotted,thick] (6.15,4.7)--(7.1,4.7);
 	\path[-,font=\scriptsize]
 			(a0) edge node[descr]{{\tiny{$i$}}} (a1-2);
  	\path[-,font=\scriptsize]
 			(a1-2) edge node[descr]{{\tiny{$k_1$}}} (a2-2);
    \path[-,font=\scriptsize]
 			(a1-2) edge node[descr]{{\tiny{$k_{q-1}$}}} (a2-4);
     \path[-,font=\scriptsize]
 			(a1-2) edge node[descr]{{\tiny{$k_{l}$}}} (a2-3);
 		\end{tikzpicture} \\
  && \begin{tikzpicture}[scale=0.5,descr/.style={fill=white},baseline=(b)]
\tikzstyle{every node}=[minimum size=3pt, inner sep=0.8pt]
 			\node(b) at(-16,1.8)[minimum size=0pt,label=right:{$ - \sum\limits_{\substack{2\leqslant p \leqslant n\\ 2\leqslant q\leqslant p\\ 2\leqslant l \leqslant q}}\sum\limits_{\substack{r_1+\cdots+r_q + p -q =n \\ r_1, \cdots, r_q \geqslant 1\\   1\leqslant k_1 < \cdots < k_{q-1} \leqslant p}}
 \sum\limits_{\substack{2\leqslant s \leqslant r_{l+1}\\ 1\leqslant t \leqslant s \\ 1\leqslant i_1 \leqslant r_1 \\ 1\leqslant i_2 \leqslant v_1 }}
 \sum\limits_{\substack{ v_1+ \cdots + v_t  + s -t = r_{l+1} \\ v_1, \cdots, v_t \geqslant 1 \\ 1\leqslant j_1  < \cdots < j_{t-1} \leqslant s}}
 \sum\limits_{\sigma} \lambda^{p-q+s-t}$}]{};
 \node(r) at (5,-2){};
 \node(a0) at (5,-1)[rectangle, draw]{\tiny{$y_{r_1}$}};
 \node(a1-1) at (1.5,0.5){};
 \node(a1-2) at (5,0.5)[rectangle, draw]{\tiny{$x_p$}};
 \node(a1-3) at (8.5,0.5){};
 \node(a2-1) at (1,2){};
 \node(a2-2) at (2.5, 2.5)[rectangle, draw]{\tiny $y_{r_2}$};
 \node(b2-1) at (4.2,2.1){};
 \node(b2-2) at (5.8,2.1){};
  \node(a2-3) at (5, 2.5)[rectangle, draw]{\tiny $y_{v_1}$};
  \node(a2-4) at (7.5, 2.5)[rectangle, draw]{\tiny $y_{r_q}$};
    \node(a2-5) at (9, 2){};
    \node(a3-1) at (1.5,3.5){};
    \node(a3-2) at (3.1,3.5){};
        \node(a3-3) at (6.9,3.5){};
    \node(a3-4) at (8.5,3.5){};
    \node(b3-1) at (2.5,4){};
    \node(b3-1') at (5,4)[rectangle, draw]{\tiny $x_{s}$};
    \draw(b3-1') --(1,5);
    \draw(b3-1') -- (5,6);
    \draw(b3-1') --(9,5);
    \node(b3-2) at (7.5,4){};
    \node(b4-1) at (3,6)[rectangle, draw]{\tiny $y_{v_2}$};
    \draw(b4-1) -- (2,7);
    \draw(b4-1) -- (4,7);
    \draw[dotted,thick](2.6,6.7)--(3.45,6.7);
    \node(b4-2) at (7,6)[rectangle, draw]{\tiny $y_{v_{t}}$};
    \draw(b4-2) -- (6,7);
    \draw(b4-2) -- (8,7);
    \draw[dotted,thick](7.4,6.7)--(6.55,6.7);
\draw(r)--(a0);
\draw(a0)--(a1-1);
\draw(a0)--(a1-2);
\draw(a0)--(a1-3);
\draw(a1-2)--(a2-1);
\draw(a1-2)--(a2-2);
\draw(a1-2)--(a2-3);
\draw(a1-2)--(a2-4);
\draw(a1-2)--(a2-5);
\draw(a2-2) --(a3-1);
 \draw(a2-2) --(a3-2);
 \draw(a2-4) --(a3-3);
 \draw(a2-4) --(a3-4);
 \draw(b2-1) -- (a1-2);
 \draw(b2-2) -- (a1-2);
 \draw(b3-1) -- (a2-3);
 \draw(b3-2) -- (a2-3);
 \node(sgn) at (9.5,2.6){$\cdot \sigma $};
 			\draw[dotted,thick](3.6,-0.25)--(4.85,-0.25);
  			\draw[dotted,thick](5.15,-0.25)--(6.4,-0.25);
    		\draw[dotted,thick](2.9,1.5)--(3.4,1.5);
    	    \draw[dotted,thick](3.85,1.5)--(4.7,1.5);
        	\draw[dotted,thick](5.1,1.5)--(5.6,1.5);
            \draw[dotted,thick](6.4,1.5)--(7.3,1.5);
            \draw[dotted,thick] (2.1,3.1)--(2.75,3.1);
            \draw[dotted,thick] (4,3.3)--(4.65,3.3);
            \draw[dotted,thick] (6,3.3)--(5.35,3.3);
            \draw[dotted,thick] (7.3,3.1)--(7.95,3.1);
 	\path[-,font=\scriptsize]
 			(a0) edge node[descr]{{\tiny{$i_1$}}} (a1-2);
  	\path[-,font=\scriptsize]
 			(a1-2) edge node[descr]{{\tiny{$k_1$}}} (a2-2);
    \path[-,font=\scriptsize]
 			(a1-2) edge node[descr]{{\tiny{$k_{q-1}$}}} (a2-4);
     \path[-,font=\scriptsize]
 			(a1-2) edge node[descr]{{\tiny{$k_{l}$}}} (a2-3);
  \path[-,font=\scriptsize]
 			(b3-1') edge node[descr]{{\tiny{$i_2$}}} (a2-3);
   \path[-,font=\scriptsize]
 			(b3-1') edge node[descr]{{\tiny{$j_1$}}} (b4-1);
   \path[-,font=\scriptsize]
 			(b3-1') edge node[descr]{{\tiny{$j_{t-1}$}}} (b4-2);
 \draw[dotted,thick] (2.2,5)--(3.7,5);
  \draw[dotted,thick] (4.5,5)--(5.4,5);
 \draw[dotted,thick] (6.5,5)--(8,5);
 		\end{tikzpicture}
 \end{eqnarray*}
where all the $\sum_\sigma$ are summations over those $\sigma \in \s_n$ that render the tree monomials into shuffle tree monomials.
Note that in the final summation above, each tree appears twice with opposite signs, resulting in pairwise cancellation. This establishes the identity $\partial^2(y_n)=0$.
\end{proof}

\begin{defn}
	The symmetric homotopy cooperad $\mathscr{S}\RBLA^\antish\ot_{\mathrm{H}} \calS^{-1}$ is called the Koszul dual cooperad of $\RBLA$, and it is denoted as $\RBLA^\antish$.
\end{defn}

 Precisely, the underlying $\mathbb{S}$-module of $\RBLA^\antish$ is
 $$\RBLA^\antish(n)=\bfk e_n\oplus \bfk o_n,n\geq 1$$
 with $e_n=u_n\otimes \delta^{-1}_n $ and $o_n=v_n \otimes \delta^{-1}_n,$ thus   $|e_n|=n-1,e_n\sigma=\sgn(\sigma)e_n$ and $|o_n|=n,o_n\sigma=\sgn(\sigma)o_n,$ for $n\geq 1$ and $\sigma \in \s_n.$
The family of operations $\{\Delta_T\}_{T\in \frakt}$ defining its symmetric homotopy cooperad structure is given by the following:

\begin{itemize}
	\item[(i)] for element $e_n\in\mathscr{S}\RBLA^\antish(n)$ and $T$ of type $\mathrm{(I)}$, define
	\begin{eqnarray*}
		\begin{tikzpicture}[scale=1,descr/.style={fill=white}]
			\tikzstyle{every node}=[thick,minimum size=5pt, inner sep=1pt]
			\node(r) at (0,-0.5)[minimum size=0pt,rectangle]{};
			\node(v-2) at(-1.8,0.5)[minimum size=0pt, label=left:{\Large$\Delta_T(e_n)=(-1)^{(j+1)(n-i+1)} \sgn(\sigma)$}]{};
			\node(v0) at (0,0)[draw,rectangle]{{\small $e_{n-j+1}$}};
\draw(r)--(v0);
			\node(v1-1) at (-1.5,1){};
			\node(v1-2) at(0,1)[draw,rectangle]{\small$e_j$};
			\node(v1-3) at(1.5,1){};
			\node(v2-1)at (-1,2){};
			\node(v2-2) at(1,2){};
\node(v3-1) at (2,0.5) {$\cdot \sigma$};
			\draw(v0)--(v1-1);
			\draw(v0)--(v1-3);
			\draw(v1-2)--(v2-1);
			\draw(v1-2)--(v2-2);
			\draw[thick,dotted](-0.4,1.5)--(0.4,1.5);
			\draw[thick,dotted](-0.5,0.5)--(-0.1,0.5);
			\draw[thick,dotted](0.1,0.5)--(0.5,0.5);
			\path[-,font=\scriptsize]
			(v0) edge node[descr]{{\tiny$i$}} (v1-2);
		\end{tikzpicture}
	\end{eqnarray*}
	\item[(ii)] for the element $o_n\in\mathscr{S}\RBLA^\antish(n)$ and tree $T$ of type $\mathrm{(I)}$ with $j=1$, define
	
	\begin{eqnarray*}
		\begin{tikzpicture}[scale=1,descr/.style={fill=white}]
			\tikzstyle{every node}=[thick,minimum size=5pt, inner sep=1pt]
			\node(r) at (0,-0.5)[minimum size=0pt,rectangle]{};
			\node(v-1) at(-2,0.5)[minimum size=0pt, label=left:{\Large$\Delta_T(o_n)=\sgn(\sigma)$}]{};
			\node(v0) at (0,0)[draw,rectangle]{\small$o_n$};
\draw(r) -- (v0);
			\node(v1-1) at (-1.3,1){};
			\node(v1-2) at(0,1)[draw,rectangle]{\small$e_1$};
			\node(v1-3) at(1.3,1){};
			\node(v2-1)at (0,1.8){};
\node(v3-1) at (1.7,0.5){$\cdot \sigma$};
			\draw(v0)--(v1-1);
			\draw(v0)--(v1-3);
			\draw(v1-2)--(v2-1);
			\draw[thick,dotted](-0.5,0.5)--(-0.1,0.5);
			\draw[thick,dotted](0.1,0.5)--(0.5,0.5);
			\path[-,font=\scriptsize]
			(v0) edge node[descr]{{\tiny$i$}} (v1-2);
		\end{tikzpicture}
	\end{eqnarray*}
	for tree $T$ of type $\mathrm{(I)}$ with $2\leqslant j\leqslant n-1$, define
	\begin{eqnarray*}
		\begin{tikzpicture}[scale=1,descr/.style={fill=white}]
			\tikzstyle{every node}=[thick,minimum size=5pt, inner sep=1pt]
			\node(r) at (0,-0.5)[minimum size=0pt,rectangle]{};
			\node(v-2) at(-2,0.5)[minimum size=0pt, label=left:{\Large$\Delta_T(o_n)=(-1)^{(j+1)(n-i+1)}\lambda^{j-1}\sgn(\sigma)$}]{};
			\node(v0) at (0,0)[draw, rectangle]{$o_{n-j+1}$};
\draw(v0) -- (r);
			\node(v1-1) at (-1.5,1) {};
			\node(v1-2) at(0,1)[draw,rectangle]{\small $e_j$};
			\node(v1-3) at(1.5,1) {};
			\node(v2-1)at (-1,2){};
			\node(v2-2) at(1,2){};
\node(v3-1) at (2,0.5){$\cdot\sigma$};
			\draw(v0)--(v1-1);
			\draw(v0)--(v1-3);
			\draw(v1-2)--(v2-1);
			\draw(v1-2)--(v2-2);
			\draw[thick,dotted](-0.4,1.5)--(0.4,1.5);
			\draw[thick,dotted](-0.5,0.5)--(-0.1,0.5);
			\draw[thick,dotted](0.1,0.5)--(0.5,0.5);
			\path[-,font=\scriptsize]
			(v0) edge node[descr]{{\tiny$i$}} (v1-2);
		\end{tikzpicture}
	\end{eqnarray*}
	and for tree $T$ of type $(1)$ with $j=n$, define
	\begin{eqnarray*}
		\begin{tikzpicture}[scale=1,descr/.style={fill=white}]
			\tikzstyle{every node}=[thick,minimum size=5pt, inner sep=1pt]
			\node(r) at (0,-0.5)[minimum size=0pt,rectangle]{};
			\node(va) at(-4,0.5)[minimum size=0pt, label=left:{\Large$\Delta_T(o_n)=$}]{};
			\node(vb) at(-1,0.5)[minimum size=0pt,label=left:{\Large$\lambda^{n-1}\sgn(\sigma)$}]{};
\begin{scope}[yshift=-0.5cm]
			\node(vc) at (0,0)[draw, rectangle]{\small $o_1$};
\draw(vc) -- (0,-0.5);
			\node(v1) at(0,1)[draw,rectangle]{\small $e_n$};
			\node(v2-1)at (-1,2) {};
			\node(v2-2) at(1,2) {};
			\draw[thick,dotted](-0.4,1.5)--(0.4,1.5);
\node(v3-1) at (1.3,1){$\cdot\sigma$};
\end{scope}
\begin{scope}[xshift=0.5cm]
			\node(vd) at(1.3,0.5)[minimum size=0, label=right:$+\ \ $\Large$\sgn(\sigma) $]{};
\begin{scope}[yshift=-0.5cm]
			\node(ve) at (4,0)[draw, rectangle]{\small $e_1$};
\draw(ve) -- (4,-0.5);
			\node(ve1) at (4,1)[draw,rectangle]{\small $o_n$};
			\node(ve2-1) at(3,2) {};
			\node(ve2-2) at(5,2) {};
\node(ve3-1) at (5.3,1){$\cdot\sigma$};
			\draw[thick,dotted](3.6,1.5)--(4.4,1.5);
\end{scope}
\end{scope}
			\draw(v1)--(v2-1);
			\draw(v1)--(v2-2);
			\draw(vc)--(v1);
			\draw(ve)--(ve1);
			\draw(ve1)--(ve2-1);
			\draw(ve1)--(ve2-2);
		\end{tikzpicture}
	\end{eqnarray*}
	
	\item[(iii)] for tree $T$ of type $\mathrm{(II)}$ with $r_1+\dots+r_k=n$, define
	
	\begin{eqnarray*}
		\begin{tikzpicture}[scale=1,descr/.style={fill=white}]
			\tikzstyle{every node}=[thick,minimum size=5pt, inner sep=1pt]
			\node(v-2) at (-6,-0.5)[minimum size=0pt, label=left:{\Large$\Delta_T(o_n)=$}]{};
			\node(v-1) at(-6,-0.5)[minimum size=0pt,label=right:{\Large$(-1)^{\frac{k(k-1)}{2}}\sgn(\sigma)$}]{};
			\node(v0) at (0,-1.5)[rectangle, draw]{\small $e_k$};
\draw(v0) -- (0,-2);
			\node(v1) at(-1.2,-0.3)[rectangle,draw]{\small $o_{r_1}$};
			\node(v1-1) at(-2,0.8){};
			\node(v1-2) at (-1,0.8){};
			\node(v3) at (1.2,-0.3)[rectangle, draw]{\small $o_{r_k}$};
			\node(v3-1) at (1,0.8){};
			\node(v3-2) at (2,0.8){};
			\node(v2-1) at (0,0) [rectangle,draw]{\small $o_{r_i}$};
			\node(v2-1-1) at (-0.5,1){};
			\node(v2-1-2) at (0.5, 1){};
\node(v4-1) at (2.3,-0.5){$\cdot\sigma$};
			\draw [thick,dotted ] (-0.7,-0.5)--(-0.2,-0.5);
			\draw [thick,dotted ] (0.3,-0.5)--(0.8,-0.5);
			\draw [thick,dotted ] (-0.2,0.5)--(0.2,0.5);
			\draw [thick,dotted ] (-1.6,0.5)--(-1.2, 0.5);
			\draw [thick,dotted ] (1.2,0.5)--(1.6,0.5);
			\draw        (v0)--(v1);
			\draw         (v0)--(v3);
			\draw         (v1)--(v1-1);
			\draw          (v1)--(v1-2);
			\draw        (v2-1)--(v2-1-1);
			\draw        (v2-1)--(v2-1-2);
			\draw        (v3)--(v3-1);
			\draw        (v3)--(v3-2);
			\path[-,font=\scriptsize]
			(v0) edge node[descr]{{\tiny$i$}} (v2-1);
		\end{tikzpicture}
	\end{eqnarray*}
	\item[(iv)] for tree $T$ of type $\mathrm{(III)}$ with $r_1+\dots+r_q+p-q=n$, define
	\begin{eqnarray*}
		\begin{tikzpicture}[scale=1,descr/.style={fill=white}]
			\tikzstyle{every node}=[thick,minimum size=5pt, inner sep=1pt]
			\node(v-2) at (-7,1.2)[minimum size=0pt, label=left:{\Large$\Delta_T(o_n)=$}]{};
			\node(v-1) at(-6.8,1.2)[minimum size=0pt,label=right:{\Large$(-1)^\beta\sgn(\sigma)\lambda^{p-q}$}]{};
			\node(v0) at (0,0)[rectangle,draw]{\small $o_{r_1}$};
\draw(v0)--(0,-0.5);
			\node(v1-1) at (-2,1){};
			\node(v1-2) at(0,1.2)[rectangle,draw]{\small $e_p$};
			\node(v1-3) at(2,1){};
			\node(v2-1) at(-1.9,2.6){};
			\node(v2-2) at (-0.9, 2.8)[rectangle,draw]{\small$o_{r_2}$};
			\node(v2-3) at (0,2.9){};
			\node(v2-4) at(0.9,2.8)[rectangle,draw]{\small $o_{r_q}$};
			\node(v2-5) at(1.9,2.6){};
			\node(v3-1) at (-1.5,3.5){};
			\node(v3-2) at (-0.3,3.5){};
			\node(v3-3) at (0.3,3.5){};
			\node(v3-4) at(1.5,3.5){};
\node(v4-1) at (2.5,1.2) {$\cdot\sigma$};
			\draw(v0)--(v1-1);
			\draw(v0)--(v1-3);
			\path[-,font=\scriptsize]
			(v0) edge node[descr]{{\tiny$i$}} (v1-2);
			\draw(v1-2)--(v2-1);
			\draw(v1-2)--(v2-3);
			\draw(v1-2)--(v2-5);
			\path[-,font=\scriptsize]
			(v1-2) edge node[descr]{{\tiny$k_1$}} (v2-2)
			edge node[descr]{{\tiny$k_{q-1}$}} (v2-4);
			\draw(v2-2)--(v3-1);
			\draw(v2-2)--(v3-2);
			\draw(v2-4)--(v3-3);
			\draw(v2-4)--(v3-4);
			\draw[thick,dotted](-1,0.7)--(-0.1,0.7);
			\draw[thick,dotted](0.1,0.7)--(1,0.7);
			\draw[thick,dotted](-0.5,2.4)--(-0.1,2.4);
			\draw[thick,dotted](0.1,2.4)--(0.5,2.4);
			\draw[thick,dotted](-1.4,2.4)--(-0.8,2.4);
			\draw[thick,dotted](1.4,2.4)--(0.8,2.4);
			\draw[thick,dotted](-1.1,3.2)--(-0.6,3.2);
			\draw[thick,dotted](1.1,3.2)--(0.6,3.2);
		\end{tikzpicture}
	\end{eqnarray*}
	where \begin{eqnarray*} \beta=\sum_{j=1}^{q-1}(q-j)r_j+(p-1)(q-1)+(\sum_{j=2}^qr_j+p-q)(r_1-i) +\sum_{j=2}^q(r_j-1)(p-k_{j-1});    	\end{eqnarray*}
	\item[(v)] all other components of $\Delta_T, T\in \frakt$ vanish.
\end{itemize}

\bigskip

\subsection{The minimal model for the operad of Rota-Baxter Lie algebras}\label{sect: minimal model}
In the previous section, we constructed a coaugmented symmetric homotopy cooperad $\RBLA^\antish.$  We now prove that $\Omega(\RBLA^\antish) $ is the minimal model of the symmetric operad $\RBLA$, using algebraic Morse theory.

\begin{defn}\label{Def: RBLA-infty}
The symmetric dg  operad $ \Omega \RBLA^\antish,$ denoted by   $\RBLA_\infty,$  is called  \textbf{the operad of homotopy Rota-Baxter Lie algebras of weight $\lambda.$}
\end{defn}
Let us make this dg operad explicit.  The dg operad $\RBLA_\infty$ is the free operad generated by the graded $\mathbb{S}$-module $s^{-1}\overline{\RBLA^\antish}$ endowed with the differential induced from the homotopy cooperad structure.
More precisely,
$$s^{-1}\overline{\RBLA^\antish}(1)=\bfk s^{-1}o_1 \ \  \text{and} \ \  s^{-1}\overline{\RBLA^\antish}(n)=\bfk s^{-1}e_n\oplus \bfk s^{-1}o_n, n\geqslant2.$$
We denote
$\ell_n = s^{-1} e_n, n\geqslant 2 \ \ \text{and}  \ \ T_n = s^{-1}o_n, n\geqslant 1.$
Then $|\ell_n|=n-2, |T_n|=n-1,$ and $\ell_n \cdot \sigma = \sgn(\sigma) \ell_n, T_n \cdot \sigma = \sgn(\sigma)T_n$ for any  $\sigma \in \s_n.$
We describe the elements of the dg operad $\RBLA_\infty$ using shuffle trees.
The generator $\ell_n (n\geqslant 2)$ is represented by a corolla with $n$ leaves and a black vertex, while the generator $T_n (n\geqslant 1)$ is represented by a corolla with $n$ leaves and a white vertex:
\begin{eqnarray}\label{eq: generator of RBLA infty}
	\begin{tikzpicture}[grow'=up,scale=0.6,baseline=(current bounding box.center) ]
		\tikzstyle{every node}=[thick,minimum size=6pt, inner sep=1pt]
		\node(r)[fill=black,circle,label=right:$\ell_n$]{}
		child{node(1){$1$}}
		child{node(i) { $i$}}
		child{node(n){$n$}};
\draw(r) -- (0,-0.5);
		\draw [line width=1pt,dotted] (-0.7,0.9)--(-0.1,0.9);
		\draw [line width=1pt,dotted] (0.1,0.9)--(0.7,0.9);
	\end{tikzpicture}
	\hspace{8mm}
	\begin{tikzpicture}[grow'=up,scale=0.6,baseline=(current bounding box.center) ]
		\tikzstyle{every node}=[thick,minimum size=6pt, inner sep=1pt]
		\node(r)[draw,circle,label=right:$T_n$]{}
		child{node(1){$1$}}
		child{node(i){$i$}}
		child{node(n){$n$}};
\draw(r) -- (0,-0.5);
		\draw [line width=1pt,dotted] (-0.7,0.9)--(-0.1,0.9);
		\draw [line width=1pt,dotted] (0.1,0.9)--(0.7,0.9);
\node(s) at (2,0){$.$};
	\end{tikzpicture}
\end{eqnarray}
The action of differential operator $\partial$ on generators can be expressed by shuffle trees as follows:
\begin{eqnarray}\label{Eq: partial l_n}
	\begin{tikzpicture}[scale=0.6,baseline=(current bounding box.center) ]
		\tikzstyle{every node}=[thick,minimum size=6pt, inner sep=1pt]
        \begin{scope}[yshift=-0.8cm]
        \node(a) at (-4,0.5){\begin{huge}$\partial$\end{huge}};
		\node (b0) [circle, fill=black, label=right:$\ell_n$]at (-2,-0.5)  {};
\draw(b0) -- (-2,-1);
		\node (b1) at (-3.5,1)  [minimum size=0pt,label=above:$1$]{};
		\node (b2) at (-2,1)  [minimum size=0pt,label=above:$i$]{};
		\node (b3) at (-0.5,1)  [minimum size=0pt,label=above:$n$]{};
		\draw        (b0)--(b1);
		\draw        (b0)--(b2);
		\draw        (b0)--(b3);
		\draw [dotted,line width=1pt] (-2.8,0.5)--(-2.2,0.5);
		\draw [dotted,line width=1pt] (-1.8,0.5)--(-1.2,0.5);
\end{scope}
		\begin{scope}[xshift=5cm]\node(0)at (0,-0.3){{\large$=\sum\limits_{j=2}^{n-1} \sum\limits_{i=1}^{n-j+1}\sum\limits_{\sigma
}(-1)^{i+j(n-i)}\sgn(\sigma)$}};\end{scope}
		\begin{scope}[xshift=12.5cm,scale=1.5]
\node(e0) at (0,-1.5)[circle, fill=black,label=right:\tiny$\ell_{n-j+1}$]{};
\draw(e0) -- (0,-2);
		\node(e1) at(-1.3,-0.2){};
		\node(e2-0) at (0,-0.75){{\tiny$i$}};
		\node(e3) at (1.3,-0.2){};
		\node(e2-1) at (0,0) [circle,fill=black,label=right: \tiny $\ell_j$]{};
		\node(e2-1-1) at (-1,1){};
		\node(e2-1-2) at (1, 1){};
		\draw [dotted,line width=1pt] (-0.6,-0.75)--(-0.2,-0.75);
		\draw [dotted,line width=1pt] (0.2,-0.75)--(0.6,-0.75);
		\draw [dotted,line width=1pt] (-0.4,0.5)--(0.4,0.5);
		\draw        (e0)--(e1);
		\draw         (e0)--(e3);
		\draw         (e0)--(e2-0);
		\draw         (e2-0)--(e2-1);
		\draw        (e2-1)--(e2-1-1);
		\draw        (e2-1)--(e2-1-2);
	\node(1) at(1.8,-0.3){\Large$\cdot \sigma$};
\end{scope}
\end{tikzpicture}
\end{eqnarray}

\begin{eqnarray}\label{Eq: partial T_n}
	\begin{tikzpicture}[scale=0.6,baseline=(current bounding box.center)]
		\tikzstyle{every node}=[thick,minimum size=6pt, inner sep=1pt]
		\begin{scope}[yshift=-0.5cm]\node(a) at (-4,0.5){\begin{huge}$\partial$\end{huge}};
		\node[circle, draw, label=right:$T_n$] (b0) at (-2,-0.5)  {};
\draw(b0) -- (-2,-1);
		\node (b1) at (-3.5,1)  [minimum size=0pt,label=above:$1$]{};
		\node (b2) at (-2,1)  [minimum size=0pt,label=above:$i$]{};
		\node (b3) at (-0.5,1)  [minimum size=0pt,label=above:$n$]{};
		\draw        (b0)--(b1);
		\draw        (b0)--(b2);
		\draw        (b0)--(b3);
		\draw [dotted,line width=1pt] (-2.85,0.5)--(-2.2,0.5);
		\draw [dotted,line width=1pt] (-1.8,0.5)--(-1.15,0.5);
\end{scope}
\begin{scope}[xshift=5cm]
		\node(0){{\large$= \sum\limits_{k=2}^n\sum\limits_{l_1+\cdots+l_k=n \atop l_1, \dots, l_k\geqslant 1}\sum\limits_{\sigma}(-1)^{\alpha'}\sgn(\sigma)$}};
\end{scope}
\begin{scope}[xshift=13cm,scale=1.5,yshift=0.2cm]
		\tikzstyle{every node}=[thick,minimum size=6pt, inner sep=1pt]
		\node(e0) at (0,-1.5)[circle, fill=black,label=right:\tiny$\ell_{k}$]{};
\draw(e0) -- (0,-2);
		\node(e1) at(-1.2,-0.3)[circle, draw, label=left:\tiny$T_{l_1}$]{};
		\node(e1-1) at(-2,0.8){};
		\node(e1-2) at (-1,0.8){};
		\node(e2-0) at (0,-0.7){{\tiny$i$}};
		\node(e3) at (1.2,-0.3)[draw, circle, label=right: \tiny$T_{l_k}$]{};
		\node(e3-1) at (1,0.8){};
		\node(e3-2) at (2,0.8){};
		\node(e2-1) at (0,0) [draw,circle,label=right: \tiny$T_{l_i}$]{};
		\node(e2-1-1) at (-0.5,1){};
		\node(e2-1-2) at (0.5, 1){};
		\draw [line width=1pt,dotted ] (-0.65,-0.7)--(-0.1,-0.7);
		\draw [line width=1pt,dotted ] (0.1,-0.7)--(0.65,-0.7);
		\draw [line width=1pt,dotted ] (-0.2,0.5)--(0.2,0.5);
		\draw [line width=1pt,dotted ] (-1.55,0.3)--(-1.15, 0.3);
		\draw [line width=1pt,dotted ] (1.15,0.3)--(1.55,0.3);
		\draw        (e0)--(e1);
		\draw         (e0)--(e3);
		\draw         (e1)--(e1-1);
		\draw          (e1)--(e1-2);
		\draw         (e0)--(e2-0);
		\draw         (e2-0)--(e2-1);
		\draw        (e2-1)--(e2-1-1);
		\draw        (e2-1)--(e2-1-2);
		\draw        (e3)--(e3-1);
		\draw        (e3)--(e3-2);
\end{scope}
\begin{scope}[xshift=16.5cm,yshift=-1.3cm]
	\node(0) at (0,0){};
	\node(1) at(0,1.35){\Large$\cdot \sigma$};
\end{scope}
\end{tikzpicture}
\\ \notag
\begin{tikzpicture}[scale=0.6,baseline=(current bounding box.center)]
		\tikzstyle{every node}=[thick,minimum size=6pt, inner sep=1pt]
	\node(0){{\large$+\sum\limits_{{\small\substack{2\leqslant p\leqslant n \\ 1\leqslant q\leqslant p
				}}}\sum\limits_{\small\substack{ r_1+\dots+r_q+p-q=n\\r_1, \dots, r_q\geqslant 1\\1\leqslant i\leqslant r_1\\1\leqslant k_1<\dots< k_{q-1}\leqslant p }}\sum\limits_{\sigma}(-1)^{\beta'}\sgn(\sigma)\lambda^{p-q}$}};
\begin{scope}[xshift=9cm,scale=1.5]
		\tikzstyle{every node}=[thick,minimum size=4pt, inner sep=1pt]
		\node(e0) at (0,-1.5)[circle, draw,label=right:{\tiny$\ T_{r_1}$}]{};
\draw(e0)--(0,-2);
		\node(e1) at(-1.8,-0.3){};
		\node(e1-1) at(-2,0.8){};
		\node(e1-2) at (-1,0.8){};
		\node(e2-0) at (0,-0.9){{\tiny$i$}};
		\node(e3) at (1.8,-0.3){};
		\node(e3-1) at (1,0.8){};
		\node(e3-2) at (2,0.8){};
		\node(e2-1) at (0,-0.3) [fill=black,draw,circle,label=right: {\tiny$\ \ell_p$}]{};
		\node(e2-1-1) at (-1.9,1){};
		\node(e2-1-2) at(-0.6,0.6){{\tiny$k_1$}};
		\node(e2-1-2-0) at (-1,1.2)[draw, circle,label=left:{\tiny $T_{r_2}$}]{};
		\node(e2-1-2-1) at (-1.4,1.9){};
		\node(e2-1-2-2) at (-0.6,1.9){};
		\node(e2-1-3) at(-0.4,1.5){};
		\node(e2-1-4) at (0.4,1.5){};
		\node(e2-1-5) at (0.6,0.6){{\tiny $k_{q-1}$}};
		\node(e2-1-5-0) at (1,1.2) [circle, draw,label=right: {\tiny $\ T_{r_q}$}]{};
		\node(e2-1-5-1) at (0.6,1.9){};
		\node(e2-1-5-2) at (1.4,1.9){};
		\node(e2-1-6) at (1.9, 1){};
		\draw [dotted,line width=1pt] (-0.75,-0.9)--(-0.1,-0.9);
		\draw [dotted,line width=1pt] (0.1,-0.9)--(0.8,-0.9);
		\draw [dotted, line width=1pt] (-1.2,0.6)--(-0.8,0.6);
		\draw [dotted, line width=1pt] (1.2,0.6)--(0.9,0.6);
		\draw [dotted, line width=1pt] (0.3,0.6)--(-0.45,0.6);
		\draw        (e0)--(e1);
		\draw [dotted, line width =1pt](-1.2,1.7)--(-0.8,1.7);
		\draw [dotted, line width=1pt](0.8,1.7)--(1.2,1.7);
		\draw         (e0)--(e3);
		\draw         (e0)--(e2-0);
		\draw         (e2-0)--(e2-1);
		\draw        (e2-1)--(e2-1-1);
		\draw      (e2-1-2)--(e2-1-2-0);
		\draw        (e2-1)--(e2-1-2);
		\draw        (e2-1)--(e2-1-3);
		\draw        (e2-1)--(e2-1-4);	
		\draw        (e2-1)--(e2-1-5);
		\draw         (e2-1-5)--(e2-1-5-0);
		\draw        (e2-1)--(e2-1-6);
		\draw (e2-1-2-0)--(e2-1-2-1);
		\draw (e2-1-2-0)--(e2-1-2-2);
		\draw (e2-1-5-0)--(e2-1-5-1);
		\draw (e2-1-5-0)--(e2-1-5-2);
\end{scope}
\begin{scope}[xshift=13cm,yshift=-0.5cm]
	\node(0) at (0,0){};
	\node(1) at(-0.5,1){\Large$\cdot \sigma$};
\end{scope}
\end{tikzpicture}
\end{eqnarray}
where all  $\sum_\sigma$ denote the summations over permutations  $\sigma \in \s_n$ that make the tree monomials into shuffle tree monomials, and the sign are given by
\begin{eqnarray}\label{Eq: sign   alpha'}
 	\alpha'&=&1+\frac{k(k-1)}{2}+\sum_{j=1}^k(k-j)l_j=1+\sum_{j=1}^k(k-j)(l_j-1),\\
 	\label{Eq: sign   beta'}\beta' &=& 1+i+\big(p+\sum\limits_{j=2}^q(r_j-1)\big)\big(r_1-i\big)+\sum\limits_{j=2}^q(r_j-1)(p-k_{j-1}).	\end{eqnarray}

\begin{remark}
It should be noted that the graded $\mathbb{S}$-module spanned by the shuffle trees in $\RBLA_\infty$ that only have black vertices is a symmetric dg suboperad of $\RBLA_\infty$, and this suboperad is exactly the $L_\infty$ operad.
\end{remark}

The following theorem is the main result of this section, whose proof occupies the rest of this section.
\begin{thm}\label{Thm: Minimal model}
	The dg operad $\RBLA_\infty$ is the minimal model of the operad $\RBLA$.
\end{thm}
In order to prove Theorem~\ref{Thm: Minimal model}, we are going to construct a quasi-isomorphism of symmetric dg operad $\RBLA_\infty \to \RBLA$ via algebraic Morse theory, where $\RBLA$ is considered as a symmetric dg operad concentrated in degree $0.$

First, according to the direct sum decomposition of $\RBLA_\infty$ given by the shuffle tree monomials, we obtain a weight quiver $Q=Q_{\RBLA_\infty}.$
Specifically, the vertex  set of $Q$ consist of the shuffle tree monomials of $\RBLA_\infty.$ For the shuffle tree monomial of $\RBLA_\infty$ with weight $1$, the arrows in $Q$ originating from them are as follows:
 \begin{itemize}
 \item[(i)] For each  $n\geqslant 3, 2\leq j\leq n-1, 1\leq i\leq n-j+1,$ and for each $ \sigma \in \s_n$ such that the target of the arrow is a shuffle tree monomial,
 
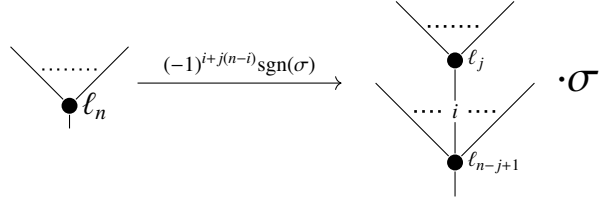
\begin{eqnarray*}
\begin{tikzpicture}[scale=0.6]
 \tikzstyle{every node}=[thick,minimum size=6pt, inner sep=1pt]
 \node(v1-1) at (4,-1) [circle, fill=black,label=right:$\ell_{n}$]{};
 \draw(v1-1) -- (4,-1.5);
 \node(v2-1) at (2.5,0.5){};
 \node(v2-2) at (5.5,0.5){};
 \draw(v1-1) to (v2-1);
  \draw(v1-1) to (v2-2);
  \draw[dotted,thick] (3.4, -0.2) to (4.6, -0.2);
  \draw[->] (5.5,-0.5) to node[midway,above=0.05]{\tiny$(-1)^{i+j(n-i)}\sgn(\sigma)$} (10,-0.5);
		\begin{scope}[xshift=12.5cm,scale=1.5]
\node(e0) at (0,-1.5)[circle, fill=black,label=right:\tiny$\ell_{n-j+1}$]{};
\draw(e0) -- (0,-2);
		\node(e1) at(-1.3,-0.2){};
		\node(e2-0) at (0,-0.75){{\tiny$i$}};
		\node(e3) at (1.3,-0.2){};
		\node(e2-1) at (0,0) [circle,fill=black,label=right: \tiny $\ell_j$]{};
		\node(e2-1-1) at (-1,1){};
		\node(e2-1-2) at (1, 1){};
		\draw [dotted,line width=1pt] (-0.6,-0.75)--(-0.2,-0.75);
		\draw [dotted,line width=1pt] (0.2,-0.75)--(0.6,-0.75);
		\draw [dotted,line width=1pt] (-0.4,0.5)--(0.4,0.5);
		\draw        (e0)--(e1);
		\draw         (e0)--(e3);
		\draw         (e0)--(e2-0);
		\draw         (e2-0)--(e2-1);
		\draw        (e2-1)--(e2-1-1);
		\draw        (e2-1)--(e2-1-2);
	\node(1) at(1.8,-0.3){\Large$\cdot \sigma$};
\end{scope}
 \end{tikzpicture}
 \end{eqnarray*}
 \item[(ii)] for each $n\geqslant 2, 2\leq k\leq n,$ $l_1+\dots+l_k=n,$ $ l_1,\dots,l_k\geq 1,$ and for each $ \sigma \in \s_n$ such that the target of the arrow is a shuffle tree monomial,
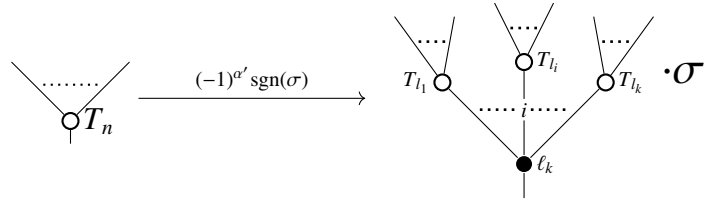
\begin{eqnarray*}
 && \begin{tikzpicture}[scale=0.6,baseline=(current bounding box.center) ]
 \tikzstyle{every node}=[thick,minimum size=6pt, inner sep=1pt]
 \node(v1-1) at (3,-1) [circle, draw,label=right:$T_{n}$]{};
 \draw(v1-1) -- (3,-1.5);
 \node(v2-1) at (1.5,0.5){};
 \node(v2-2) at (4.5,0.5){};
 \draw(v1-1) to (v2-1);
  \draw(v1-1) to (v2-2);
  \draw[dotted,thick] (2.4, -0.2) to (3.6, -0.2);
  \draw[->] (4.5,-0.5) to node[midway,above=0.05]{\tiny $(-1)^{\alpha'}\sgn(\sigma)$} (9.5,-0.5);
\begin{scope}[xshift=13cm,scale=1.5,yshift=0.2cm]
		\tikzstyle{every node}=[thick,minimum size=6pt, inner sep=1pt]
		\node(e0) at (0,-1.5)[circle, fill=black,label=right:\tiny$\ell_{k}$]{};
\draw(e0) -- (0,-2);
		\node(e1) at(-1.2,-0.3)[circle, draw, label=left:\tiny$T_{l_1}$]{};
		\node(e1-1) at(-2,0.8){};
		\node(e1-2) at (-1,0.8){};
		\node(e2-0) at (0,-0.7){{\tiny$i$}};
		\node(e3) at (1.2,-0.3)[draw, circle, label=right: \tiny$T_{l_k}$]{};
		\node(e3-1) at (1,0.8){};
		\node(e3-2) at (2,0.8){};
		\node(e2-1) at (0,0) [draw,circle,label=right: \tiny$T_{l_i}$]{};
		\node(e2-1-1) at (-0.5,1){};
		\node(e2-1-2) at (0.5, 1){};
		\draw [line width=1pt,dotted ] (-0.65,-0.7)--(-0.1,-0.7);
		\draw [line width=1pt,dotted ] (0.1,-0.7)--(0.65,-0.7);
		\draw [line width=1pt,dotted ] (-0.2,0.5)--(0.2,0.5);
		\draw [line width=1pt,dotted ] (-1.55,0.3)--(-1.15, 0.3);
		\draw [line width=1pt,dotted ] (1.15,0.3)--(1.55,0.3);
		\draw        (e0)--(e1);
		\draw         (e0)--(e3);
		\draw         (e1)--(e1-1);
		\draw          (e1)--(e1-2);
		\draw         (e0)--(e2-0);
		\draw         (e2-0)--(e2-1);
		\draw        (e2-1)--(e2-1-1);
		\draw        (e2-1)--(e2-1-2);
		\draw        (e3)--(e3-1);
		\draw        (e3)--(e3-2);
\end{scope}
\begin{scope}[xshift=16.5cm,yshift=-1.3cm]
	\node(0) at (0,0){};
	\node(1) at(0,1.35){\Large$\cdot \sigma$};
\end{scope}
 \end{tikzpicture}
 \end{eqnarray*}
 \item[(iii)] for each $n\geqslant 2, 2\leq p \leq n,$ $1\leq q\leq p,$ $r_1+\dots+r_q+p-q=n,$ $r_1,\dots, r_q\geq 1,$ $1\leq i \leq r_1, 1\leq k_1<\dots < k_{q-1} \leq p,$ and for each $ \sigma \in \s_n$ such that the target of the arrow is a shuffle tree monomial,
 \begin{eqnarray*}
 && 
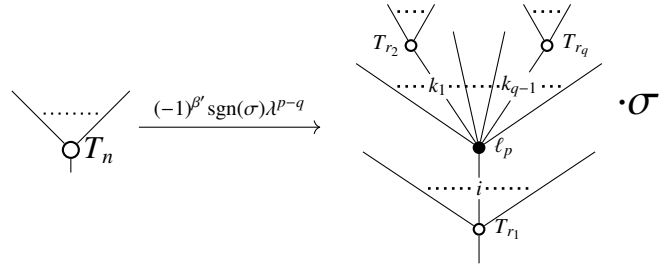
\begin{tikzpicture}[scale=0.6,baseline=(current bounding box.center) ]
 \tikzstyle{every node}=[thick,minimum size=6pt, inner sep=1pt]
 \begin{scope}[yshift=0.5cm]
 \node(v1-1) at (0,-1) [circle, draw,label=right:$T_{n}$]{};
 \draw(v1-1) -- (0,-1.5);
 \node(v2-1) at (-1.5,0.5){};
 \node(v2-2) at (1.5,0.5){};
 \draw(v1-1) to (v2-1);
  \draw(v1-1) to (v2-2);
  \draw[dotted,thick] (-0.6, -0.2) to (0.6, -0.2);
  \draw[->] (1.5,-0.5) to node[midway,above=0.05]{\tiny $(-1)^{\beta'}\sgn(\sigma)\lambda^{p-q}$} (5.5,-0.5);
  \end{scope}
\begin{scope}[xshift=9cm,scale=1.5]
		\tikzstyle{every node}=[thick,minimum size=4pt, inner sep=1pt]
		\node(e0) at (0,-1.5)[circle, draw,label=right:{\tiny$\ T_{r_1}$}]{};
\draw(e0)--(0,-2);
		\node(e1) at(-1.8,-0.3){};
		\node(e1-1) at(-2,0.8){};
		\node(e1-2) at (-1,0.8){};
		\node(e2-0) at (0,-0.9){{\tiny$i$}};
		\node(e3) at (1.8,-0.3){};
		\node(e3-1) at (1,0.8){};
		\node(e3-2) at (2,0.8){};
		\node(e2-1) at (0,-0.3) [fill=black,draw,circle,label=right: {\tiny$\ \ell_p$}]{};
		\node(e2-1-1) at (-1.9,1){};
		\node(e2-1-2) at(-0.6,0.6){{\tiny$k_1$}};
		\node(e2-1-2-0) at (-1,1.2)[draw, circle,label=left:{\tiny $T_{r_2}$}]{};
		\node(e2-1-2-1) at (-1.4,1.9){};
		\node(e2-1-2-2) at (-0.6,1.9){};
		\node(e2-1-3) at(-0.4,1.5){};
		\node(e2-1-4) at (0.4,1.5){};
		\node(e2-1-5) at (0.6,0.6){{\tiny $k_{q-1}$}};
		\node(e2-1-5-0) at (1,1.2) [circle, draw,label=right: {\tiny $\ T_{r_q}$}]{};
		\node(e2-1-5-1) at (0.6,1.9){};
		\node(e2-1-5-2) at (1.4,1.9){};
		\node(e2-1-6) at (1.9, 1){};
		\draw [dotted,line width=1pt] (-0.75,-0.9)--(-0.1,-0.9);
		\draw [dotted,line width=1pt] (0.1,-0.9)--(0.8,-0.9);
		\draw [dotted, line width=1pt] (-1.2,0.6)--(-0.8,0.6);
		\draw [dotted, line width=1pt] (1.2,0.6)--(0.9,0.6);
		\draw [dotted, line width=1pt] (0.3,0.6)--(-0.45,0.6);
		\draw        (e0)--(e1);
		\draw [dotted, line width =1pt](-1.2,1.7)--(-0.8,1.7);
		\draw [dotted, line width=1pt](0.8,1.7)--(1.2,1.7);
		\draw         (e0)--(e3);
		\draw         (e0)--(e2-0);
		\draw         (e2-0)--(e2-1);
		\draw        (e2-1)--(e2-1-1);
		\draw      (e2-1-2)--(e2-1-2-0);
		\draw        (e2-1)--(e2-1-2);
		\draw        (e2-1)--(e2-1-3);
		\draw        (e2-1)--(e2-1-4);	
		\draw        (e2-1)--(e2-1-5);
		\draw         (e2-1-5)--(e2-1-5-0);
		\draw        (e2-1)--(e2-1-6);
		\draw (e2-1-2-0)--(e2-1-2-1);
		\draw (e2-1-2-0)--(e2-1-2-2);
		\draw (e2-1-5-0)--(e2-1-5-1);
		\draw (e2-1-5-0)--(e2-1-5-2);
\end{scope}
\begin{scope}[xshift=13cm,yshift=-0.5cm]
	\node(0) at (0,0){};
	\node(1) at(-0.5,1){\Large$\cdot \sigma$};
\end{scope}
 \end{tikzpicture}
 \end{eqnarray*}
 where $\alpha'$ and $\beta'$ are given by  \eqref{Eq: sign   alpha'} and \eqref{Eq: sign   beta'}.
 \end{itemize}
Furthermore, $Q$ contains additional arrows where
 \begin{itemize}
 \item[(i)] the source is a shuffle tree monomial of weight $\geqslant 2$,
 \item[(ii)] the target is obtained by expanding an internal vertex of the source according to the aforementioned arrow construction,
 \item[(iii)] the weight of this arrow equals the weight of its corresponding internal vertex expansion (as defined above) multiplied by $-1$ raised to the power of the sum of degrees of all internal vertices preceding this internal  vertex under the planar order.
\end{itemize}

Next, we will construct a partial matching $\M$ on $Q$. To proceed, we  introduce the following notions.
Let $\calf(M)$ be a free symmetric operad.  Fix a basis $B_M$ of $M$ and a monomial order $\sqsubset$ on the free monoid $B_M^{*}.$
For each shuffle tree monomial $\calt,$ it determines a path sequence $\pth(\calt)$  \cite[Chapter 5]{BD}.
Specifically, for any  shuffle tree monomial $\calt$ of arity $n$ (with $n$ leaves), there corresponds a sequence  $\pth(\calt) = (x_1, \dots, x_n),$ where each $x_i$ is obtained by multiplying (in order) all decorations on internal vertices along the unique path from the root of $\calt$ to the leaf labeled  $i$.
The following definition is a slight modification of the path-permutation extension of $\sqsubset$ introduced in \cite{BD}.
\begin{defn}\label{Def: right path-permutation order} Let $\sqsubset$ be the aforementioned monomial order on $B_M^*.$ The \textbf{right path-permutation extension}   $\prec$ of $\sqsubset$ on shuffle tree monomials of $\calf(M)$ is defined as follows:
 \begin{itemize}
 	\item[(i)] If for two shuffle tree monomial $\calt$ and $\calt',$ the number of leaves of $\calt$ is less than the number of leaves  of $\calt',$ we put $\calt \prec \calt';$
 	\item[(ii)] if $\calt$ and $\calt'$ have the same numbers of leaves, we compare the path sequences $\pth(\calt)$ and $\pth(\calt')$ from \textbf{right to left} according to the monomial order $\sqsubset$;
 	\item[(iii)] if  $\calt$ and $\calt'$ have the same numbers of leaves and $\pth(\calt) = \pth(\calt'),$  we compare the labels of the leaves of $\calt$ and $\calt'$ from right to left according to the planar order.
 \end{itemize}
\end{defn}

We now consider the right path-permutation extension $\prec$ of the monomial order generated by
\begin{equation}\label{eq: monomial order of l and T}
T_1 \sqsubset \ell_2 \sqsubset T_2 \sqsubset \ell_3 \sqsubset \cdots \sqsubset T_n \sqsubset \ell_{n+1} \sqsubset \cdots
\end{equation}
in $\RBLA_{\infty}.$
Under this order $\prec$, the leading term  of $\partial(\ell_{n+1}), \partial(T_n)$ (for $n\geq 2$) are the following shuffle tree monomials, respectively:
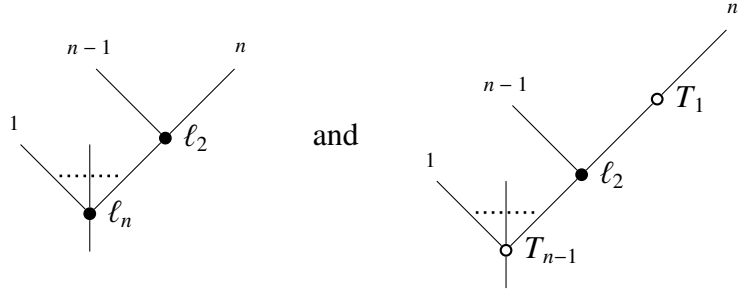
\begin{eqnarray*}
	\begin{tikzpicture}[scale=1,baseline=(current bounding box.center) ]
		\tikzstyle{every node}=[thick,minimum size=4pt, inner sep=1pt]
		\node(1) at (0,0) [draw, circle, fill=black, label=right:$\ \ell_{n}$]{};
\draw(1) -- (0,-0.5);
		\node(2-3) at (1,1) [draw, circle, fill=black, label=right: $\ \ell_2$]{};
		\node(3-1) at (0,2){};
		\node(3-2) at (2,2){};
		\node(2-2) at (0,1){};
		\node(2-1) at (-1,1){};
        \node(lab1) at (-1,1.2) {\tiny$1$};
        \node(lab2) at (0,2.2) {\tiny$n-1$};
        \node(lab3) at (2,2.2) {\tiny $n$};
		\draw (1)--(2-1);
		\draw  (1)--(2-2);
		\draw  (1)--(2-3);
		\draw (2-3)--(3-1);
		\draw (2-3)--(3-2);
		\draw [dotted,line width=1pt](-0.4,0.5)--(0.4,0.5);
	\end{tikzpicture}
\quad \quad  \text{and} \quad \quad
\begin{tikzpicture}[scale=1,baseline=(current bounding box.center) ]
	\tikzstyle{every node}=[thick,minimum size=4pt, inner sep=1pt]
	\node(1) at (0,0) [draw, circle, label=right:$\ T_{n-1}$]{};
\draw(1) -- (0,-0.5);
	\node(2-1) at (1,1) [draw, circle, fill=black, label=right: $\ \ell_2$]{};
	\node(3-1) at (2,2)[draw, circle, label=right: $\ T_1$]{};
	\node(3-2) at (0,2){};
	\node(2-2) at (0,1){};
	\node(2-3) at (-1,1){};
	\node(4) at (3,3){};
    \node(lab1) at (-1, 1.2) {\tiny$1$};
    \node(lab2) at (0,2.2) {\tiny$n-1$};
    \node(lab3) at (3,3.2) {\tiny$n$};
	\draw (1)--(2-1);
	\draw  (1)--(2-2);
	\draw  (1)--(2-3);
	\draw (2-1)--(3-1);
	\draw (2-1)--(3-2);
	\draw (3-1)--(4);
	\draw [dotted,line width=1pt](-0.4,0.5)--(0.4,0.5);
\end{tikzpicture}
\end{eqnarray*}

Let $\mathcal{S}$ be a generator of degree $\geqslant 1$ in $ \RBLA_\infty$. Denote the leading monomial of $\partial \mathcal{S}$ by $\widehat{\mathcal{S}}.$ 
A shuffle tree monomial of the form $\widehat{\mathcal{S}}$ is called \textbf{typical}.


 \begin{defn}\label{Def: efficient tree monomials}  A shuffle tree monomial $\mathcal{T}$ in $\RBLA_\infty$ is called \textbf{effective} if there exists a typical divisor  $\mathcal{T}'$ of $\mathcal{T}$ satisfies the following conditions:
 \begin{itemize}
 \item[(i)] Under the planar order of $\calt$, there exists no internal vertex decorated by a nonzero-degree element that succeeds the first internal vertex of $\calt';$
 \item[(ii)] under the planar order, no set of  internal vertex in $\calt$ that lie after the first internal vertex of $\calt'$ forms a typical divisor of $\calt$.
 \end{itemize}
The typical divisor $\mathcal{T}'$ is called the \textbf{effective divisor} of $\mathcal{T}$.
\end{defn}
Intuitively, the effective divisor of $\calt$  is the ``top-rightmost" typical divisor of $\calt$ under the planar order.
\begin{exam} Consider the following two shuffle tree monomials whose leaf labels are in increasing order according to the planar order:
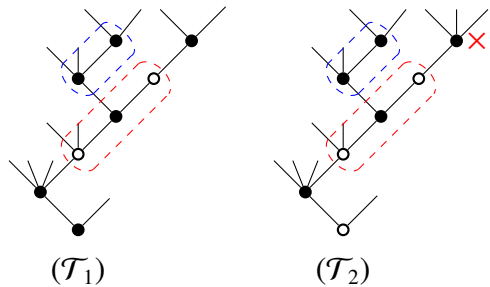
\begin{eqnarray*}
\begin{tikzpicture}
	\tikzstyle{every node}=[thick,minimum size=4pt, inner sep=1pt]
	\node(1)at (0,0)[circle, draw, fill=black]{};
	\node(2-1) at (0.5,0.5){};
	\node(2-2) at (-0.5,0.5)[circle,draw, fill=black]{};
	\node(3-1) at (0,1)[circle, draw]{};
	\node(3-2) at (-0.7,1){};
	\node(3-3) at (-0.3,1){};
	\node(3-4) at (-1,1){};
	\node(4-1) at (0.5,1.5)[circle, draw,fill=black]{};
	\node (4-2) at(0,1.5){};
	\node(4-3) at (-0.5,1.5){};
	\node(5-1) at (1,2)[circle, draw]{};
	\node(5-2) at(0,2)[circle, draw, fill=black] {};
	\node (6-1) at (1.5, 2.5)[circle,draw, fill=black]{};
	\node(6-2) at (0.5, 2.5)[circle, draw, fill=black]{};
	\node(6-3) at(0,2.5){};
	\node(6-4) at (-0.5,2.5){};
	\node(7-1) at (2,3)[minimum size=0pt]{};
	\node(7-2) at (1,3){};
	\node(7-3) at(0.9,3){};
	\node(7-4) at (0,3){};
	\draw (1)--(2-1);
	\draw (1)--(2-2);
	\draw (2-2)--(3-1);
	\draw (2-2)--(3-2);
	\draw (2-2)--(3-3);
	\draw (2-2)--(3-4);
	\draw (3-1)--(4-1);
	\draw (3-1)--(4-2);
	\draw (3-1)--(4-3);
	\draw (4-1)--(5-1);
	\draw (4-1)--(5-2);
	\draw (5-1)--(6-1);
	\draw (6-1)--(7-1);
	\draw (6-1)--(7-2);
	\draw (5-2)--(6-2);
	\draw (5-2)--(6-3);
	\draw (5-2)--(6-4);
	\draw(6-2)--(7-3);
	\draw(6-2)--(7-4);
\draw[dashed,red] (0.8, 2.2) to [in=60, out=30] (1.2, 1.8);
\draw[dashed,red] (0.8,2.2) to (-0.2, 1.2);
\draw[dashed,red] (-0.2, 1.2) to [in=-150,out=-120] (0.2,0.8);
\draw[dashed,red] (0.2,0.8) to (1.2, 1.8);
\draw[dashed,blue](0.3,2.7) to [in=30, out=60] (0.7,2.3) ;
\draw[dashed,blue](0.3,2.7)to(-0.2,2.2);
\draw[dashed,blue] (0.2,1.8)to [in=-120, out=-150] (-0.2,2.2);
	\node[minimum size=0pt,inner sep=0pt,label=below:$(\mathcal{T}_1)$] (name) at (0,-0.3){};
	\draw[dashed, blue](0.2,1.8)to(0.7,2.3);
\end{tikzpicture}
\hspace{4mm}
\begin{tikzpicture}
	\tikzstyle{every node}=[thick,minimum size=4pt, inner sep=1pt]
	\node(1)at (0,0)[circle, draw]{};
	\node(2-1) at (0.5,0.5)[minimum size=0pt]{};
	\node(2-2) at (-0.5,0.5)[circle,draw, fill=black]{};
	\node(3-1) at (0,1)[circle, draw]{};
	\node(3-2) at (-0.3,1){};
	\node(3-3) at (-0.7,1){};
	\node(3-4) at (-1,1){};
	\node(4-1) at (0.5,1.5)[circle, draw,fill=black]{};
	\node (4-2) at(0,1.5){};
	\node(4-3) at (-0.5,1.5){};
	\node(5-1) at (1,2)[circle, draw]{};
	\node(5-2) at(0,2)[circle, draw, fill=black] {};
	\node (6-1) at (1.5, 2.5)[circle,draw, fill=black,label=right:$\red\times$]{};
	\node(6-2) at (0.5, 2.5)[circle, draw, fill=black]{};
	\node(6-3) at(0,2.5){};
	\node(6-4) at (-0.5,2.5){};
	\node(7-1) at (2,3){};
	\node(7-2) at (1,3){};
	\node(7-3) at (0.9,3){};
	\node(7-4) at(0,3){};
    \node(7-5) at (1.5,3){};
	\draw(6-2)--(7-3);
	\draw(6-2)--(7-4);
	\draw (1)--(2-1);
	\draw (1)--(2-2);
	\draw (2-2)--(3-1);
	\draw (2-2)--(3-2);
	\draw (2-2)--(3-3);
	\draw (2-2)--(3-4);
	\draw (3-1)--(4-1);
	\draw (3-1)--(4-2);
	\draw (3-1)--(4-3);
	\draw (4-1)--(5-1);
	\draw (4-1)--(5-2);
	\draw (5-1)--(6-1);
	\draw (6-1)--(7-1);
	\draw (6-1)--(7-2);
	\draw (5-2)--(6-2);
	\draw (5-2)--(6-3);
	\draw (5-2)--(6-4);
    \draw(6-1)--(7-5);
    \draw[dashed,red] (0.8, 2.2) to [in=60, out=30] (1.2, 1.8);
\draw[dashed,red] (0.8,2.2) to (-0.2, 1.2);
\draw[dashed,red] (-0.2, 1.2) to [in=-150,out=-120] (0.2,0.8);
\draw[dashed,red] (0.2,0.8) to (1.2, 1.8);
    \draw[dashed,blue](0.3,2.7) to [in=30, out=60] (0.7,2.3) ;
\draw[dashed,blue](0.3,2.7)--(-0.2,2.2);
\draw[dashed,blue] (0.2,1.8)to [in=-120, out=-150] (-0.2,2.2);
	\draw[dashed, blue](0.2,1.8)--(0.7,2.3);
	\node[minimum size=0pt,inner sep=0pt,label=below:$(\mathcal{T}_2)$] (name) at (0,-0.3){};
\end{tikzpicture}
\end{eqnarray*}

For the two tree monomials shown above, each contains two typical divisors (indicated by blue and red dashed circles).
\begin{itemize}
	\item $\mathcal{T}_1$ is  effective with its effective divisor marked by the blue dashed circle. The red-dashed typical divisor $\calt_1'$ is not effective because, under the planar order of $\calt_1,$ the set of the internal vertices of $\calt$ strictly following the first internal vertex of $\calt_1'$ fully contains the internal  vertices of the blue-dashed typical divisor, thereby violating Condition (ii) in Definition \ref{Def: efficient tree monomials};
	\item $\mathcal{T}_2$ is not effective because, among the internal vertices succeeding the first internal  vertex of the blue-dashed typical divisor, the last vertex has positive degree, violating Condition (i) in Definition \ref{Def: efficient tree monomials}.
\end{itemize}
\end{exam}

Let $\mathcal{T}$ be an effective shuffle  tree monomial with effective divisor $\widehat{\mathcal{S}}$.
We denote by $m_{\widehat{\calS},\calS}(\calt)$ the shuffle tree monomial obtained by replacing the effective divisor $\widehat{\calS}$ of $\calt$ with $\calS.$
Observe that since the leaf labels of $\widehat{\calS}$ and $\calS$ are strictly increasing with respect to the planar order, the above definition is unambiguous.
Now, we define the set of arrows $\M$ of $Q$ as follows.
$$\M : = \Bigg\{m_{\widehat{\calS},\calS}(\calt) \xrightarrow{\pm 1}\mathcal{T} \Bigg|
 \begin{array}{l}
 \text{$\mathcal{T}$ is effective with effective divisor $\widehat{\mathcal{S}}$}
 \end{array} \Bigg\} $$

\begin{lem}
$\M$ is a partial matching.
\end{lem}

\begin{proof}
Clearly, the weight of every arrow in $\M$ is invertible (being $\pm 1$).
Therefore, we only need to verify that no distinct pair of arrows in $\M$ satisfy any of the following three conditions.
$$
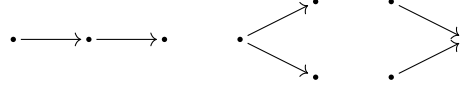
\begin{tikzpicture}
\fill(-1,0) circle (1pt);
\fill(0,0) circle (1pt);
\fill(1,0) circle (1pt);
\draw[->](-0.9,0)--(-0.1,0);
\draw[->](0.1,0)--(0.9,0);
\begin{scope}[xshift=2cm]
\fill(0,0) circle (1pt);
\fill(1,0.5) circle (1pt);
\fill(1,-0.5) circle (1pt);
\draw[->](0.1,0.05)--(0.9,0.45);
\draw[->](0.1,-0.05)--(0.9,-0.45);
\end{scope}
\begin{scope}[xshift=4cm]
\fill(0,0.5) circle (1pt);
\fill(0,-0.5) circle (1pt);
\fill(1,0) circle (1pt);
\draw[->](0.1,0.45)--(0.9,0.05);
\draw[->](0.1,-0.45)--(0.9,-0.05);
\end{scope}
\end{tikzpicture}$$

For the first case, suppose there exist two distinct  arrows $\alpha_1$ and $\alpha_2$ in $\M$ of the following form:
$$\begin{tikzpicture}
\node(1) at (-2,0) {$\mathcal{T}$};
\node(2) at (0,0) {$\mathcal{T}'$};
\node(3) at (2,0) {$\mathcal{T}''$};
\draw[->] (1) -- (2) node[midway,above=0.05]{$\alpha_1$};
\draw[->] (2) -- (3) node[midway,above=0.05]{$\alpha_2$};
\end{tikzpicture}.$$
Then both $\calt'$ and $\calt''$ are effective shuffle tree monomials, with effective divisors $\widehat{\calS'}$ and $\widehat{\calS''}$, respectively.
Therefore, we obtain the identities:
$m_{\widehat{\calS''},\calS''}(\calt'') = \calt'$ and $m_{\widehat{\calS'},\calS'}(\calt') = \calt.$
Under the planar order, let
\begin{itemize}
\item $\calS'$ be the $i'$-th internal vertex in $\calt$;
\item $\calS''$ be the $i''$-th internal vertex in $\calt'.$
\end{itemize}
By the effectiveness of  $\mathcal{T}''$ and Definition~\ref{Def: efficient tree monomials} (ii), it follows that $i''\geq i'.$
On the other hand, since $\calt'$ is effective,  and the internal vertex $\calS''$ in $\calt'$ has positive degree, Definition~\ref{Def: efficient tree monomials} (i) yields $i'\geqslant i''.$
We therefore conclude that $i' = i''.$
This shows that $\calt''$ contains the following divisor, where the effective divisor of $\calt''$ is encircled by a blue dashed line.
$$
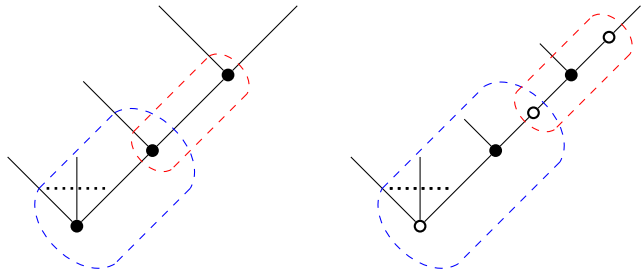
\begin{tikzpicture}
\tikzstyle{every node}=[thick,minimum size=4pt, inner sep=1pt]
\node(1-1) at (0,0)[circle,draw, fill=black]{};
\node(2-1) at (-1,1) {};
\node(2-2) at (0,1) {};
\node(2-3) at (1,1) [circle,draw, fill=black]{};
\node(3-1) at (0,2) {};
\node(3-2) at (2,2) [circle,draw, fill=black]{};
\node(4-1) at (1,3) {};
\node(4-2) at (3,3) {};
\draw (1-1)--(2-1);
\draw (1-1)--(2-2);
\draw (1-1)--(2-3);
\draw (2-3)--(3-1);
\draw (2-3)--(3-2);
\draw (3-2)--(4-1);
\draw (3-2)--(4-2);
\draw [dotted,line width=1pt](-0.4,0.5)--(0.4,0.5);
\draw [dashed,blue] (-0.5,0.5) to (0.5,1.5);
\draw [dashed,blue] (0.5,1.5) to [in=60,out=30] (1.5,0.5);
\draw [dashed,blue] (1.5,0.5) to (0.5,-0.5);
\draw [dashed,blue] (0.5,-0.5) to [in=-120,out=-150](-0.5,0.5);
\draw [dashed,red] (0.75,1.25) to (1.75,2.25);
\draw [dashed,red] (1.75,2.25) to [in=60,out=30](2.25,1.75);
\draw [dashed,red] (2.25,1.75) to (1.25,0.75);
\draw [dashed,red] (1.25,0.75) to [in=-120,out=-150](0.75,1.25);
\end{tikzpicture}
\hspace{4mm}
\begin{tikzpicture}
\tikzstyle{every node}=[thick,minimum size=4pt, inner sep=1pt]
\node(1-1) at (0,0)[circle,draw]{};
\node(2-1) at (-1,1) {};
\node(2-2) at (0,1) {};
\node(2-3) at (1,1) [circle,draw, fill=black]{};
\node(3-1) at (0.5,1.5) {};
\node(3-3) at (1.5,1.5)[circle, draw]{};
\node(3-2) at (2,2) [circle,draw, fill=black]{};
\node(4-1) at (1.5,2.5) {};
\node(4-2) at (3,3) {};
\node(4-3) at (2.5,2.5)[circle, draw]{};
\draw (1-1)--(2-1);
\draw (1-1)--(2-2);
\draw (1-1)--(2-3);
\draw (2-3)--(3-1);
\draw (2-3)--(3-3);
\draw (3-2)--(3-3);
\draw (3-2)--(4-1);
\draw (3-2)--(4-3);
\draw (4-3)--(4-2);
\draw [dotted,line width=1pt](-0.4,0.5)--(0.4,0.5);
\draw [dotted,line width=1pt](-0.4,0.5)--(0.4,0.5);
\draw [dashed,blue] (-0.5,0.5) to (0.85,1.85);
\draw [dashed,blue] (0.85,1.85) to [in=60,out=30] (1.85,0.85);
\draw [dashed,blue] (1.85,0.85) to (0.5,-0.5);
\draw [dashed,blue] (0.5,-0.5) to [in=-120,out=-150](-0.5,0.5);
\begin{scope}[xshift=15,yshift=15]\draw [dashed,red] (0.75,1.25) to (1.75,2.25);
\draw [dashed,red] (1.75,2.25) to [in=60,out=30](2.25,1.75);
\draw [dashed,red] (2.25,1.75) to (1.25,0.75);
\draw [dashed,red] (1.25,0.75) to [in=-120,out=-150](0.75,1.25);
\end{scope}
\end{tikzpicture}
$$
However, $\calt''$ additionally contains a typical divisor marked by a red dashed enclosure, which contradicts Definition~\ref{Def: efficient tree monomials} (ii). Consequently, the case of both $\alpha_1$ and $\alpha_2$ existing simultaneously is impossible.

For the second case, suppose there exist two distinct  arrows $\alpha_1$ and $\alpha_2$ in $\M$ of the following form:
$$
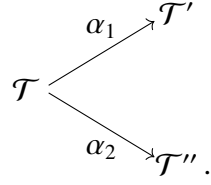
\begin{tikzpicture}
\node(1) at (0,0) {$\mathcal{T}$};
\node(2) at (2,1) {$\mathcal{T}'$};
\node(3) at (2,-1) {$\mathcal{T}''$};
\draw[->](0.3,0.05)--(1.7,0.9)node[midway,above=0.05]{$\alpha_1$};
\draw[->](0.3,-0.05)--(1.7,-0.9)node[midway,below=0.05]{$\alpha_2$};
\node(s) at (2.4,-1.1){$.$};
\end{tikzpicture}$$
Then both $\calt'$ and $\calt''$ are effective shuffle tree monomials, with effective divisors $\widehat{\calS'}$ and $\widehat{\calS''}$, respectively.
Therefore, we obtain the identities:
$m_{\widehat{\calS''},\calS''}(\calt'') = \calt = m_{\widehat{\calS'},\calS'}(\calt').$
By Definition~\ref{Def: efficient tree monomials} (ii), we conclude that $\calS'$ and $\calS''$ coincide as the same internal vertex in $\calt,$ which implies $\alpha_1 = \alpha_2.$  This contradicts the  assumption that $\alpha_1 \neq \alpha_2.$

For the third case, suppose there exist two distinct  arrows $\alpha_1$ and $\alpha_2$ in $\M$ of the following form:
$$
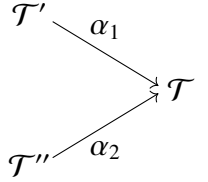
\begin{tikzpicture}
\node(1) at (0,1) {$\mathcal{T}'$};
\node(2) at (0,-1) {$\mathcal{T}''$};
\node(3) at (2,0) {$\mathcal{T}$};
\draw[->](0.3,0.9)--(1.7,0.05)node[midway,above=0.05]{$\alpha_1$};
\draw[->](0.3,-0.9)--(1.7,-0.05)node[midway,below=0.05]{$\alpha_2$};
\node(s) at (2.4,-1.1){$.$};
\end{tikzpicture}$$
Consequently, $\calt$ possesses a unique effective divisor $\calS$. Combined with the definition of arrows in $\M,$ this immediately yields $\alpha_1 = \alpha_2,$ which contradicts the  assumption that $\alpha_1 \neq \alpha_2.$
\end{proof}

\begin{lem}
$\M$ is a Morse matching.
\end{lem}

\begin{proof}
Given a shuffle tree monomial $\calt$ of $\RBLA_\infty,$ we define the shuffle tree monomial $D(\mathcal{T})$ as the result of replacing:
\begin{itemize}
\item  every divisor corresponding to an internal vertex decorated by $\ell_n$ in $\calt$ with
$$
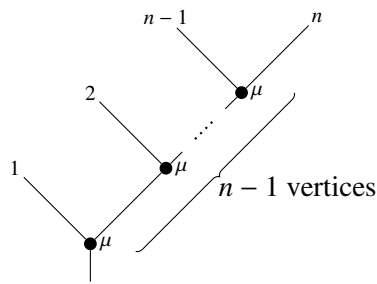
\begin{tikzpicture}
\tikzstyle{every node}=[thick,minimum size=4pt, inner sep=1pt]
\node(1-1) at (0,0)[circle,draw, fill=black,label=right:\tiny$\mu$]{};
\draw(1-1) -- (0,-0.5);
\node(2-1) at (-1,1) {\tiny$1$};
\node(2-3) at (1,1) [circle,draw, fill=black,label=right:\tiny$\mu$]{};
\node(3-1) at (0,2) {\tiny$2$};
\node(3-2) at (2,2) [circle,draw, fill=black,label=right:\tiny$\mu$]{};
\node(4-1) at (1,3) {\tiny$n-1$};
\node(4-2) at (3,3) {\tiny$n$};
\node(3-3) at (1.3, 1.3){};
\node(3-4) at (1.7,1.7){};
\draw (1-1)--(2-1);
\draw (1-1)--(2-3);
\draw (2-3)--(3-1);
\draw (2-3)--(3-3);
\draw (3-4)--(3-2);
\draw (3-2)--(4-1);
\draw (3-2)--(4-2);
\draw [dotted,thick](3-3)--(3-4);
\draw[decorate,decoration={brace,mirror}] (0.6,-0.1) to node[bend right,auto,swap]{\small$n-1$ vertices}(2.7,2);
\end{tikzpicture}$$
\item every divisor corresponding to an internal vertex decorated by $\ell_n$ in $\calt$ with
$$
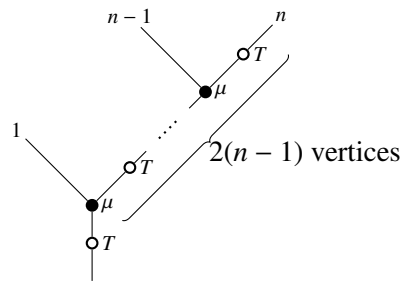
\begin{tikzpicture}
\tikzstyle{every node}=[thick,minimum size=4pt, inner sep=1pt]
\node(1-1) at (0.5,0)[circle,draw,label=right:\tiny$T$]{};
\draw(1-1) --(0.5,-0.5);
\node(2-3) at (0.5,0.5) [circle,draw, fill=black,label=right:\tiny$\mu$]{};
\node(3-1) at (-0.5,1.5) {\tiny$1$};
\node(3-3) at (1,1)[circle, draw,label=right:\tiny$T$]{};
\node(3-2) at (2,2) [circle,draw, fill=black,label=right:\tiny$\mu$]{};
\node(4-1) at (1,3) {\tiny$n-1$};
\node(4-2) at (3,3) {\tiny$n$};
\node(4-3) at (2.5,2.5)[circle, draw,label=right:\tiny$T$]{};
\node(3-4) at (1.3,1.3){};
\node(3-5) at (1.7,1.7){};
\draw (1-1)--(2-3);
\draw (2-3)--(3-1);
\draw (2-3)--(3-3);
\draw (3-2)--(3-5);
\draw (3-4)--(3-3);
\draw (3-2)--(4-1);
\draw (3-2)--(4-3);
\draw (4-3)--(4-2);
\draw[dotted,thick] (3-4)--(3-5);
\draw[decorate,decoration={brace,mirror}] (0.9,0.3) to node[bend right,auto,swap]{\small$2(n-1)$ vertices}(3.1,2.5);
\end{tikzpicture}.$$
\end{itemize}
By direct computation, we conclude that for each  generator $\calS,$  the equality $D(\widehat{\calS}) = D(\calS)$ holds. Moreover, for every shuffle tree monomial $\calS'$ appearing in $\partial(\calS)$ with a non-zero coefficient, we have
$D(\calS')  \preceq D(\calS),$ where $\prec$   is the monomial order defined immediately after Definition~\ref{Def: right path-permutation order}.
Specifically, the equality $D(\calS')  = D(\calS)$ holds if and only if:
\begin{itemize}
\item $\calS = \ell_n$ ($n\geqslant 2$) and $\calS'$ is the following shuffle tree  monomial whose leaf labels are in increasing order according to the planar order
$$
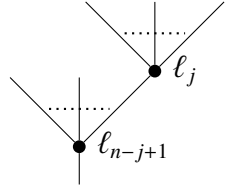
\begin{tikzpicture}
\tikzstyle{every node}=[thick,minimum size=4pt, inner sep=1pt]
\node(1-1) at (0,0)[circle,draw,fill=black,label=right: {$\ \ell_{n-j+1}$}]{};
\draw(1-1) -- (0,-0.5);
\node(2-1) at (-1,1){};
\node(2-2) at (0,1) {};
\node(2-3) at (1,1)[circle,draw,fill=black,label=right:{$\ \ell_j$}]{};
\node(3-1) at (0,2){};
\node(3-2) at (1,2){};
\node(3-3) at (2,2){};
\draw(1-1)to (2-1);
\draw(1-1)to (2-2);
\draw(1-1)to (2-3);
\draw(2-3) to (3-1);
\draw(2-3) to (3-2);
\draw(2-3) to (3-3);
\draw[dotted,thick] (-0.4,0.5) to (0.4,0.5);
\draw[dotted,thick] (0.6,1.5) to (1.4,1.5);
\end{tikzpicture}$$
for $2\leq j \leq n-1,$ or
\item $\calS = T_n$ ($n\geqslant 1$) and $\calS'$ is the following shuffle tree  monomial whose leaf labels are in increasing order according to the planar order
$$
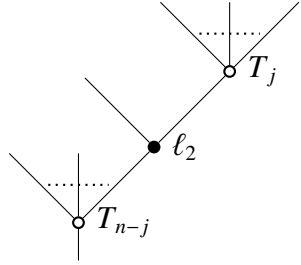
\begin{tikzpicture}
\tikzstyle{every node}=[thick,minimum size=4pt, inner sep=1pt]
\node(1-1) at (0,0)[circle,draw,label=right:{$\ T_{n-j}$}]{};
\draw(1-1) -- (0,-0.5);
\node(2-1) at (-1,1){};
\node(2-2) at (0,1) {};
\node(2-3) at (1,1)[circle,draw,fill=black,label=right:{$\ \ell_2$}]{};
\node(3-1) at (0,2){};
\node(3-3) at (2,2)[circle,draw,label=right:{$\ T_{j}$}]{};
\node(4-1) at (1,3){};
\node(4-2) at (2,3){};
\node(4-3) at (3,3){};
\draw(1-1)to (2-1);
\draw(1-1)to (2-2);
\draw(1-1)to (2-3);
\draw(2-3) to (3-1);
\draw(2-3) to (3-3);
\draw(3-3) to (4-1);
\draw(3-3) to (4-2);
\draw(3-3) to (4-3);
\draw[dotted,thick] (-0.4,0.5) to (0.4,0.5);
\draw[dotted,thick] (1.6,2.5) to (2.4,2.5);
\end{tikzpicture}$$
for $1\leq j \leq n.$
\end{itemize}
Therefore, in the quiver $Q^{\M}:$
\begin{itemize}
\item For any dashed arrow $\calt \dashrightarrow \calt'$, we have $D(\calt) = D(\calt');$
\item for any thick arrow $ \calt'' \rightarrow \calt'''$, we have $D(\calt''') \preceq D(\calt'').$
\end{itemize}
Hence, for the following zigzag path in the quiver $Q^{\M}$, we have $D(\calt'') \preceq D(\calt).$
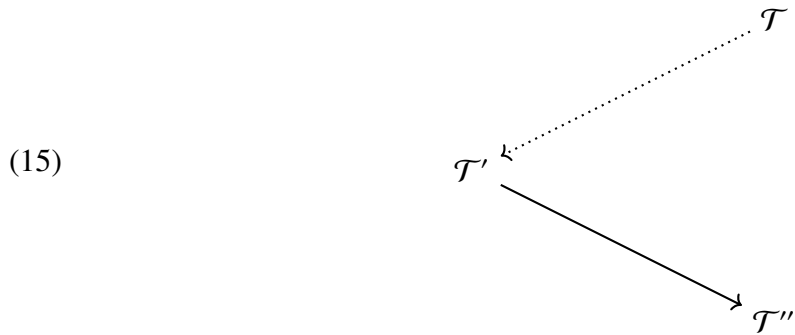
\begin{equation}\label{eq: zigzag path of l=2} \begin{tikzpicture}[baseline=(current bounding box.center)]
\node(1) at (4,2){$\mathcal{T}$};
\node(2) at (0,0){$\mathcal{T}'$};
\node(3) at (4,-2){$\mathcal{T}''$};
\draw[dotted,thick,->] (1) to  node[midway,above=0.05cm]{}(2);
\draw[thick,->] (2) to  node[midway,below=0.05cm]{}(3);
\end{tikzpicture}\end{equation}
Since the order $\prec$ is well-ordered, we only need to prove that in $Q^{\M},$ any zigzag path $p$ whose endpoints yield the same shuffle tree monomial when applying $D$  must have finite length. Then, by Lemma~\ref{Lem: criterion of morse matching}, we obtain the desired conclusion.

Consider zigzag paths of the form \eqref{eq: zigzag path of l=2} such that $D(\calt) = D(\calt') = D(\calt'').$
Since $\calt$ is the source  of the  dotted arrow, it is effective, and we denote its effective divisor by $\widehat{\calS}.$   Then $\calt' = m_{\widehat{\calS},\calS}.$
By Definition~\ref{Def: efficient tree monomials} (i), the decorated element at each internal vertex of $\calt'$ that succeeds $\calS$ in the planar order has degree zero.
Therefore, the target  $\calt''$ of the thick arrow is obtained by replacing the subtree $\mu$ corresponding to either an internal vertex preceding $\calS$ in $\calt'$ (under the planar order) or $\calS$ itself, with  some nonzero-coefficient shuffle tree monomial from $\partial(\mu).$
Based on the condition $D(\calt) = D(\calt'')$ and the fact that $\calS$ is typical, we conclude that  the first differing decorated elements $\mu$ (in $\calt$) and $\mu''$ (in $\calt''$) in the planar order satisfy $\mu'' \sqsubset \mu,$ where the order $\sqsubset$ is defined in \eqref{eq: monomial order of l and T}.
Finally, for any zigzag path in $Q^{\M}$ where all vertices yield the same shuffle tree monomial under $D$, the finiteness of internal vertices in its source shuffle tree monomial, combined with the well-ordered property of $\sqsubset,$ implies that such a path must have finite length.
\end{proof}

\begin{lem}\label{lem: morse matching of RBLA}
The set of critical vertices $\V^{\M}$ in $Q^{\M}$ forms a $\bfk$-basis of $\RBLA.$
\end{lem}

\begin{proof}
First, let $\calt$  be a positive-degree shuffle tree monomial in  $\RBLA_{\infty}$.
If $\calt$ is effective, then by the definition of $\M$, it serves as the source of the dotted arrow in $Q^{\M}.$
Therefore, without loss of generality, we may assume $\calt$ is not effective.
Since $\calt$ is of positive degree, let $\calS$ denote its last internal vertex of positive degree with respect to the planar order. Simultaneously, let $\calt'$ be the shuffle tree monomial obtained by replacing the divisor corresponding to the internal vertex  $\calS$ with $\widehat{\calS}$ in $\calt$. Therefore, we have
$$m_{\widehat{\calS},\calS}(\calt') = \calt.$$
We assert that  $\widehat{\calS}$ is the effective divisor of $\calt'.$
It follows directly that $\widehat{\calS}$ is the typical divisor of $\calt'.$
Since $\calS$ is the last internal vertex of positive degree in  $\calt$  under the planar order, the divisor $\widehat{\calS}$ of $\calt'$ satisfies condition (i) of Definition~\ref{Def: efficient tree monomials}.
Suppose that the divisor $\widehat{\calS}$ of $\calt'$ fails to satisfy condition (ii) of Definition~\ref{Def: efficient tree monomials}.
Then $\calt'$ has a typical divisor $\widehat{\calr}$ consisting of certain internal vertices that lie behind the first internal vertex of $\widehat{\calS}$.
Since $\widehat{\calS}$ satisfies Definition~\ref{Def: efficient tree monomials} (i), the typical divisor $\widehat{\calr}$ of $\calt'$ can only take one of the following two forms:
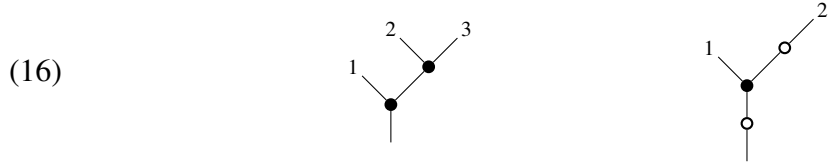
\begin{eqnarray}\label{Eq: lead monomial of RBL}
\begin{tikzpicture}[baseline=(current bounding box.center)]
\tikzstyle{every node}=[thick,minimum size=4pt, inner sep=1pt]
\node(1-1) at (0,0)[circle,draw, fill=black]{};
\draw(1-1) -- (0,-0.5);
\node(2-1) at (-0.5,0.5) {\tiny$1$};
\node(2-3) at (0.5,0.5) [circle,draw, fill=black]{};
\node(3-1) at (0,1) {\tiny$2$};
\node(3-2) at (1,1) {\tiny$3$};
\draw (1-1)--(2-1);
\draw (1-1)--(2-3);
\draw (2-3)--(3-1);
\draw (2-3)--(3-2);
\end{tikzpicture}
\hspace{3cm}
\begin{tikzpicture}[baseline=(current bounding box.center)]
\tikzstyle{every node}=[thick,minimum size=4pt, inner sep=1pt]
\node(1-1) at (1,0.5)[circle,draw]{};
\draw(1-1) -- (1,0);
\node(2-3) at (1,1) [circle,draw, fill=black]{};
\node(3-1) at (0.5,1.5) {\tiny$1$};
\node(3-3) at (1.5,1.5)[circle, draw]{};
\node(3-2) at (2,2) {\tiny$2$};
\draw (1-1)--(2-3);
\draw (2-3)--(3-1);
\draw (2-3)--(3-3);
\draw (3-2)--(3-3);
\node(s) at (2.3,0.7){$.$};
\end{tikzpicture}
\end{eqnarray}
Then the internal vertices of $\widehat{\calS}$ and $\widehat{\calr}$ in $\calt'$ must intersect. Otherwise, $\calt$ would admit a typical divisor $\widehat{\calr}$ whose internal vertices lie entirely behind $\calS$.
If we take $\widehat{\calr}$ such that no internal vertex of any typical divisor of $\calt$ lies behind the first internal vertex of $\widehat{\calr}$ under the planar order,
then $\widehat{\calr}$ would become an effective divisor of $\calt$, contradicting the  non-effectiveness of $\calt$.
Therefore, $\calt'$ contains a divisor of one of the following two types:
$$
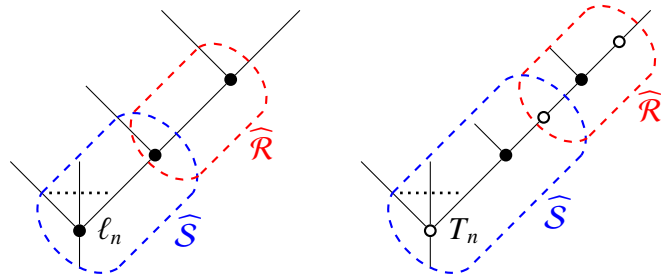
\begin{tikzpicture}
\tikzstyle{every node}=[thick,minimum size=4pt, inner sep=1pt]
\node(0-1) at (-1,-1)[circle,draw, fill=black,label=right:{$\ \ell_n$}]{};
\draw(0-1) --(-1,-1.5);
\node(0-2) at (-2,0){};
\node(0-3) at (-1,0){};
\node(1-1) at (0,0)[circle,draw, fill=black]{};
\node(2-1) at (-1,1) {};
\node(2-3) at (1,1) [circle,draw, fill=black]{};
\node(3-1) at (0,2) {};
\node(3-2) at (2,2) {};
\draw (0-1)--(0-2);
\draw (0-1)--(0-3);
\draw (0-1)--(1-1);
\draw (1-1)--(2-1);
\draw (1-1)--(2-3);
\draw (2-3)--(3-1);
\draw (2-3)--(3-2);
\draw[dotted,line width = 1pt] (-1.4,-0.5) to (-0.6,-0.5);
\draw[dashed,blue,thick] (-1.5,-0.5) to (-0.5,0.5);
\draw[dashed,blue,thick] (0.5,-0.5) to node[midway,right]{$ \ \ \widehat{\mathcal{S}} $}(-0.5,-1.5);
\draw[dashed,blue,thick] (-0.5,0.5) to[in=60,out=30] (0.5,-0.5);
\draw[dashed,blue,thick] (-0.5,-1.5) to[in=-120,out=-150]  (-1.5,-0.5);
\begin{scope}[xshift=1cm,yshift=1cm,scale=0.85]
\draw[dashed,red,thick] (-1.5,-0.5) to (-0.5,0.5);
\draw[dashed,red,thick] (0.5,-0.5) to node[midway,right]{$ \ \ \widehat{\mathcal{R}} $}(-0.5,-1.5);
\draw[dashed,red,thick] (-0.5,0.5) to[in=60,out=30] (0.5,-0.5);
\draw[dashed,red,thick] (-0.5,-1.5) to[in=-120,out=-150]  (-1.5,-0.5);
\end{scope}
\end{tikzpicture}
\hspace{5mm}
\begin{tikzpicture}
\tikzstyle{every node}=[thick,minimum size=4pt, inner sep=1pt]
\node(0-0) at (-1,-1)[circle,draw,label=right:{$\ T_n$}]{};
\draw(0-0) -- (-1,-1.5);
\node(0-1) at (-2,0){};
\node(0-2) at (-1,0){};
\node(0-3) at (0,0)[circle,draw, fill=black]{};
\node(0-4) at (-0.5,0.5){};
\node(1-1) at (0.5,0.5)[circle,draw]{};
\node(2-3) at (1,1) [circle,draw, fill=black]{};
\node(3-1) at (0.5,1.5) {};
\node(3-3) at (1.5,1.5)[circle, draw]{};
\node(3-2) at (2,2) {};
\draw(0-0)--(0-1);
\draw(0-0)--(0-2);
\draw(0-0)--(0-3);
\draw(0-3)--(0-4);
\draw(0-3)--(1-1);
\draw (1-1)--(2-3);
\draw (2-3)--(3-1);
\draw (2-3)--(3-3);
\draw (3-2)--(3-3);
\draw[dotted,,line width = 1pt] (-1.4,-0.5) to (-0.6,-0.5);
\draw[dashed,blue,thick] (-1.5,-0.5) to (0,1);
\draw[dashed,blue,thick] (1,0) to node[midway,right]{$ \ \ \widehat{\mathcal{S}} $}(-0.5,-1.5);
\draw[dashed,blue,thick] (0,1) to[in=60,out=30] (1,0);
\draw[dashed,blue,thick] (-0.5,-1.5) to[in=-120,out=-150]  (-1.5,-0.5);
\begin{scope}[xshift=1.5cm,yshift=1.5cm,scale=0.85]
\draw[dashed,red,thick] (-1.5,-0.5) to (-0.5,0.5);
\draw[dashed,red,thick] (0.5,-0.5) to node[midway,right]{$ \ \ \widehat{\mathcal{R}} $}(-0.5,-1.5);
\draw[dashed,red,thick] (-0.5,0.5) to[in=60,out=30] (0.5,-0.5);
\draw[dashed,red,thick] (-0.5,-1.5) to[in=-120,out=-150]  (-1.5,-0.5);
\end{scope}
\node(s) at (2.5,-1.3){.};
\end{tikzpicture}
$$
Observe that replacing the divisor $\widehat{\calS}$ with $\calS$ in the aforementioned shuffle tree monomial, the resulting expression becomes the typical divisor of $\calt.$
By the choice of $\widehat{\calS}$, this typical divisor satisfies the conditions of Definition~\ref{Def: efficient tree monomials}, rendering $\calt$ effective. This contradicts the non-effectiveness of $\calt.$
Consequently, the typical divisor $\widehat{\calS}$ of $\calt'$ satisfies Definition~\ref{Def: efficient tree monomials} (ii),  thereby constituting an effective divisor of $\calt'.$
Ultimately, we conclude that $\calt \to \calt'$ is an arrow in $\M,$ whence $\calt$ does not belong to $\V^{\M}.$
Therefore, there exist no shuffle tree monomials of positive degree in $\V^{\M}.$

Now, let $\calt$ be a degree-zero shuffle tree monomial in $\RBLA_{\infty}.$
If $\calt$ has a typical divisor, then among all typical divisors, the one whose first internal vertex is the rearmost (under the plane order) is an effective divisor.
Thus, $\calt$ is effective and is the target  of an arrow in $\M,$  and consequently $\calt$ does not belong to $\V^{\M}.$
Observe that in $\M$, the shuffle tree monomial at the source of any arrow is non-zero degree, while the target is always effective. Therefore, if $\calt$ admits no typical divisor, then $\calt$ cannot be effective. Moreover, by our hypothesis that $\calt$ is of degree zero, we conclude that $\calt$ must belong to $\V^{\M}.$
Thus,  $\V^{\M}$ consists precisely of those degree-zero shuffle tree monomials that admit no typical divisors.

Let $\RBLA = \calf(E) / \lan R \ran.$ A straightforward computation shows that the generators  $R$ of the operadic ideal  form  a Gr\"{o}bner basis of $ \lan R \ran$ with respect to the monomial order $\prec',$ where $\prec'$ is the right path-permutation extension induced by relation $T < \ell.$
After identifying $T_1$ with $T$ and $\ell_2$ with $\ell,$ the degree-zero typical divisors of $\calf(E)$ precisely correspond to the leading terms of $R$ with respect to $\prec'.$
Hence $\V^{\M}$ consists exactly of shuffle tree monomials admitting no divisors which are leading terms of elements in the Gr\"{o}bner basis $R$, and consequently they constitute a $\bfk$-linear basis of $\RBLA$.
\end{proof}

\begin{proof}[Proof of Theorem~\ref{Thm: Minimal model}]
By applying Theorem~\ref{Thm: main result of algebraic Morse theory} to the Morse matching $\M$ constructed in Lemma~\ref{lem: morse matching of RBLA}, we obtain that $\RBLA_{\infty}$ and $\RBLA_{\infty}^{\M} \cong \RBLA$ are homotopy equivalent.
Thus it suffices to verify that the differential vanishes on $\RBLA_{\infty}^{\M}$. This is immediate since $\RBLA_{\infty}^{\M}$ is concentrated in degree $0$ while the differential has degree $-1.$
\end{proof}

\bigskip

By the general theory of minimal models of operads,  the space spanned by the generators of the minimal model for an operad is exactly the desuspension of its Quillen homology, i.e., the homology of the bar construction of the operad. So for the operad $\RBLA$,
 denote by $\mathbf{B} (\RBLA)$ the bar construction of $\RBLA,$ and
we have the following result.
\begin{cor}
There is an isomorphism of graded collection $${\mathrm{H}}_\bullet\big(\mathbf{B}(\RBLA)\big)\cong \RBLA^{ \antish}.$$
Moreover, we have a quais-isomorphism of homotopy cooperads $$\RBLA^\antish\simeq \mathbf{B}(\RBLA).$$
\end{cor}

\bigskip

\section{The $L_\infty$-algebra and homotopy version of Rota-Baxter Lie algebras}\label{sect: Linfty RBLA}

\subsection{The $L_\infty$-algebra on the deformation complex of Rota-Baxter Lie algebras}\label{sect: Linf of RBLA}\
In this section, we will use the minimal model $\RBLA_{\infty},$ or more precisely  the Koszul dual homotopy cooperad $(\RBLA)^{\antish},$ to  determine the deformation complex as well as the $L_{\infty}$-algebra structure on it for Rota-Baxter Lie algebras of arbitrary weight.

\begin{defn}
Let $V$ be a graded space. Introduce an $L_\infty$-algebra $\frakC_{\rmRBLA}(V)$ associated to $V$ as $\frakC_{\rmRBLA}(V):=\mathbf{Hom}(\RBLA^\antish, \End_V)^{\prod}_{\s}.$
\end{defn}

Now, let's determine the $L_\infty$-algebra $\frakC_{\rmRBLA}(V)$ explicitly. The sign rules in the symmetric homotopy cooperad $(\RBLA)^{\antish}$ are complicated, so we need some transformations.
Notice that there is a natural isomorphism of symmetric  operads
$$\mathbf{Hom}(\calS, \End_{sV})\cong \End_V.$$
Explicitly, any $f\in \End_{V}(n)$  corresponds to an element $\widetilde{f} \in \mathbf{Hom}(\calS, \End_{sV})(n)$ which is defined as $\big( \widetilde{f}(\delta_n) \big)(sv_1 \ot \cdots \ot sv_{n}) = (-1)^{\sum_{k=1}^{n-1}\sum_{j=1}^{k}|v_j| } sf(v_1 \ot \cdots \ot v_n)$
for any $v_1, \cdots, v_n \in V.$

Thus we have the following isomorphisms of symmetric homotopy oeprads:
$$\begin{array} {rl}
\bfhom(\RBLA^{\antish}, \End_{V} ) \hspace{-2mm}& \cong \bfhom(\RBLA^{\antish}, \mathbf{Hom}(\calS, \End_{sV})) \\
& \cong \bfhom(\RBLA^{\antish} \ot_{\rmH} \calS, \End_{sV}) \\
& = \bfhom(\mathscr{S}( \RBLA^{\antish}), \End_{sV}).  \end{array}$$
We obtain
$$ \frakC_{\rmRBLA}(V)\cong \mathbf{Hom}(\mathscr{S}(\RBLA)^{\antish}, \End_{sV})_{\s}^{\prod}.$$

Recall that $\mathscr{S}(\RBLA)^{\antish}(n)=\bfk u_n\oplus \bfk v_n,$ where $|u_n|=0, |v_n|=1,$ and the right $\bfk \s_n$ action on both $u_n$ and $v_n$ is trivial.
By definition
$$\begin{array}{rcl} \mathbf{Hom}\big(\mathscr{S}(\RBLA^{\antish}), \End_{sV}\big)_{\s}(n)
 & = & \Hom_{\s_n}(\bfk u_n\oplus \bfk v_n, \Hom((sV)^{\ot n},sV)) \\
 & \cong& \Hom(\bfk u_n\oplus \bfk v_n, \Hom(\s^{n}(sV),sV)),\end{array}$$
 where the symmetric power $\s^{n}(sV)$  was introduced in Section~\ref{sect: notations}.
Each $f\in \Hom(\s^{n}(sV),sV)$ determines bijectively a map  $\widehat{f} \in  \Hom$ \!\!$(\bfk u_n, \Hom(\s^{n}(sV),sV))$ by imposing $\widehat{f} (u_n) = f,$
  and each  $g\in \Hom(\s^{n}(sV),V)$ determines bijectively a map $\overline{g} \in \Hom(\bfk v_n, $ \!\! $ \Hom(\s^{n}(sV),sV))$ by imposing $\overline{g}(v_n) = (-1)^{|g|} g.$
 Denote
 $$  \frakC_{\Lie}(V) =\prod\limits_{n\geqslant 1}\Hom(\s^{n}(sV),sV) \quad \text{and} \quad  \frakC_{\RBLO}(V) = \prod\limits_{n\geqslant 1}\Hom(\s^{n}(sV),V).$$
  In this way, we identify $\frakC_{\rmRBLA}(V)$ with $\frakC_{\Lie}(V) \oplus \frakC_{\RBLO}(V).$
  By the general theory recalled in Subsection~\ref{Sect: homotopy cooperad},  a direct computation gives the $L_\infty$-algebra structure on $\frakC_{\rmRBLA}(V):$

\begin{itemize}
	
	\item[(I)] For homogeneous elements $sf, sh\in \mathfrak{C}_{\Lie}(V)$, define $$l_2(sf\ot sh):= [sf, sh]_{\mathrm{NR}}\in\mathfrak{C}_{\Lie}(V),$$
	where $[ , ]_{\mathrm{NR}}$ is the Nijenhuis-Richardson bracket (see Remark~\ref{rmk: brace operation and NR bracket}).

	\item[(II)]
	\begin{itemize}	
		\item[(i)] Let $n\geqslant 1$.  For homogeneous elements $sh\in \Hom(\s^{n}(sV),sV)\subset \mathfrak{C}_{\Lie}(V)$ and $g_1,\dots, g_n\in \mathfrak{C}_{\RBLO}(V)$,	define $$l_{n+1}(sh\ot g_1\ot \cdots \ot g_n)\in \mathfrak{C}_{\RBLO}(V)$$  as :
        \begin{align*}&l_{n+1}(sh\ot g_1\ot \cdots \ot g_n)=\\
			&  \sum_{\sigma\in S_n}(-1)^{\eta}\Big(s^{-1} (sh)\{sg_{\sigma(1)},\cdots,sg_{\sigma(n)}\}-(-1)^{(|g_{\sigma(1)}|+1)(|h|+1)}s^{-1}(sg_{\sigma(1)})\big\{sh\big\{sg_{\sigma(2)},\dots,sg_{\sigma(n)}\big\}\big\}\Big),
		\end{align*}
		where $(-1)^{\eta}=\chi(\sigma; g_1,\dots,g_n)(-1)^{n(|h|+1)+\sum\limits_{k=1}^{n-1}\sum\limits_{j=1}^k|g_{\sigma(j)}|}$.

		\item[(ii)]  Let $n\geqslant 2$.  For homogeneous elements $sh\in \Hom(\s^{n}(sV),sV)\subset \mathfrak{C}_{\Lie}(V)$ and $g_1,\dots ,g_m\in \mathfrak{C}_{\RBLO}(V)$ with $1\leqslant m\leqslant n-1$, define
		$$l_{m+1}(sh\ot g_1\ot \cdots\ot g_m)\in \mathfrak{C}_{\RBLO}(V)$$ to be:
		\[l_{m+1}(sh\ot g_1\ot \cdots\ot g_m)=\sum_{\sigma\in S_m}(-1)^\xi\lambda^{n-m}  s^{-1} (sg_{\sigma(1)})\big\{sh\big\{sg_{\sigma(2)},\dots,sg_{\sigma(m)}\big\}\big\},\]
		where $(-1)^\xi=\chi(\sigma; g_1,\dots,g_m)(-1)^{1+m(|h|+1)+\sum\limits_{k=1}^{m-1}\sum\limits_{j=1}^k|g_{\sigma(j)}|+(|h|+1)(|g_{\sigma(1)}|+1)}$.
	\end{itemize}
	
	\smallskip
	
	\item[(III)]  Let $m\geqslant 1$.  For homogeneous elements $sh\in \Hom(\s^{n}(sV),sV)\subset \mathfrak{C}_{\Lie}(V), g_1,\dots,g_m\in \Hom(\s^c(sV),V)\subset \mathfrak{C}_{\RBLO}(V)$ with $1\leqslant m\leqslant n$, for $1\leqslant k\leqslant m$,  define $$l_{m+1}(g_1\ot \cdots\ot g_k\ot sh \ot g_{k+1}\ot \cdots\ot g_m)\in \mathfrak{C}_{\RBLO}(V)$$ to be
	$$l_{m+1}(g_1\ot \cdots\ot g_k\ot sh \ot g_{k+1}\ot \cdots\ot g_m)=(-1)^{(|h|+1)(\sum\limits_{j=1}^k|g_j|)+k}l_{m+1}(sh\ot g_1\ot \cdots \ot g_m),$$
	where the RHS has been introduced in (II) (i) and (ii).

	\item[(IV)] All other  components of operators $\{l_n\}_{n\geqslant 1}$ vanish.
\end{itemize}
The Nijenhuis-Richardson brackets $[-,-]_{\mathrm{NR}}$  and brace operations  $\{-;-,\cdots,-\}$  appearing in the above definition were  introduced earlier in Definition~\ref{Def: brace operation of shuffle tree} and Remark~\ref{rmk: brace operation and NR bracket}. Specifically,
for $sf\in \Hom(S^{n}(sV),sV), sg_i \in \Hom(S^{m_i}(sV), sV)$ and with $1\leq i \leq k\leq n,$
we have
$$sf\{sg_1,\cdots,sg_k\}=\sum\limits_{\sigma\in \sh(m_1,\cdots,m_k,n-k), \atop \sigma(1)<\sigma(m_1+1)<\cdots < \sigma(m_1+\cdots+m_{k-1}+1)}sf\circ(sg_1 \otimes \cdots \otimes sg_k\ot \id_{sV}^{\ot n-k})\circ r_{\sigma} .$$

\begin{thm}\label{Thm: rb-L-infty}
	Given a graded space $V$  and a scalar $\lambda\in \bfk$, the graded space ${\mathfrak{C}_{\rmRBLA}}(V)$ endowed with operations $\{l_n\}_{n\geqslant 1}$ defined above forms an $L_\infty$-algebra.
\end{thm}

Proposition~\ref{prop:Linfinity give MC} gives immediately an alternative definition of homotopy Rota-Baxter Lie algebras.
\begin{prop}\label{prop: RBLA 1-1 MC}
	A homotopy Rota-Baxter Lie algebra structure of weight $\lambda$ on a graded space is equivalent to a Maurer-Cartan element in the $L_\infty$-algebra $\frakC_{\rmRBLA}(V)$. In particular, when $V$ is concentrated in degree $0$, a Maurer-Cartan element in $\frakC_{\rmRBLA}(V)$ gives a Rota-Baxter Lie algebra structure of weight $\lambda$ on $V$.
\end{prop}

\subsection{Cohomology theory of Rota-Baxter Lie algebras}
Now we introduce the cochain complex of a Rota-Baxter Lie algebra with coefficients in a Rota-Baxter Lie representation.
We will see that this complex can be obtained by twisting the $L_{\infty}$-algebra in Section~\ref{sect: Linf of RBLA}.

Let $(\frakg, \ell = [-,-])$ be a Lie algebra and $(M, \rho)$ be a Lie representation over it. Recall that the \textbf{Chevalley-Eilenberg cochain complex of $\frakg$ with coefficients in $M$} is
$$\rmC_{\Lie}^{\bullet}(\frakg,M) : = \bigg(\bigoplus_{n=0}^{\infty} \rmC_{\Lie}^{n}(\frakg,M), \delta^{\bullet}\bigg),$$
where $\rmC_{\Lie}^{n}(\frakg,M) = \Hom(\wedge^{n}(\frakg), M)$  and the differential $\delta^{n}: \rmC_{\Lie}^{n}(\frakg,M) \to \rmC_{\Lie}^{n+1}(\frakg,M) $ is defined as:
$$ \delta^{n}(f)(x_1 \wedge \cdots \wedge x_{n+1}) =  \sum\limits_{i=1}^{n+1} (-1)^{n+i} x_i \cdot f(x_1 \wedge \cdots \wedge \widehat{x_i} \wedge \cdots \wedge x_{n+1}) \hspace{4cm}$$
$$ \hspace{2cm} +  \sum\limits_{1\leqslant i < j \leqslant n+1}(-1)^{i+j+n+1} f([x_i,x_j] \wedge x_1 \wedge \cdots \wedge \widehat{x_i} \wedge \cdots \wedge \widehat{x_j} \wedge \cdots \wedge x_{n+1}), $$
for all $f \in \rmC_{\Lie}^{n}(\frakg,M), x_1, \cdots, x_{n+1} \in \frakg,$  where  $a \cdot m : = \rho(a)(m),$ for $a\in \frakg,m\in M.$
In the above definition, $\wedge^{n}(\frakg) $ is the $n$-th exterior power of $\frakg$ (as defined in Section~\ref{sect: notations}). The notation  $x_1 \wedge \cdots \wedge \widehat{x_i} \wedge \cdots \wedge x_{n+1}$ denotes the wedge product of the sequence with $x_i$ omitted. The cohomology of the Chevalley-Eilenberg cochain complex $\rmC_{\Lie}^{\bullet}(\frakg,M)$ is called the \textbf{Chevalley-Eilenberg cohomology of $\frakg$ with coefficient in $M$}, denoted by $\rmH_{\Lie}^{\bullet}(\frakg,M).$
When the Lie representation $M$ the regular Lie presentation  $\frakg$ itself, we just denote $\rmC_{\Lie}^{\bullet}(\frakg,\frakg)$ by $\rmC_{\Lie}^{\bullet}(\frakg)$ and call it the  Chevalley-Eilenberg cochain complex of Lie algebra $\frakg$.
Denote $\rmH_{\Lie}^{\bullet}(\frakg,\frakg)$ by $\rmH_{\Lie}^{\bullet}(\frakg),$ called the Chevalley-Eilenberg cohomology of Lie algebra $\frakg$.
	
We now introduce the concept of Rota-Baxter Lie representation.
\begin{defn}
Let $(\frakg, \ell = [-,-], T)$ be a Rota-Baxter Lie algebra of weight $\lambda$, $(M, \rho)$ be a Lie representation over Lie algebra $(\frakg, \ell)$. We say that $M$ is a \textbf{Rota-Baxter Lie representation} if $M$ is endowed with a linear operation $T_{M} : M \to M$ such that
$$ T(a)\cdot T_{M}(m) = T_{M}\bigg(a \cdot T_{M}(m) +  T(a) \cdot m  + \lambda a\cdot m\bigg) $$
holds for any $a\in \frakg$ and $m \in M.$
\end{defn}

Of course, $(\frakg,\ell,T)$ itself is a Rota-Baxter Lie representation  over the Rota-Baxter Lie algebra $(\frakg,\ell,T),$ called the regular Rota-Baxter Lie representation.

Parallel to the case of Rota-Baxter associative algebras (see \cite{Guo12,WZ24,WZ24}), we obtain corresponding descendent properties for Rota-Baxter Lie algebras, whose proofs follow by obvious modifications and are left as exercises.
\begin{prop}
Let $(\frakg, \ell = [-,-], T)$ be a Rota-Baxter Lie algebra of weight $\lambda$. Define a new binary operation as:
$$[a, b]_{\star} : = [a, T(b)] + [T(a),b] + \lambda [a,b]$$
for any $a,b \in \frakg.$ Then
\begin{itemize}
\item[(i)] $(\frakg, [-,-]_{\star})$ is a Lie algebra;
\item[(ii)] the triple $(\frakg, [-,-]_{\star}, T)$ also forms a Rota-Baxter Lie algebra of weight $\lambda$ and denote it by $\frakg_{\star};$
\item[(iii)] the map $T: (\frakg, [-,-]_{\star}, T)\to(\frakg, \ell = [-,-], T)$ is a morphism of Rota-Baxter Lie algebras.
\end{itemize}
\end{prop}
\begin{prop}
Let $(\frakg, \ell = [-,-], T)$ be a Rota-Baxter Lie algebra of weight $\lambda$ and $(M,\rho, T_M)$ be a Rota-Baxter Lie representation over it. We define a new Lie representation $\rhd$ of $\frakg$ on $M$ as follows: for any $a\in \frakg, m \in M,$
$$a \rhd m : = T(a) \cdot m - T_{M}(a \cdot m).$$
Then this operation makes $M$ into a Rota-Baxter Lie representation over $\frakg_{\star}$ and denote this new Lie representation by $_{\rhd}M.$
\end{prop}

Let $(\frakg, \ell = [-,-], T)$ be a Rota-Baxter Lie algebra of weight $\lambda$ and $(M,\rho, T_M)$ be a Rota-Baxter Lie representation over it. Consider the Chevalley-Eilenberg cochain complex of $\frakg_{\star}$ with coefficients in $_{\rhd}M:$
$$\rmC_{\Lie}^{\bullet}(\frakg_{\star}, _{\rhd}M) : = \bigg(\bigoplus_{n=0}^{\infty} \rmC_{\Lie}^{n}(\frakg_{\star}, _{\rhd}M), \partial^{\bullet}\bigg).$$
More precisely, for $n\geqslant 0,$ $ \rmC_{\Lie}^{n}(\frakg_{\star}, _{\rhd}M) = \Hom(\wedge^{n}(\frakg), M)$ and its differential $\partial^{n} : \rmC_{\Lie}^{n}(\frakg_{\star}, _{\rhd}M) \to \rmC_{\Lie}^{n+1}(\frakg_{\star}, _{\rhd}M)$ is defined as
$$ \partial^{n}(f)(x_1 \wedge \cdots \wedge x_{n+1}) \hspace{14cm}$$
$$ \hspace{0cm}\begin{array} {rl}
= &  \sum\limits_{i=1}^{n+1} (-1)^{n+i} x_i \rhd f(x_1 \wedge \cdots \wedge \widehat{x_i} \wedge \cdots \wedge x_{n+1})\\
  &+  \sum\limits_{1\leqslant i < j \leqslant n+1}(-1)^{i+j+n+1} f([x_i,x_j]_{\star} \wedge x_1 \wedge \cdots \wedge \widehat{x_i} \wedge \cdots \wedge \widehat{x_j} \wedge \cdots \wedge x_{n+1})\\
= &  \sum\limits_{i=1}^{n+1} (-1)^{n+i} T(x_i)\cdot f(x_1 \wedge \cdots \wedge \widehat{x_i} \wedge \cdots \wedge x_{n+1}) -   \sum\limits_{i=1}^{n+1} (-1)^{n+i}  T_{M}(x_i \cdot f(x_1 \wedge \cdots \wedge \widehat{x_i} \wedge \cdots \wedge x_{n+1}))\\
& +  \sum\limits_{1\leqslant i < j \leqslant n+1} (-1)^{i+j+n+1} f([T(x_i),x_j] \wedge x_1 \wedge \cdots \wedge \widehat{x_i} \wedge \cdots \wedge \widehat{x_j} \wedge \cdots \wedge x_{n+1})\\
& +  \sum\limits_{1\leqslant i < j \leqslant n+1}(-1)^{i+j+n+1} f([x_i,T(x_j)] \wedge x_1 \wedge \cdots \wedge \widehat{x_i} \wedge \cdots \wedge \widehat{x_j} \wedge \cdots \wedge x_{n+1})\\
& +  \sum\limits_{1\leqslant i < j \leqslant n+1}(-1)^{i+j+n+1} \lambda f([x_i,x_j] \wedge x_1 \wedge \cdots \wedge \widehat{x_i} \wedge \cdots \wedge \widehat{x_j} \wedge \cdots \wedge x_{n+1}),
\end{array}$$
for any $f\in \rmC_{\Lie}^{n}(\frakg_{\star}, _{\rhd}M)$ and $x_1, \cdots, x_{n+1} \in \frakg.$
\begin{defn}
Let $(\frakg, \ell, T)$ be a Rota-Baxter Lie algebra of weight $\lambda$ and $(M,\rho, T_M)$ be a Rota-Baxter Lie representation over it. Then the cochain complex  $\rmC_{\Lie}^{\bullet}(\frakg_{\star}, _{\rhd}M)$ is called the \textbf{cochain complex of Rota-Baxter operator $T$ with coefficients in $(M, \rho, T_{M})$}, denoted by $\rmC_{\RBLO}^{\bullet}(\frakg,M)$. The cohomology of $\rmC_{\RBLO}^{\bullet}(\frakg,M)$, denoted by $\rmH_{\RBLO}^{\bullet}(\frakg, M)$, is called the \textbf{cohomology of Rota-Baxter operator $T$ with coefficients in $(M, \rho, T_{M})$}.
\end{defn}
When $(M, \rho, T_{M})$ is the regular Rota-Baxter Lie representation $(\frakg, \ell, T),$ we denote $\rmC_{\RBLO}^{\bullet}(\frakg,\frakg)$ by $\rmC_{\RBLO}^{\bullet}(\frakg)$ and call it the cochain complex of Rota-Baxter operator $T$, and denote $\rmH_{\RBLO}^{\bullet}(\frakg, \frakg)$ by $\rmH_{\RBLO}^{\bullet}(\frakg)$ and call it the cohomology of Rota-Baxter operator $T$.

Let $(\frakg, \ell = [-,-], T)$ be a Rota-Baxter Lie algebra of weight $\lambda$ and $(M,\rho, T_M)$ be a Rota-Baxter Lie representation over it. Now, let's construct a chain map
$$\Phi^{\bullet}:\rmC_{\Lie}^{\bullet}(\frakg,M) \to \rmC_{\RBLO}^{\bullet}(\frakg,M), $$
namely, the following commutative diagram:
$$\small {\xymatrix@C=5mm@W=3mm{
 0  \ar@{->}[r]&  \rmC_{\Lie}^{0}(\frakg,M)\ar@{->}[r]^{\delta^0} \ar@{->}[d]^{\Phi^0} & \rmC_{\Lie}^{1}(\frakg,M)\ar@{.}[r]\ar@{->}[d]^{\Phi^1} & \rmC_{\Lie}^{n}(\frakg,M)\ar@{->}[r]^{\delta^n}\ar@{->}[d]^{\Phi^{n}} & \rmC_{\RBLO}^{n+1}(\frakg,M)\ar@{.}[r]\ar@{->}[d]^{\Phi^{n+1}} & \\
 0 \ar@{->}[r]&  \rmC_{\RBLO}^{0}(\frakg,M)\ar@{->}[r]^{\partial^0} &\rmC_{\RBLO}^{1}(\frakg,M)\ar@{.}[r] & \rmC_{\RBLO}^{n}(\frakg,M)\ar@{->}[r]^{\partial^n}& \rmC_{\RBLO}^{n+1}(\frakg,M)\ar@{.}[r] & .\\
}}$$

Define $\Phi^{0} = \id_{\Hom(\bfk,M)} = \id_{M},$ and for $n\geqslant 1$ and $f \in \rmC_{\Lie}^{n}(\frakg,M),$ define $\Phi^{n}(f) \in \rmC_{\RBLO}^{n}(\frakg,M)$ as:
$$\Phi^{n}(f)(x_1 \wedge \cdots \wedge x_{n}) = f(T(x_1)\wedge \cdots \wedge T(x_n)) \hspace{5cm}$$
$$ - \sum\limits_{k=0}^{n-1} \sum_{1\leqslant i_1 < \cdots < i_k \leqslant n} \lambda^{n-k-1} T_{M} \circ f(x_{1\wedge (i_1-1)}  \wedge T(x_{i_1}) \wedge x_{(i_1+1) \wedge (i_2-1)}  \wedge T(x_{i_2}) \wedge \cdots \wedge T(x_{i_k}) \wedge x_{(i_k+1) \wedge n}),$$
where the notation $x_{n\wedge m}$ denotes the exterior product $x_n \wedge x_{n+1} \wedge \cdots \wedge x_{m}.$
\begin{prop}\label{prop: Phi is chain map}
The map $\Phi^{\bullet}:\rmC_{\Lie}^{\bullet}(\frakg,M) \to \rmC_{\RBLO}^{\bullet}(\frakg,M)$ is a chain map.
\end{prop}

This result is equivalent to the fact that the cochain complex $\rmC_{\rmRBLA}^{\bullet}(\frakg,M)$ of Rota-Baxter Lie algebra $(\frakg, \ell = [-,-], T)$ with coefficients in $(M, \rho, T_{M})$ in the following definition is a cochain complex, so it follows from Proposition~\ref{prop: cochain complex of RBLA}.
\begin{defn}
Let $(\frakg, \ell, T)$ be a Rota-Baxter Lie algebra of weight $\lambda$ and $(M,\rho, T_M)$ be a Rota-Baxter Lie representation over it. We define the \textbf{cochain complex} $$\rmC_{\rmRBLA}^{\bullet}(\frakg,M) : =\bigg( \bigoplus_{n=0}^{\infty} \rmC_{\rmRBLA}^{n}(\frakg,M), d^{\bullet} \bigg) $$
\textbf{of Rota-Baxter Lie algebra  $(\frakg, \ell, T)$ with coefficients in $(M,\rho, T_M)$} to be the negative shift of the mapping cone of $\Phi^{\bullet},$ that is, let
$$\rmC_{\rmRBLA}^{0}(\frakg,M)= \rmC_{\Lie}^{0}(\frakg,M) \quad \text{and} \quad \rmC_{\rmRBLA}^{n}(\frakg,M)= \rmC_{\Lie}^{n}(\frakg,M)\oplus \rmC_{\RBLO}^{n-1}(\frakg,M), \forall n \geqslant 1, $$
and the differential $d^{n}:\rmC_{\rmRBLA}^{n}(\frakg,M) \to \rmC_{\rmRBLA}^{n+1}(\frakg,M)$ is given by
$$d^{n}(f,g) = (\delta^{n}(f), -\partial^{n-1}(g) - \Phi^{n}(f))$$
for any $f \in \rmC_{\Lie}^{n}(\frakg,M)$ and $g \in \rmC_{\RBLO}^{n-1}(\frakg,M).$ The cohomology of  $\rmC_{\rmRBLA}^{\bullet}(\frakg,M)$, denoted by $\rmH_{\rmRBLA}^{\bullet}(\frakg,M)$, is called the cohomology of the Rota-Baxter Lie algebra $(\frakg, \ell, T)$ with coefficients in $(M,\rho, T_M)$. When $(M,\rho, T_M)= (\frakg, \ell, T),$ we just denote $ \rmC_{\rmRBLA}^{\bullet}(\frakg,\frakg), \rmH_{\rmRBLA}^{\bullet}(\frakg,\frakg)$ by $\rmC_{\rmRBLA}^{\bullet}(\frakg), \rmH_{\rmRBLA}^{\bullet}(\frakg)$ respectively, and call them the \textbf{cochain complex}, \textbf{the cohomology of Rota-Baxter Lie algebra $(\frakg, \ell, T)$} respectively.
\end{defn}

\begin{prop}\label{prop: cochain complex of RBLA}
Let $(\frakg, \ell, T)$ be a Rota-Baxter Lie algebra of weight $\lambda.$
 Twisting the $L_\infty$-algebra $\frakC_{\rmRBLA}(\frakg)$ by the Maurer-Cartan element  corresponding to the Rota-Baxter Lie algebra structure $(\ell, T),$
 then its underlying complex is precisely $s\rmC_{\rmRBLA}^{\bullet}(\frakg)$,  the shift of the cochain complex of the Rota-Baxter Lie algebra $(\frakg, \ell, T)$.
	\end{prop}
\begin{proof}
We first make explicit the bijection in Proposition~\ref{prop: RBLA 1-1 MC}. Regard  $\frakg$ as a  graded vector space concentrated  in degree $0$.
Set $\widetilde{\ell} = - \ell (s^{-1} \ot s^{-1} ): s\frakg \ot s\frakg \to \frakg $ and $\widetilde{T} = T (s^{-1}) : s\frakg \to \frakg.$
Then $\alpha = (s\widetilde{\ell}, \widetilde{T})$ corresponds to a Maurer-Cartan element in  $\frakC_{\rmRBLA}(\frakg)$.
Conversely, given a Maurer-Cartan element $\alpha = (s\widetilde{\ell}, \widetilde{T})$ in  $\frakC_{\rmRBLA}(\frakg)$, define $\ell = \widetilde{\ell} (s\ot s)$ and $T=\widetilde{T}s.$
Then $(\frakg,\ell,T)$ become a Rota-Baxter Lie algebra.

We now compute the twisted differential  $\ell_1^{\alpha} : \frakC_{\rmRBLA}(\frakg) \to \frakC_{\rmRBLA}(\frakg)$ for $\alpha = (s\widetilde{\ell}, \widetilde{T}).$
Given $\beta=sf \in \Hom((s\frakg)^{\ot n}, s\frakg) \subset\frakC_{\Lie}(\frakg),$  the twisted differential acts as
$$\ell_1^{\alpha}(\beta) =  -\ell_2(s\widetilde{\ell}, sf) + \sum_{i=1}^{n} \frac{1}{i!} (-1)^{\frac{i(i+1)}{2}} \ell_{i+1}(\widetilde{T}^{\ot i} \ot sf),$$
according to  Proposition~\ref{Prop: deformed-L-infty} and the construction of $\{\ell_n\}_{n\geqslant 1}$ from Theorem~\ref{Thm: rb-L-infty}.
Furthermore, we compute that
$$ \ell_2(s \widetilde{\ell} \ot sf) = \sum_{\sigma \in \sh(n,1)} s\widetilde{\ell} \circ (sf \ot \id)\circ r_{\sigma} - (-1)^{n-1} \sum_{\sigma \in \sh(2,n-2)} sf\circ (s\widetilde{\ell} \ot \id^{\ot n-1}) \circ r_{\sigma} \in \frakC_{\Lie}(\frakg),   $$
$$ \ell_{i+1}(\widetilde{T}^{\ot i} \ot sf) =  i! \sum_{\sigma \in \sh(i-1,n-i+1)} (-1)^{1+\frac{i(i+1)}{2}} \lambda^{n-i} \widetilde{T} \circ (sf \circ (\widetilde{T}^{\ot i-1} \ot \id^{\ot n-i+1})) \circ r_{\sigma} \in \frakC_{\RBLO}(\frakg),  $$
for $1\leqslant i \leqslant n-1, $ and
$$ \ell_{n+1}(\widetilde{T}^{\ot n} \ot sf)  =  n! \sum_{\sigma \in \sh(n-1,1)} (-1)^{1+\frac{n(n+1)}{2}} \widetilde{T} \circ (sf \circ (\widetilde{T}^{\ot n-1} \ot \id)) \circ r_{\sigma} +n! (-1)^{\frac{n(n+1)}{2}}f\circ(s\widetilde{T}^{\ot n}) \in \frakC_{\RBLO}(\frakg).  $$
Given $\beta=g \in \Hom((s\frakg)^{\ot n}, \frakg) \subset\frakC_{\RBLO}(\frakg),$  the twisted differential acts as
$$ \ell_1^{\alpha}(\beta) = -\ell_2( s\widetilde{\ell} \ot g)  - \ell_3 (s\widetilde{\ell} \ot \widetilde{T} \ot g) $$
according to  Proposition~\ref{Prop: deformed-L-infty} and the construction of $\{\ell_n\}_{n\geqslant 1}$ from Theorem~\ref{Thm: rb-L-infty}.
Furthermore, we compute that
$$ \ell_2( s\widetilde{\ell} \ot g) = \sum_{\sigma \in \sh(2,n-2)}(-1)^{n-1} \lambda g\circ(s\widetilde{\ell} \ot \id^{\ot n-2})\circ r_{\sigma}  $$
and
$$\ell_3 (s\widetilde{\ell} \ot \widetilde{T} \ot g) = - \sum_{\sigma \in \sh(1,n)}  \widetilde{\ell}\circ (s\widetilde{T} \ot sg)\circ r_{\sigma}  + \sum_{\sigma \in \sh(n,1)} \widetilde{T}\circ (s\widetilde{\ell} \circ (sg \ot \id)) \circ r_{\sigma}$$
$$ + \sum_{\sigma \in \sh(1,1,n-2)} (-1)^{n-1} g \circ (s\widetilde{\ell} \circ(s\widetilde{T} \ot \id) \ot \id^{n-2})\circ r_{\sigma}.$$
After identifying the spaces $\Hom((s\frakg)^{\ot n}, s\frakg)\subset\frakC_{\Lie}(\frakg)$ with $\rmC_{\Lie}^{n}(\frakg)) $ and $\Hom((s\frakg)^{\ot n}, \frakg)\subset\frakC_{\RBLO}(\frakg)$ with $\rmC_{\RBLO}^{n}(\frakg) $, we obtain
$(\frakC_{\rmRBLA}(\frakg), \ell_1^{\alpha}) \cong s\rmC_{\rmRBLA}^{\bullet}(\frakg). $
\end{proof}

\begin{prop}\label{prop: dgLa RBLO and twisting}
Let $(\frakg, \ell)$ be a Lie algebra.
\begin{itemize}
\item[(i)]  Then the graded space $ \frakC_{\RBLO}(\frakg)$   can be endowed with a dg Lie algebra structure,
and the set of its  Maurer-Cartan elements is in bijection with the set of Rota-Baxter operators  of weight $\lambda$ on $(\frakg, \ell)$;
\item[(ii)] Given a Rota-Baxter operator $T$ on the Lie algebra  $(\frakg, \ell)$, the underlying complex of the twisted dg Lie algebra $\frakC_{\RBLO}(\frakg)$
by the corresponding Maurre-Cartan    element is exactly the cochain complex of Rota-Baxter operator   $\rmC_{\RBLO}^{\bullet}(\frakg)$.
\end{itemize}
\end{prop}
\begin{proof}
Regard  $\frakg$ as a  graded vector space concentrated  in degree $0$. According to the correspondence given in the proof of Proposition~\ref{prop: cochain complex of RBLA}, the Lie bracket of
$\frakg$ corresponds to the MC element $\alpha = (s\widetilde{\ell}, 0)$  of $ \frakC_{\rmRBLA}(\frakg)$.
By the construction of the $L_{\infty}$-algebra structure $\{\ell_n\}_{n\geqslant 1}$ on $ \frakC_{\rmRBLA}(\frakg)$, the graded subspace $ \frakC_{\RBLO}(\frakg) $ is closed under the action of $\{\ell_n^{\alpha} \}_{n\geqslant 1}$.

We now compute the twisted $L_{\infty}$-algebra structure $\{\ell_n^{\alpha} \}_{n\geqslant 1}$ on $ \frakC_{\RBLO}(\frakg)$ for $\alpha = (s\widetilde{\ell}, 0).$
\begin{itemize}
\item Given $\beta = g \in \Hom((s\frakg)^{\ot n}, \frakg) \subset \frakC_{\RBLO}(\frakg),$ define
$$ \ell_1^{\alpha}(\beta) = -\ell_2(s\widetilde{\ell} \ot g). $$
\item Given $\beta_i = g_i \in \Hom((s\frakg)^{\ot n_i}, \frakg) \subset \frakC_{\RBLO}(\frakg)$   $(i=1,2),$ define
$$\ell_2^{\alpha}(\beta_1 \ot \beta_2) = \ell_3(s\widetilde{\ell} \ot g_1 \ot g_2).   $$
\item All other components of operators  $\{\ell_n^{\alpha} \}_{n\geqslant 1}$ on $ \frakC_{\RBLO}(\frakg)$   vanish.
\end{itemize}
Therefore $(\frakC_{\RBLO}(\frakg), \ell_1^{\alpha}, \ell_2^{\alpha}) $ forms a dg Lie algebra.

Furthermore, we compute that
$$ \ell_2(s\widetilde{\ell} \ot g) =  \sum_{\sigma \in \sh(2,n-1)}(-1)^{n+1} \lambda g \circ (s\widetilde{\ell} \ot \id^{n-1}) \circ r_{\sigma} $$
and
$$ \ell_3(s\widetilde{\ell} \ot g_1 \ot g_2) = \sum_{\sigma \in \sh(n_1,n_2)}  (-1)^{n_1} \widetilde{\ell} \circ (sg_1 \ot sg_2) \circ  r_{\sigma}
+ \sum_{\sigma \in \sh(n_2,1,n_1-1)} g_1 \circ (s\widetilde{\ell} \circ (sg_2 \ot \id) \ot \id^{n_1-1})\circ r_{\sigma} $$
$$+ \sum_{\sigma \in \sh(n_1, 1,n_2-1)} (-1)^{n_1n_2 + 1} g_2 \circ (s\widetilde{\ell}\circ (sg_1 \ot \id) \ot \id^{n_2-1})\circ r_{\sigma}. $$

Since $\frakg$ is concentrated in degree $0$, $\tau \in \frakC_{\RBLO}(\frakg)_{-1}$ lies in $\Hom((s\frakg), \frakg)$. Then $\tau$ satisfies the Maurer-Cartan equation:
$$ \ell_1^{\alpha}(\tau) - \frac{1}{2} \ell_2^{\alpha}(\tau \ot \tau) = 0,  $$
if and only if
$$  -\lambda \tau \circ s\widetilde{\ell} - \big(-\widetilde{\ell} \circ (s\tau \ot s \tau)  + \tau \circ s\widetilde{\ell} \circ (s\tau \ot \id) + \tau \circ s\widetilde{\ell} \circ ( \id \ot s\tau )\big)=0 $$
Let $T = \tau \circ s : \frakg \to \frakg.$  Then the above equation shows that $T$ is a Rota-Baxter operator of weight $\lambda$ on the Lie algebra $(\frakg, \ell)$.

Now let $T$ be a Rota-Baxter operator on the Lie algebra $(\frakg, \ell)$.
According to part (i) of this proposition, $\tau = T \circ s^{-1}$ is a Maurer-Cartan element in the twisted dg Lie algebra $ (\frakC_{\RBLO}(\frakg), \ell_1^{\alpha}, \ell_2^{\alpha}).$
Thus by Proposition~\ref{Prop: deformed-L-infty}, for any $f \in \Hom((s\frakg)^{\ot n}, \frakg) \subset \frakC_{\RBLO}(\frakg),$  the  twisted differential $(\ell_1^{\alpha})^{\tau}$ is given as
$$\hspace{-10cm}(\ell_1^{\alpha})^{\tau} (f) = \ell_1^{\alpha}(f) - \ell_2^{\alpha}(\tau \ot f)  $$
$$ \hspace{-1cm}= \sum_{\sigma \in \sh(2,n-1)}(-1)^{n} \lambda f \circ (s\widetilde{\ell} \ot \id^{n-1}) \circ r_{\sigma} +  \sum_{\sigma \in \sh(1,n)}   \widetilde{\ell} \circ (s\tau \ot sf) \circ  r_{\sigma}
 $$
$$\hspace{2.5cm}- \sum_{\sigma \in \sh(n,1)} \tau \circ s\widetilde{\ell} \circ (sf \ot \id)\circ r_{\sigma} + \sum_{\sigma \in \sh(1, 1,n-1)} (-1)^{n} f \circ (s\widetilde{\ell}\circ (s\tau \ot \id) \ot \id^{n-1})\circ r_{\sigma}). $$
After identifying the spaces  $\Hom((s\frakg)^{\ot n}, \frakg)\subset\frakC_{\RBLO}(\frakg)$ with $\rmC_{\RBLO}^{n}(\frakg) $, we obtain
$$(\frakC_{\RBLO}(\frakg), (\ell_1^{\alpha})^{\tau}) \cong \rmC_{\RBLO}^{\bullet}(\frakg). $$
\end{proof}

\subsection{Homotopy Rota-Baxter Lie algebras}\
Since we have found the  operad of ``homotopy Rota-Baxter Lie algebras of weight $\lambda$", we could now give the definition of homotopy Rota-Baxter Lie algebras.
		
\begin{defn}
	Let $(V,d_V)$ be a complex. A homotopy Rota-Baxter Lie algebra of weight $\lambda$ on $V$ is defined to be a morphism of symmetric dg operads from $\RBLA_\infty$ to $\End_V$.
\end{defn}

Let $(V,d_V)$ be a homotopy Rota-Baxter Lie algebra. Still denote by $\ell_n:V^{\otimes n} \to V, n \geq 2$ (resp. $T_n: V^{\otimes n} \to V, n\geq 1$) the image of $\ell_n\in \RBLA_{\infty}$ (resp. $T_n \in \RBLA_{\infty}$). We also rewrite $\ell_1=d_V.$
The $\mathbb{S}$-module structure of $\RBLA_{\infty}$ gives the anti-symmetry of $\ell_n$ and $T_n$
\begin{eqnarray}
\ell_n(v_1\otimes \dots \otimes v_n)= \chi(\sigma ; v_1,\dots ,v_n) \ell_n(v_{\sigma(1)}\otimes \dots \otimes v_{\sigma(n)}) \\
 T_n(v_1\otimes \dots \otimes v_n)= \chi(\sigma ; v_1,\dots ,v_n)T_n(v_{\sigma(1)}\otimes \dots \otimes v_{\sigma(n)})
\end{eqnarray}
for all $\sigma \in \s_n,v_1,\dots, v_n \in V,n\geq 1.$
 Equations~(\ref{Eq: partial l_n}) and (\ref{Eq: partial T_n}) give
\begin{eqnarray}\label{Eq: homotopy ell_n}
\sum\limits_{i=1}^n\sum\limits_{\sigma\in \Sh(i,n-i)}\sgn(\sigma) (-1)^{i(n-i)}\ell_{n-i+1}(\ell_i\otimes \id^{\otimes n-i}) r_{\sigma}=0
\end{eqnarray}	
and
\begin{eqnarray}\label{Eq: homotopy T_n}
&& \sum\limits_{k=1}^{n} \sum\limits_{\substack{r_1+ \dots + r_k=n, \\ r_1,\dots, r_k \geq 1}}\sum\limits_{\substack{\sigma\in \sh(r_1,\dots, r_k), \\ \sigma(1)< \sigma(r_1+1)<\dots < \sigma(r_1+\dots +r_{k-1}+1) }} (-1)^{\delta+1}\sgn(\sigma)\ell_{k} (T_{r_1}\otimes \dots \otimes T_{r_k}) r_{\sigma} \\ \notag
=&&  \sum\limits_{1\leq q\leq p} \sum\limits_{\substack{r_1 + \dots + r_q + p - q =n, \\ r_1, \dots, r_q \geq 1 }} \sum\limits_{\substack{\sigma\in \sh(r_2, \dots, r_q, p-q+1, r_1-1), \\ \sigma(1)< \sigma(r_2+1)< \dots < \sigma(r_2+\dots +r_{q-1}+1)}}\lambda^{p-q}(-1)^{\eta}\sgn(\sigma) \\ \notag
 &&  T_{r_1}(\ell_p(T_{r_2}\otimes \dots \otimes T_{r_q} \otimes \id^{\otimes p-q+1}) \otimes \id^{r_1-1}) r_{\sigma},
\end{eqnarray}
where \begin{eqnarray*}
 	\delta&=&1+\frac{k(k-1)}{2}+\sum_{j=1}^k(k-j)r_j=1+\sum_{j=1}^k(k-j)(r_j-1),\\
 \eta &=& \big(p+\sum\limits_{j=2}^q(r_j-1)\big)\big(r_1-1\big)+\sum\limits_{j=2}^q(r_j-1)(p-j+1).	\end{eqnarray*}

We obtain thus an equivalent definition of homotopy Rota-Baxter Lie algebras.

\begin{defn}
Let $V$ be a graded space. A homotopy Rota-Baxter Lie algebra of weight $\lambda$ structure on $V$ consists of two family of anti-symmetric graded maps $\ell_n:V^{\otimes n } \to V, n\geq 1$ and $T_n:V^{\otimes n } \to V,n \geq 1,$ with $|\ell_n|=n-2$ and $|T_n|=n-1,$  subject to Equations~(\ref{Eq: homotopy ell_n}) and (\ref{Eq: homotopy T_n}).
\end{defn}
Equation~(\ref{Eq: homotopy ell_n}) is exactly the generalised Jacobi identity defining $L_{\infty}-$algebras. In particular, the operator $\ell_1$ is a differential on $V$ and the operator $\ell_2$ induces a Lie algebra structure on the homology $\h_{\bullet}(V,\ell_1).$
\begin{exam}
 Expanding Equation~(\ref{Eq: homotopy T_n}) for small $n$'s gives the following:
 \begin{itemize}
 \item[(i)] when $n=1,$ we have $$ \ell_1 T_1=T_1 \ell_1,$$
 which implies that $T_1:(V, \ell_1) \to (V, \ell_1)$ is a chain map;
 \item[(ii)] when $n=2,$ we have
 \begin{eqnarray*} && \ell_2(T_1\otimes T_1)-T_1 \ell_2 (T_1\otimes \id)- T_1 \ell_2 (\id \otimes T_1) - \lambda T_1\ell_2 \\
 = && \ell_2(T_1\otimes T_1)-T_1 \ell_2 (T_1\otimes \id)+ T_1 \ell_2 (T_1\otimes \id)r_{(12)} - \lambda T_1\ell_2 \\
 = && -\ell_1 T_2 -T_2 (\ell_1 \otimes \id) + T_2 (\ell_1 \otimes \id)r_{(12)}\\
 =&&-(\ell_1 T_2 +T_2 (\ell_1 \otimes \id) + T_2 ( \id \otimes \ell_1)) =-\partial(T_2),
 \end{eqnarray*}
 which shows that $T_1$ is a Rota-Baxter Lie operator of weight $\lambda$ with respect to $\ell_2,$ but only up to homotopy given by the operator $T_2.$
  \end{itemize}
\end{exam}
Observe that for a homotopy Rota-Baxter Lie algebra $(V, \{\ell_n\}_{n\geq 1}, \{T_n\}_{n\geq 1}),$ its homology $\h_{\bullet}(V,\ell_1)$ endowed with the operators induced by $\ell_2$ and $T_1$ is an usual Rota-Baxter Lie algebra.

\subsection{Comparison with the deformation theory of relative Rota-Baxter Lie algebras}\label{subsect: comparison with LST}
In this subsection, we compare our $L_\infty$-algebras, cohomology theories and homotopy Rota-Baxter Lie algebras with the corresponding results of Tang, Bai, Guo and Sheng \cite{TBGS19} and of Lazarev, Sheng and Tang \cite{LST, LST2} on relative Rota-Baxter Lie algebras, and show that their results, when the representation is specialized to the adjoint one, coincide with ours at weight zero.

Recall from \cite{TBGS19, LST} that a \textbf{relative Rota-Baxter operator} (of weight zero), also called an $\mathcal{O}$-operator, on a Lie algebra $(\frakg,[-,-])$ with respect to a representation $(V,\rho)$ is a linear map $T: V\to \frakg$ satisfying
$$[T(u),T(v)]=T\big(\rho(T(u))(v)-\rho(T(v))(u)\big), \quad \forall\, u,v\in V.$$
A relative Rota-Baxter Lie algebra consists of a Lie algebra, a representation and a relative Rota-Baxter operator. When $(V,\rho)$ is the adjoint representation $(\frakg,\mathrm{ad})$, a relative Rota-Baxter operator is nothing but a Rota-Baxter operator of weight zero on $\frakg$. We now compare the two theories item by item under this specialization.

\textbf{(i) The controlling $L_\infty$-algebras.} In \cite{LST}, applying Voronov's higher derived brackets \cite{Vor, Vor2} to a $V$-data built out of the Nijenhuis-Richardson bracket on $C^\bullet(\frakg\oplus V, \frakg\oplus V)$, the authors constructed an $L_\infty$-algebra whose Maurer-Cartan elements are precisely relative Rota-Baxter Lie algebra structures; the part governing the operator lives on $\bigoplus_{n\geqslant 1}\Hom(\wedge^n V, \frakg)$, which, under the standard suspension identification, is the same as $\prod_{n\geqslant 1}\Hom(s^n(sV),s\frakg)$. Specializing $(V,\rho)=(\frakg,\mathrm{ad})$, their $L_\infty$-algebra coincides with our $L_\infty$-algebra $\frakC_{\rmRBLA}(\frakg)$ of Theorem~\ref{Thm: rb-L-infty} at $\lambda=0$, and their Maurer-Cartan characterization of relative Rota-Baxter Lie algebra structures corresponds to Proposition~\ref{prop: RBLA 1-1 MC}.

\textbf{(ii) The cochain complexes and cohomology.} In \cite{TBGS19, LST}, the cochain complex of a (relative) Rota-Baxter operator is the Chevalley-Eilenberg complex of the descendent Lie algebra $(V,[-,-]_T)$, where $[u,v]_T=\rho(T(u))(v)-\rho(T(v))(u)$, with coefficients in a suitable module, and the cohomology of a relative Rota-Baxter Lie algebra is defined as the cohomology of the mapping cone of a canonical chain map from the Chevalley-Eilenberg complex of $\frakg$ to the complex of the operator. For $(V,\rho)=(\frakg,\mathrm{ad})$, the descendent bracket $[-,-]_T$ is exactly our bracket $[-,-]_\star$ with $\lambda=0$, their cochain complex of the operator reduces to our complex $\rmC_{\RBLO}^{\bullet}(\frakg)$ of the Rota-Baxter operator, and their chain map coincides with our map $\Phi^{\bullet}$ at $\lambda=0$ (Proposition~\ref{prop: Phi is chain map}). Consequently, their cohomology of relative Rota-Baxter Lie algebras reduces to our cohomology $\rmH_{\rmRBLA}^{\bullet}(\frakg)$, which arises from twisting the $L_\infty$-algebra $\frakC_{\rmRBLA}(\frakg)$ (Proposition~\ref{prop: cochain complex of RBLA}).

\textbf{(iii) The homotopy versions.} In \cite{LST2}, homotopy relative Rota-Baxter Lie algebras were introduced and characterized as Maurer-Cartan elements $T=\sum_{k\geqslant 1}T_k\in \Hom(\overline{S}(V),\frakg)$ of a weakly filtered $L_\infty$-algebra on $\prod_{i\geqslant 1}\Hom(S^i(V),\frakg)$, again constructed by higher derived brackets. Specializing $(V,\rho)=(\frakg,\mathrm{ad})$, their defining identities for the family $\{T_k\}_{k\geqslant 1}$ coincide with Equation~\eqref{Eq: homotopy T_n} at $\lambda=0$. Hence a homotopy relative Rota-Baxter Lie algebra over the adjoint representation in the sense of \cite{LST2} is exactly a homotopy Rota-Baxter Lie algebra of weight zero in the sense of Proposition~\ref{prop: RBLA 1-1 MC}.

\textbf{(iv) Nonzero weights.} For nonzero weights, Jiang and Sheng \cite{JiangSheng21} introduced representations of relative Rota-Baxter Lie algebras of arbitrary weight $\lambda$ and defined their cohomology with coefficients in a representation, by the same descendent Lie algebra and mapping cone pattern as in the weight zero case; they also deduced a cohomology theory for absolute Rota-Baxter Lie algebras of arbitrary weight. When the representation is the adjoint one, their construction specializes precisely to our complexes $\rmC_{\RBLO}^{\bullet}(\frakg, M)$ and $\rmC_{\rmRBLA}^{\bullet}(\frakg, M)$ in Section~\ref{sect: Linf of RBLA}: the underlying cochain complexes and hence the cohomology groups coincide. What we add to their cohomological theory is the $L_\infty$-structure behind it: our complex $\rmC_{\rmRBLA}^{\bullet}(\frakg)$ is obtained by twisting the $L_\infty$-algebra $\frakC_{\rmRBLA}(\frakg)$ extracted from the minimal model $\RBLA_\infty$ (Proposition~\ref{prop: cochain complex of RBLA}), the chain map $\Phi^{\bullet}$ and the dg Lie algebra governing deformations of the operator alone arise from the same twisting procedure (Proposition~\ref{prop: dgLa RBLO and twisting}), and the same $L_\infty$-algebra encodes the homotopy version in item (iii), which has no counterpart in \cite{JS21}. On the associative side, Das \cite{Das21} constructed a dg Lie algebra whose Maurer-Cartan elements are $\lambda$-weighted relative Rota-Baxter operators and developed the cohomology and deformation theory of weighted Rota-Baxter operators, and Das and Mishra \cite{DasM} determined the $L_\infty$-algebra governing $L_\infty$-deformations of associative Rota-Baxter algebras and introduced homotopy Rota-Baxter operators. Through the commutator functor, these associative results correspond to ours; the precise relation on the homotopy level is given by the comparison morphism $\Phi: \RBLA_\infty\to \RB_\infty$ in Section~\ref{sect: comparison}. For relative Rota-Baxter operators of weight $1$ on Lie groups and Lie algebras, a cohomology theory based on the Chevalley-Eilenberg complex of the descendent Lie algebra was established and related to integration and factorization problems by Jiang, Sheng and Zhu \cite{JSZ24}; at the infinitesimal level, their complex is again the weight one case of the operator complex above.

In summary, the results of \cite{TBGS19, LST, LST2} specialized to the adjoint representation agree with the weight zero case of our results, and the cohomology theory of \cite{JS21} for arbitrary weights coincides with ours at the level of cochain complexes, while our theory works for Rota-Baxter Lie algebras of arbitrary weight $\lambda$, is constructed intrinsically from the minimal model $\RBLA_\infty$ of the operad $\RBLA$, and provides in addition the homotopy version, namely homotopy Rota-Baxter Lie algebras of arbitrary weight.

\bigskip

\section{From Rota-Baxter   associative algebras to Rota-Baxter Lie algebras}\label{sect: comparison}
		
		For a given Rota-Baxter associative algebra $(A,\cdot,T)$ of weight $\lambda$, let $[-,-]$ be the Lie bracket given by the commutator of $A$. It is easy to check the triple $(A,[-,-],T)$ is  a Rota-Baxter Lie algebra  of weight $\lambda$. Let $\RB$ be the regular symmetric operad of $\nsRB$, the (non-symmetric) operad of the Rota-Baxter associative algebra of weight $\lambda$. The construction above gives a mophism of symmetric operads from $\RBLA$  to $\RB$. In this section, we will give an anologue of the mophism between the minimal models of these two operads.
		
		Let  $\RB_\infty$ be the minimal model of $\RB$, i.e., $\RB_\infty$ is the regular symmetric operad of the minimal model $\nsRB_{\infty}$ of $\nsRB$ given in \cite{WZ24}. 
Specifically, the generating $\bbS$-module $M_{\bbS}$ of the dg quasi-free operad $\RB_\infty$ is defined as $ M_{\bbS}(1) = \bfk R_1 \ot \bfk[\s_1]  ,$ $M_{\bbS}(n) = (\bfk R_n  \oplus \bfk \mu_{n}) \ot \bfk[\s_n],$ for $n\geqslant 2$, endowed with the regular representation as its $\bbS$-module structure. If $\sigma $ is the identity element in $\s_n$, we abbreviate $\mu_n \ot \sigma$ and  $R_n\ot \sigma$ as $\mu_n$ and $R_n$, respectively.
The action of differential operator $\partial_{\bbS}$ on generators can be expressed as follows:
		\begin{eqnarray*}
			\partial(\mu_n) \ot \sigma =  (\sum_{j=2}^{n-1}\sum_{i=1}^{n-j+1}(-1)^{i+j(n-i)}\mu_{n-j+1}  \circ_i \mu_j ) \cdot \sigma
		\end{eqnarray*} and
		\begin{eqnarray*}
			   &\partial (R_n \ot \sigma ) &= \sum\limits_{k=2}^n\sum_{l_1+\cdots+l_k=n\atop l_1, \dots, l_k\geqslant 1}(-1)^{\alpha'}\Big(\cdots\big((\mu_k\circ_1 R_{l_1})\circ_{l_1+1}R_{l_2}\big)\cdots\Big)\circ_{l_1+\cdots+l_{k-1}+1}R_{l_k}\Big) \cdot \sigma \\
			    & &\hspace{-1cm} + \sum\limits_{{\small\substack{2\leqslant p\leqslant n \\ 1\leqslant q\leqslant p
			}}}\sum\limits_{\small\substack{ r_1+\dots+r_q+p-q=n\\r_1, \dots, r_q\geqslant 1\\1\leqslant i\leqslant r_1\\1\leqslant k_1<\dots< k_{q-1}\leqslant p }}(-1)^{\beta'}\lambda^{p-q}\Big(R_{r_1}\circ_i \Big(\big((\cdots(( \mu_p\circ_{k_1}R_{r_2})\circ_{k_2+r_2-1} R_{r_3}))\cdots \big)\circ_{k_{q-1}+r_2+\dots+r_{q-1}-q+2}R_{r_q}\Big)\Big)\cdot \sigma,
		\end{eqnarray*}
		where the sign $(-1)^{\alpha'}$ and $(-1)^{\beta'}$ are defined in equations~\eqref{Eq: sign   alpha'} and \eqref{Eq: sign   beta'}, respectively. 

\begin{prop}\label{prop: comparison morphism}
		There is a morphism of dg operads $\Phi: \RBLA_\infty\rightarrow\RB_\infty$.
\end{prop}		
		
\begin{proof}
Define the morphism by, on the generators: 
	\begin{align*}
	\Phi: \RBLA_\infty&\rightarrow\RB_\infty\\
	\ell_n&\mapsto \sum_{\sigma\in \s_n}\sgn(\sigma)\mu_n \ot \sigma\\
	T_n&\mapsto \sum_{\sigma\in \s_n}\sgn(\sigma)R_n \ot \sigma.
		\end{align*}
	It suffices to check the map commutes with the differentials on the generators.  On one hand,
	\begin{align*}
		\partial\Phi(\ell_n)=&\sum_{\sigma\in \s_n}\sum_{j=2}^{n-1}\sum_{i=1}^{n-j+1}(-1)^{i+j(n-i)}\sgn(\sigma)(\mu_{n-j+1}\circ_i \mu_j )\cdot \sigma.
	\end{align*}
On the other hand,
	\begin{align*}
	\Phi\partial(\ell_n)=&\Phi\Big(\sum_{\tau}\sum_{j=2}^{n-1}\sum_{i=1}^{n-j+1}(-1)^{i+j(n-i)}\sgn(\tau)(\ell_{n-j+1}\circ_i \ell_j )\cdot \tau\Big)\\
	=&\sum_{\tau} \sum_{j=2}^{n-1} \sum_{i=1}^{n-j+1}\sum_{\small\substack{ \sigma_1\in \s_{n-j+1}\\ \sigma_2\in \s_{j}}}(-1)^{i+j(n-i)}\sgn(\tau)\sgn(\sigma_1)\sgn(\sigma_2)\big((\mu_{n-j+1} \ot \sigma_1)\circ_i (\mu_j \ot \sigma_2)\big)\cdot\tau,
\end{align*}
where all  $\sum_\tau$ denote the summations over permutations  $\tau \in \s_n$ that make the corresponding tree monomials into shuffle tree monomials. 
Since element of the form $(\mu_{n-j+1}\circ_i \mu_j )\cdot \sigma$ can be uniquely written as shuffle tree monomial $\big((\mu_{n-j+1}\ot \sigma_1)\circ_{i'} (\mu_j \ot \sigma_2)\big)\cdot\tau$  by equivariance of the $\bbS$-module structure, we only need to show the signs of the corresponding terms agree. For any fixed $\sigma$ and $i$, let $\sigma_1, \sigma_2, \tau$ and $k$ be such that
\begin{align*}
	(\mu_{n-j+1}\circ_i \mu_j )\cdot\sigma=&\big((\mu_{n-j+1} \ot \sigma_1)\circ_k (\mu_j \ot \sigma_2)\big)\cdot \tau\\
	=&(\mu_{n-j+1}\circ_{\sigma_1(k)} \mu_j)\cdot\sigma_1''\sigma_2'\tau,
\end{align*}
then $\sgn(\sigma_1'')=(-1)^{(j-1)(\sigma_1(k)-k)} \sgn(\sigma_1)$ and $\sgn(\sigma_2') = \sgn(\sigma_2)$ by computation. 
Then we have $\sigma_1(k)=i, \sigma = \sigma_1''\sigma_2'\tau$ and 
\begin{align*}
	\sgn(\sigma)=&\sgn(\sigma_1'')\sgn(\sigma_2')\sgn(\tau)\\
	=&(-1)^{(j-1)(i-k)}\sgn(\sigma_1)\sgn(\sigma_2)\sgn(\tau).
\end{align*}
This ensures that $\partial\Phi(m_n)=\Phi\partial(m_n)$.
It is similar for $\partial\Phi(T_n)=\Phi\partial(T_n)$.
\end{proof}		
	
Consequently, for any homotopy Rota-Baxter associative algebra $A$ of weight $\lambda$, the above morphism induces a structure of  homotopy Rota-Baxter Lie algebra of weight $\lambda$ on $A$.
		
	\bigskip

 \textbf{Acknowledgements:}   
This work was supported by  the National Key R  $\&$ D Program of China (No. 2024YFA1013803),  and by Shanghai Key Laboratory of PMMP (No.
22DZ2229014).

%

\end{document}